\documentclass{amsart}
\usepackage[T1]{fontenc}
\usepackage[utf8]{inputenc}
\usepackage{amsmath}
\usepackage{array}
\usepackage{amsthm}
\usepackage{amssymb}
\input xy
\xyoption{all}

\usepackage[cmtip,all]{xy}

\usepackage{stmaryrd}

\usepackage{tikz}
\usetikzlibrary{matrix,arrows,decorations.pathmorphing}
\usepackage{tikz-cd}

\usepackage{enumerate}

\usepackage[french, english]{babel}

\usepackage{xcolor}

\newtheorem{thmintro}{Theorem}
\newtheorem{corintro}[thmintro]{Corollary}
\newtheorem{app}[thmintro]{Application}

\newtheorem{thm}{Theorem}[section]
\newtheorem{pr}[thm]{Proposition}
\newtheorem{cor}[thm]{Corollary}
\newtheorem{lm}[thm]{Lemma}

\theoremstyle{definition}
\newtheorem{defi}[thm]{Definition}
\newtheorem{defi-prop}[thm]{Proposition-Definition}
\newtheorem{nota}[thm]{Notation}

\newtheorem{ex}[thm]{Example}
\newtheorem{exs}[thm]{Examples}

\newtheorem{conclusion}[thm]{Conclusion}
\newtheorem{exintro}[thmintro]{Example}

\theoremstyle{remark}
\newtheorem{rk}[thm]{Remark}
\newtheorem{warning}[thm]{Warning}

\newcommand{\A}{{\mathcal{A}}}
\newcommand{\B}{{\mathcal{B}}}
\newcommand{\C}{{\mathcal{C}}}
\newcommand{\D}{{\mathcal{D}}}
\newcommand{\F}{{\mathcal{F}}}
\newcommand{\Pp}{{\mathcal{P}}}
\newcommand{\add}{\mathbf{Add}}

\newcommand{\I}{{\mathcal{I}}}
\newcommand{\K}{{\mathcal{K}}}
\newcommand{\LL}{{\mathcal{L}}}
\newcommand{\M}{{\mathcal{M}}}

\newcommand{\G}{{\mathcal{G}}}

\newcommand{\red}{\mathrm{red}}

\newcommand{\Md}{\text{-}\mathbf{Mod}}
\newcommand{\md}{\text{-}\mathbf{mod}}
\newcommand{\Mdd}{\mathbf{Mod}\text{-}}
\newcommand{\Si}{\mathfrak{S}}
\newcommand{\id}{\mathrm{id}}
\newcommand{\FF}{\mathbb{F}}
\newcommand{\simpl}{\mathrm{ss}}
\newcommand{\EML}{\mathrm{EML}}
\newcommand{\res}{{\mathrm{res}}}
\newcommand{\gen}{{\mathrm{gen}}}

\newcommand{\op}{{\mathrm{op}}}
\newcommand{\Fp}{{\mathbb{F}_p}}
\newcommand{\Fq}{{\mathbb{F}_q}}
\newcommand{\KK}{{\mathbb{K}}}

\newcommand{\Lan}{\mathrm{Lan}}
\newcommand{\colim}{\mathop{\mathrm{colim}}}
\newcommand{\Proj}{\mathbf{P}}
\newcommand{\cross}{\mathrm{cr}}
\newcommand{\Tot}{\mathrm{Tot}\,}
\newcommand{\stdeg}{\deg}   
\newcommand{\GL}{\operatorname{GL}}
\newcommand{\Sp}{\operatorname{Sp}}
\newcommand{\orth}{\operatorname{O}}
\newcommand{\kA}{{_{\,k\mkern -2mu}\A}}
\newcommand{\FA}{{_{\,\mathbb{F}\mkern -2mu}\A}}
\newcommand{\KA}{{_{\,\mathbb{K}\mkern -2mu}\A}}
\newcommand{\kB}{{_{\,k\mkern -2mu}\B}}
\newcommand{\kATor}{{_{\,k\mkern -3mu}\A}}
\newcommand{\HH}{\underline{\mathrm{H}}}
\newcommand{\rmH}{\mathrm{H}}
\newcommand{\EExt}{\underline{\mathrm{Ext}}}
\newcommand{\Ext}{\mathrm{Ext}}
\newcommand{\Tor}{\mathrm{Tor}}
\newcommand{\Hom}{\mathrm{Hom}}
\newcommand{\End}{\mathrm{End}}
\newcommand{\poly}{\mathrm{pol}}
\newcommand{\anti}{\mathrm{anti}}
\newcommand{\diag}{\mathrm{diag}}
\newcommand{\summ}{\mathrm{sum}}

\begin{document}

\title{Functor homology over an additive category}

\author[A. Djament]{Aur\'elien Djament}
\address{CNRS, Univ. Lille, UMR 8524 - Laboratoire Paul Painlev\'e, F-59000 Lille, France}
\email{aurelien.djament@univ-lille.fr}
\urladdr{https://math.univ-lille1.fr/~djament/}
\thanks{This author is partly supported by the projects ChroK (ANR-16-CE40-0003), AlMaRe (ANR-19-CE40-0001-01) and Labex CEMPI (ANR-11-LABX-0007-01)}

\author[A. Touz\'e]{Antoine Touz\'e}
\address{Univ. Lille, CNRS, UMR 8524 - Laboratoire Paul Painlev\'e, F-59000 Lille, France}
\email{antoine.touze@univ-lille.fr}
\urladdr{https://pro.univ-lille.fr/antoine-touze/}
\thanks{This author is partly supported by the project ChroK (ANR-16-CE40-0003) and Labex CEMPI (ANR-11-LABX-0007-01)}



\begin{abstract}
We uncover several general phenomenas governing functor homology over additive categories. In particular, we generalize the strong comparison theorem of Franjou Friedlander Scorichenko and Suslin to the setting of $\mathbb{F}_p$-linear additive categories. Our results have a strong impact in terms of explicit computations of functor homology, and they open the way to new applications to stable homology of groups or to $K$-theory. As an illustration, we prove comparison theorems between cohomologies of classical algebraic groups over infinite perfect fields, in the spirit of a celebrated result of Cline, Parshall, Scott et van der Kallen for finite fields.
\end{abstract}

\maketitle

\selectlanguage{french}
\renewcommand{\abstractname}{R\'esum\'e}
\begin{abstract}
Nous r\'ev\'elons plusieurs ph\'enom\`enes g\'en\'eraux contr\^olant l'homologie des foncteurs sur les cat\'egories additives. En particulier, nous g\'en\'eralisons le th\'eor\`eme de comparaison forte de Franjou, Friedlander, Scorichenko et Suslin au contexte des cat\'egories additives $\Fp$-lin\'eaires. Nos r\'esultats ont un fort impact en termes de calculs explicites d'homologie des foncteurs, et ils ouvrent la voie \`a de nouvelles applications en homologie stable des groupes et en $K$-th\'eorie. Nous illustrons cela en d\'emontrant des th\'eor\`emes de comparaison entre cohomologies de groupes alg\'ebriques classiques sur des corps parfaits infinis, dans l'esprit d'un c\'el\`ebre th\'eor\`eme de Cline, Parshall, Scott et van der Kallen pour les corps finis.
\end{abstract}

\selectlanguage{english}

\section{Introduction}

Given a small category $\C$ and a commutative ring $k$, we denote by $k[\C]\Md$ (resp. $\Mdd k[\C]$) the category of covariant (resp. contravariant) functors from $\C$ to $k$-modules, and natural transformations between such functors. As the notation suggests, $k[\C]\Md$ 
behaves very much like a category of modules over a ring $R$. In particular, one can compute $\Ext$ in these functor categories, and there is also a tensor product
\[-\otimes_{k[\C]}-:\Mdd k[\C]\times k[\C]\Md\to k\Md\]
which can be derived to define $\Tor$-modules. Such $\Ext$ and $\Tor$-modules were first considered in Mitchell's influential article \cite{Mi72} in the early seventies, and  we refer to them as `functor homology'. 
We are chiefly interested in functor homology when $\C=\Proj_R$, the category of finitely generated projective modules over a ring $R$. Our interest stems from the close relations with stable $K$-theory \cite{Sco,FranjouPira}, with topological Hochschild homology \cite{PiraWald} and with the stable homology of groups with twisted coefficients \cite{FFSS,DjaR,Ku-adv}. These close relations make explicit calculations of functor homology over $\Proj_R$ highly desirable. 

During the nineties, a series of breakthrough articles \cite{FLS,FS,FFSS} uncovered general phenomena governing functor homology when $R$ is a finite field $\Fq$.
A key result is the \emph{strong comparison theorem} \cite[Thm 3.10]{FFSS}, which bridges the $\Ext$-computations in $k[\Proj_\Fq]\Md$ with $\Ext$-computations in the much nicer category $\Pp_\Fq$ of strict polynomial functors in the sense of Friedlander and Suslin \cite{FS}. The latter category is much more amenable to computations.
Together with the results of \cite{Chalupnik,TouzeUnivNew,Touze-Survey} on strict polynomial functors, the strong comparison theorem provides an effective way of performing  many  homological computations in $k[\Proj_\Fq]\Md$.
In sharp contrast to the situation of finite fields,
very little was known for more general rings $R$ up to today. An analogue of the strong comparison theorem was lacking, and only a few isolated computations were known \cite{FPmaclane,PiraZnZ}, for $R=\mathbb{Z}$ or $\mathbb{Z}/n\mathbb{Z}$. Suslin wrote in \cite[Appendix p.717]{FFSS} that the $\Ext$-modules in $k[\Proj_R]\Md$ `do not seem to be computable unless we are dealing with finite fields'. 

The purpose of the present article is to remedy the lack of understanding of functor homology over general rings. We provide a series of new results which allow to compute functor homology (much) beyond the case of finite fields -- in fact, many of these results hold when $\Proj_R$ is replaced by a suitable additive category $\A$. If $\A=\Proj_R$ and $k$ is an infinite perfect field, our results show that functor homology is essentially controlled by three simpler quantities, namely:
\begin{enumerate}[(A)]
\item\label{intro-quant-1} functor homology over $\Proj_Q$ where $Q$ is a \emph{finite} quotient ring of $R$ with \emph{cardinal invertible} in $k$,
\item\label{intro-quant-2} Ext and Tor between $R\otimes_{\mathbb{Z}}k$-modules, 
\item\label{intro-quant-3} generic homology of strict polynomial functors, that is, certain Ext and Tor in the category $\Pp_k$ of strict polynomial functors, or equivalently \cite[Thm 3.2]{FS} between modules over classical Schur algebras.
\end{enumerate}
The first quantity is more exotic than the last two ones, but it does not have to be considered if $R$ has no finite quotient of cardinal invertible in $k$ (which is the case for many rings of interest). The second quantity is classical and related to the Hochschild homology of $R$, and the last quantity is also classical and relatively well understood \cite{Chalupnik,TouzeUnivNew}, see \cite{Touze-Survey} for a survey. 

In order to obtain our results over general rings, we perfect some techniques introduced in \cite{FFSS} (see section \ref{sec-Pol-Fq}), and we also rely on more classical techniques such as Hurewicz theorems (see section \ref{sec-simplicial}) or exact Kan extensions (see section \ref{subsec-completion}), which we use in a way that is new to the subject of functor homology.

An important aspect of our results is their computational applicability. We provide closed formulas for functor homology, which provide a wealth of new explicit computations. As an example, we give in section \ref{subsec-sample} a quick generalization over infinite perfect fields of the functor homology computations of \cite{FFSS}. 

The present work gives a control on functor homology which allows to attack some problems related to $K$-theory or group homology. To demonstrate this, we finish the article by establishing a comparison theorem between the stable cohomology of classical groups over an infinite perfect field $k$ and the stable cohomology of their underlying groups of $k$-points, in the spirit of the celebrated theorem of Cline Parshall Scott and van der Kallen \cite{CPSVdK} for finite fields. Further applications will be developped in future work.

%
%
\subsection{The main results}

We now give an overview of the main results of the paper. In order to have a feeling of functor homology, the reader should keep in mind the analogy with group homology. Namely, the category $k[\Proj_R]\Md$ is similar to the category $k[\GL_n(R)]\Md$ of $k$-linear representations over a general linear group of large rank $n$, and therefore functor homology over $\Proj_R$ is similar to group homology of $\GL_n(R)$ -- this is actually more than a mere similarity since functor homology \emph{is} group homology in good cases \cite{FFSS,DjaR,Ku-adv}. This analogy tells us that changing the ring $R$ has a drastic effect on functor homology, which is hard to control. In sharp contrast, the effect of changing $k$ is well understood and controled by a standard universal coefficient theorem. 

Although several of our results are valid for arbitrary commutative rings $k$, to unify and simplify the exposition of the main results, we make the assumption that $k$ is an infinite perfect field in this introduction. Unadorned tensor products are taken over $k$. We fix a small additive category $\A$ (that is a small $\mathbb{Z}$-linear category $\A$ having biproducts),  and we let $k[\A]\Md$ stand for the category of (non necessarily additive) functors $\A\to k\Md$, and natural transformations between them.
\medskip

It is proved in \cite[Cor 4.11]{DTV} that functors $F:\A\to k\Md$ satisfying reasonnable finiteness hypotheses have a very rigid structure, namely they can be written as 
$F(a)=B(a,a)$, for a unique bifunctor $B:\A\times\A\to k\Md$ of \emph{AP-type}. By this, we mean a bifunctor $B$ which is \emph{Antipolynomial} with respect to its first variable and \emph{Polynomial} in the sense of Eilenberg and Mac Lane \cite{EML} with respect to its second variable. Typical polynomial functors $\A\to k\Md$ are tensor products of additive functors, while the antipolynomial functors are those which factor through a \emph{$k$-trivial category} $\B$, i.e. a small additive category with finite Hom-sets of cardinal invertible in $k$.
Our first main theorem allows to exploit this rigid structure in homological computations. 
Let $\Delta:\A\to \A\times\A$ be the diagonal functor, i.e. $\Delta(a)=(a,a)$, and let $\Delta^*B$ denote the composition $B\circ\Delta$. Thus every reasonable functor $F:\A\to k\Md$ is of the form $\Delta^*B$. 
\begin{thmintro}\label{thm-AP-type}
Let $B$, $B'$ and $C$ be three bifunctors of AP-type, with $B$ contravariant in both variables. Restriction along the diagonal $\Delta$ yields isomorphisms:
\begin{align*}
&\Ext^*_{k[\A]}(\Delta^*B',\Delta^*C)\simeq \Ext^*_{k[\A\times\A]}(B',C)\;,\\
&\Tor_*^{k[\A]}(\Delta^*B,\Delta^*C)\simeq \Tor_*^{k[\A\times \A]}(B,C)\;.
\end{align*}
\end{thmintro}

Functors of the form $F(a)= F^{\anti}(a)\otimes F^{\poly}(a)$, where the first tensorand is an antipolynomial functor and the second tensorand is a polynomial functor, appear frequently. For instance, simple functors with finite-dimensional values are of this form, at least if $k$ is big enough \cite[Cor 4.13]{DTV}. For such a tensor product, we have $F=\Delta^*B$ with $B(a,b)=F^\anti(a)\otimes F^\poly(b)$, and the classical K\"unneth formula yields the following consequence of theorem \ref{thm-AP-type}.
\begin{corintro}\label{corintro-separation}Assume that $F=F^\anti\otimes F^\poly$ and $G=G^\anti\otimes G^\poly$. There is an isomorphism of graded vector spaces:
\[\Tor_*^{k[\A]}(F,G)\simeq \Tor_*^{k[\A]}(F^\anti,G^\anti)\otimes \Tor_*^{k[\A]}(F^\poly,G^\poly)\;.\]
\end{corintro}
There is a similar formula for $\Ext$ under some finiteness hypotheses (see proposition \ref{prop-Kunneth-ext} and remark \ref{rk-pfinfty}). Thus, computing functor homology essentially reduces to computing functor homology in the special situation where the arguments of $\Ext$ or $\Tor$ are simultaneously antipolynomial or polynomial.

We first examine \emph{antipolynomial homology}, that is, the graded vector spaces  
\begin{equation}\Tor_*^{k[\A]}(F^\anti,G^\anti)\quad\text{ and }\quad\Ext^*_{k[\A]}(F'{}^\anti,G^\anti)\;.
\label{eq-intro-anti}
\end{equation}
For every pair of antipolynomial functors $(F^\anti,G^\anti)$, we can always find such a full, additive, and essentially surjective $\pi:\A\to \B$ with $k$-trivial codomain such that $F^\anti$ and $G^\anti$ both factor through $\pi$. Our second main theorem allows to transport the computation of antipolynomial homology in the category $k[\B]\Md$. We actually prove it as special case of a more general theorem, which can be seen as an analogue in functor homology of the excision theorem of Suslin and Wodzicki in $K$-theory \cite{Suslin-Wodz}, see theorem \ref{thm-magique-general} and remark \ref{rk-excis}.

\begin{thmintro}\label{thm-intro-2}
Let $\pi:\A\to \B$ be a full, additive, and essentially surjective functor, with $k$-trivial codomain $\B$. Let $F$, $F'$ and $G$ be three functors from $\B$ to $k\Md$, with $F$ contravariant. Let $\pi^*F$ denote the composition $F\circ \pi$. Restriction along $\pi$ yields isomorphisms:
\begin{align*}
&\Ext^*_{k[\A]}(\pi^*F',\pi^*G)\simeq \Ext^*_{k[\B]}(F',G)\;,\\
&\Tor_*^{k[\A]}(\pi^*F,\pi^*G)\simeq \Tor_*^{k[\B]}(F,G)\;.
\end{align*}
\end{thmintro}

The practical interest of theorem \ref{thm-intro-2} lies in the fact that the category $k[\B]\Md$ has a much nicer structure than $k[\A]\Md$, see e.g. \cite[Prop 11.7]{DTV}, and it is potentially more accessible via combinatorial methods since its has finite Hom-sets.

\begin{exintro}
In the situation of theorem \ref{thm-intro-2}, if $\A$ is the category $\Proj_R$ of finitely generated projective modules over a ring $R$ then $\B$ is necessarily the category $\Proj_Q$ for some finite quotient ring $Q$ of $R$ with cardinal invertible in $k$, see \cite[Ex 4.4]{DTV}. Hence theorem \ref{thm-intro-2} reduces the computation of $\Ext$ and $\Tor$ over $k[\Proj_R]$ to that of $\Ext$ and $\Tor$ over $k[\Proj_Q]$ -- this is quantity \eqref{intro-quant-1} from the begining of the introduction.
\end{exintro}

We provide the following example of computational application of theorem \ref{thm-intro-2} in corollary \ref{cor-vanish-Kuhn}. 

\begin{app}
In the situation of theorem \ref{thm-intro-2}, if $\B=\Proj_R$ for a finite semi-simple ring $R$ and $k$ has characteristic zero, the graded vector spaces $\Ext^*_{k[\A]}(\pi^*F',\pi^*G)$ and $\Tor_*^{k[\A]}(\pi^*F,\pi^*G)$ vanish in positive degrees.
\end{app}

Next, we examine \emph{polynomial homology}, that is, the graded vector spaces 
\begin{equation}
\Tor_*^{k[\A]}(F^\poly,G^\poly)\quad\text{ and }\quad\Ext^*_{k[\A]}({F'}^\poly,G^\poly)\;.\label{eq-intro-poly}
\end{equation}
Many polynomial functors of interest -- e.g. simple functors with finite-dimensional values, at least if the field is big enough \cite[Thm 5.5]{DTV} -- are given by tensor products $F^\poly=\pi_1^*F_1\otimes\cdots\otimes\pi_m^*F_m$ where each factor $\pi_i^*F_i=F_i\circ\pi_i$ is the composition of an additive functor $\pi_i$ from $\A$ to finite dimensional vector spaces with a strict polynomial functor $F_i$ from finite dimensional vector spaces to $k\Md$, such as a symmetric power $S^d$, an exterior power $\Lambda^d$, a Schur functor or a Weyl functor associated to a partition \cite{ABW}. We restrict ourselves to computing polynomial functor homology when $F^\poly$, $F'{}^\poly$ and $G^\poly$ are such tensor products. We distinguish three cases of increasing difficulty, according to the characteristic of the field $k$.

\subsubsection*{Case 1: characteristic zero} If $k$ has characteristic zero, standard techniques recalled in section \ref{sec-prelim-polyn-hom} reduce the computation of the functor homology quantities \eqref{eq-intro-poly} to the much simpler computation of graded vector spaces 
\begin{equation}
\Tor_*^{k[\A]}(\pi,\rho)\quad\text{ and }\quad\Ext^*_{k[\A]}(\pi',\rho)\;,\label{eq-intro-add}
\end{equation}
in which $\pi$, $\pi'$ and $\rho$ are \emph{additive}. Let $\kA\Md\subset k[\A]\Md$ denote the full abelian subcategory on the \emph{additive} functors. It is known \cite[Thm 1.2]{Dja-Ext-Pol} that the embedding $\kA\Md\hookrightarrow k[\A]\Md$ providentially preserves $\Ext$ and $\Tor$ in characteristic zero. This is particularly interesting
because computations in  $\kA\Md$ are (regardless the characteristic of $k$) easier than computations in $k[\A]\Md$, and closer to classical computations, as the following example shows it. 
\begin{exintro}\label{fact-intro}
If $\A=\Proj_R$ and $k$ is a commutative ring, every additive functor from $\A$ to $k$-vector spaces is given by tensoring with an $(R,k)$-bimodule, and this yields an equivalence of categories $\kA\Md\simeq R\otimes_\mathbb{Z}k\Md$. Therefore computing $\Tor$ and $\Ext$ in $\kA\Md$ amounts to computing $\Tor$ and $\Ext$ between $R\otimes_\mathbb{Z}k$-modules -- this is quantity \eqref{intro-quant-2} in the beginning of the introduction. 
\end{exintro}

\subsubsection*{Case 2: characteristic $p\gg 0$} The standard techniques of characteristic zero work equally well in positive characteristic $p$, provided $p$ is big enough with respect to the arguments of $\Tor$ and $\Ext$ (a precise bound is given in definition \ref{defi-char-large}). Thus we reduce ourselves to computing graded vector spaces of the form \eqref{eq-intro-add}. 

However, the situation becomes quite different from characteristic zero when we try to compute functor homology between additive functors. Indeed, it is known \cite{Dja-Ext-Pol} that the embedding $\kA\Md\hookrightarrow k[\A]\Md$ does not preserve $\Ext$ and $\Tor$ in positive characteristic. Our third main result bypasses this problem.
 Although technically very different, this result is similar in spirit to the main theorems of \cite{LL,LLbis} (see also \cite[Cor 4.2]{PiraSpectral}) which compare Hochschild homology and topological Hochschild homology over smooth $\Fp$-algebras. 
\begin{thmintro}\label{thm-intro-3}
Let $k$ be an infinite perfect field of positive characteristic $p$, and let $\pi,\pi',\rho$ be three additive functors from $\A$ to $k\Md$, with $\pi$ contravariant. If $\A$ is $\Fp$-linear, there are graded isomorphisms, natural with respect to $\pi$, $\pi'$ and $\rho$:
\begin{align*}
\Ext^*_{\kATor}(\pi',\rho)\otimes E_\infty^*\simeq \Ext^*_{k[\A]}(\pi',\rho)\;,\\
\Tor_*^{\kATor}(\pi,\rho)\otimes T^\infty_*\simeq \Tor_*^{k[\A]}(\pi,\rho)\;,
\end{align*}
where $E_\infty^*$ and $T^\infty_*$ denote the graded vector spaces equal to $k$ if $*$ is even and non-negative, and to zero in the other degrees.
\end{thmintro}

\subsubsection*{Case 3: small positive characteristic $p$}
If the characteristic of $k$ is positive but not big enough, the situation is more complicated and we cannot reduce the computation of the functor homology quantities \eqref{eq-intro-poly} to functor homology between additive functors. Instead, we can reduce ourselves to computing terms of the form 
\begin{equation}
\Tor_*^{k[\A]}(\pi^*F,\rho^*G)\quad\text{ and }\quad\Ext^*_{k[\A]}(\pi'{}^*F',\rho^*G)\;,\label{eq-intro-polyadd}
\end{equation} 
for additive functors $\pi$, $\pi'$, $\rho$ and  strict polynomial functors $F$, $F'$ and $G$. Computing such functor homology terms is much harder than computing the functor homology terms \eqref{eq-intro-add}.

The strong comparison theorem \cite[Thm 3.10]{FFSS} was so far the key result to compute the quantities \eqref{eq-intro-polyadd}, but it only works when $\A=\Proj_\Fq$ and if $\Fq$ is a big finite field of characteristic $p$. Our fourth main result is the \emph{generalized comparison theorem}, which generalizes it to all additive $\Fp$-linear categories $\A$. Namely, under a certain homological vanishing condition regarding the functors $\pi$ and $\rho$, our generalized comparison theorem gives an isomorphism in all degrees $i$:
\[\Tor_i^{k[\A]}(\pi^*F,\rho^*G)\simeq \Tor_i^{\gen}(F^\dag,G)\]
where $F^\dag$ is the precomposition of $F$ by an explicit additive functor, and $\Tor^\gen_*(F^\dag,G)$ denotes the \emph{generic homology of strict polynomial functors}, that is, the stabilization along iterated Frobenius twists of the $\Tor$ between $F^\dag$ and $G$, computed in the category of strict polynomial functors (see section \ref{sec-Pol-Fq} for more details). This generic homology already appears (in its $\Ext$ form) in \cite{FFSS} and it is the quantity \eqref{intro-quant-3} in the beginning of the introduction. It is the functor homology analogue of the generic cohomology of algebraic groups used e.g. in \cite{CPSVdK}.

We refer the reader to theorem \ref{thm-gen-comp} for a precise statement of the generalized comparison theorem, in particular for the formula defining $F^\dag$ and the precise homological vanishing condition involving $\rho$ and $\pi$. When the context is $\FF$-linear over a big subfield of $k$, the formula of $F^\dag$ and the homological vanishing condition both simplify and we obtain the following result in theorem \ref{thm-F-lin-case}. (We refer the reader to section \ref{subsec-nonhomogeneous} for the notion of degree of a strict polynomial functor, and we only mention here that typical functors of degree less or equal to $d$ are the symmetric powers $S^e$ and the exterior powers $\Lambda^e$ for $e\le d$, or direct sums of such functors such as $S^1\oplus S^d$.)  
\begin{thmintro}[The generalized comparison theorem - $\FF$-linear case]\label{thm-intro-F-lin-case}
Let $k$ be an infinite perfect field and let $\A$ be an additive $\FF$-linear category, where $\FF$ is a subfield of $k$. Let $\pi$ and $\rho$ be two $\FF$-linear functors from $\A$ to $k$-modules, respectively contravariant and covariant, and let $F$ and $G$ be two strict polynomial functors of degree less or equal to the cardinal of $\FF$.
Assume furthermore that 
\[\Tor_i^{\kATor}(\pi,\rho)=0\quad \text{for $0<i<e$.}\]
Then for $0\le i<e$ there are $k$-linear isomorphisms
\[\Tor^{k[\A]}_i(\pi^*F, \rho^*G)\simeq  \Tor_i^{\gen}(D_{\pi,\rho}^*F,G)\]
where $D_{\pi,\rho}$ refers to the contravariant functor $D_{\pi,\rho}(v)=\Hom_k(v,\pi\otimes_{k[\A]}\rho)$.
\end{thmintro}

The generalized comparison theorem and its $\FF$-linear version both have analogues for $\Ext$, which require some finiteness conditions on the additive functors in play, see corollaries \ref{cor-thm-F-lin} and \ref{cor-thm-gencomp}.
Note that when $\A=\Proj_R$, the vanishing condition on $\Tor$ can be reformulated in terms of $R\otimes_\mathbb{Z}k$-modules by example \ref{fact-intro}, thus in terms of the quantity \eqref{intro-quant-2} from the beginning of the introduction. The next two examples are special cases of theorem \ref{thm-intro-F-lin-case} when $\A=\Proj_R$ for a perfect field $R$, and show that the generalized comparison theorem is interesting even in this very simple case.

\begin{exintro}\label{ex-intro-1}
If $\FF=\Fq$ is a finite subfield of $k$ and $\A=\Proj_\Fq$ then the category $_k\A\Md$ is equivalent to the category of $\Fq\otimes_\mathbb{Z}k$-modules, hence semi-simple. Therefore the vanishing condition of theorem \ref{thm-intro-F-lin-case} is automatically satisfied for all $\pi$ and $\rho$. If we take
$\pi=\Hom_\Fq(-,k)$ and $\rho=-\otimes_\Fq k$, then theorem \ref{thm-intro-F-lin-case} yields a $\Tor$-version of the strong comparison theorem of \cite{FFSS}. And taking $\pi'=\rho=-\otimes_\Fq k$ in the $\Ext$-version of theorem  \ref{thm-intro-F-lin-case} gives back the strong comparison theorem of \cite{FFSS}.
\end{exintro}

Example \ref{ex-intro-1} shows that theorem \ref{thm-intro-F-lin-case} (hence our generalized comparison theorem \ref{thm-gen-comp}) subsumes the strong comparison theorem \cite[Thm 3.10]{FFSS}. However, our work does not provide a new proof of the strong comparison theorem. Our results rather rely on the strong comparison theorem of \cite{FFSS}, that we can generalize by using a non-trivial combination of other techniques, including simplicial resolutions and exact Kan extensions.

\begin{exintro}\label{exintro-perfectfield}
Assume that $\FF=k$, that $\A=\Proj_k$, and take $\pi:\Proj_k^\op\to \Proj_k$ the $k$-linear duality functor and $\rho:\Proj_k\to \Proj_k$ the identity functor. Then the vanishing condition of theorem \ref{thm-intro-F-lin-case} is satisfied in all degrees by the Hochschild-Kostant-Rosenberg theorem, and one obtains (replacing $\pi^*F$ by $F$) that for all contravariant strict polynomial functors $F$ and all strict polynomial functors $G$ there is a graded isomorphism:
\[\Tor_*^{k[\Proj_k]}(F,G)\simeq \Tor_*^{\gen}(F,G)\;.\]
It dualizes into a graded isomorphism, for all strict polynomial functors $F$ and $G$:
\[\Ext_*^{k[\Proj_k]}(F,G)\simeq \Ext_*^{\gen}(F,G)\;.\]
\end{exintro}

Example \ref{exintro-perfectfield} is treated into details in corollary \ref{cor-infinite-perfect-field}, in particular the isomorphisms are induced by simple explicit maps. This makes it easy to verify the compatibility of the isomorphisms with extra structure, for example the $\Ext$ isomorphism is compatible with tensor products and with the Yoneda composition of extensions. Functor homology over $k[\Proj_k]$ was out of reach before our work, while the generic homology of strict polynomial functors is well understood. 
As an application, we achieve the following computation in corollary \ref{cor-FFSS-infinite}, compare \cite[Thm 6.3]{FFSS}. (The algebra structure on these $\Ext$ is given by the convolution product, recalled in section \ref{sec-applic}).

\begin{app}
Let $k$ be an infinite perfect field of positive characteristic $p$, and let $r$ be a nonnegative integer. Let $V_{s,r}$ denote the trigraded vector space with homogeneous basis $(e_i)_{i\ge 0}$ where each $e_i$ is placed in tridegree $(2ip^r+sp^r-s,1,p^r)$.
Then we have isomorphisms of trigraded algebras:
\begin{align*}
&\Ext_{k[\Proj_k]}^*(\Gamma^{*(r)},S^*)\simeq S(V_{0,r})\;,
&&
\Ext_{k[\Proj_k]}^*(\Gamma^{*(r)},\Lambda^*)\simeq \Lambda(V_{1,r})
\;, \\
&\Ext_{k[\Proj_k]}^*(\Lambda^{*(r)},S^*)\simeq \Lambda(V_{0,r})\;,
&&
\Ext_{k[\Proj_k]}^*(\Lambda^{*(r)},\Lambda^*)\simeq \Gamma(V_{1,r})
\;, \\
&\Ext_{k[\Proj_k]}^*(S^{*(r)},S^*)\simeq \Gamma(V_{0,r})\;,
&&
\Ext_{k[\Proj_k]}^*(\Gamma^{*(r)},\Gamma^*)\simeq \Gamma(V_{2,r})
\;. 
\end{align*}
\end{app}

In the same fashion, we provide explicit bifunctor homology computations in the spirit of \cite{FF}, in corollary \ref{cor-calculbif}. For simplicity, all the computations that we perform are made for the additive category $\A=\Proj_k$, which is the simplest case for which computations were out of reach before the present article. But similar computations over more general additive categories $\A$ can be easily performed in the same way, provided the $\Tor$-vanishing condition of theorem \ref{thm-intro-F-lin-case} is satisfied.

\subsection{An application to the cohomology of classical groups}
%
Let $G$ be an algebraic group over a field $k$. We denote by $\EExt_G^*(V,W)$ its cohomology as in \cite{Jantzen}, that is, the extensions in the category of comodules of the cogebra of regular functions on $G$. We denote by $\Ext_{G}^*(V,W)$ the cohomology of its underlying discrete group of $k$-points as in \cite{Brown}, that is, the extensions in the category of $k$-linear representations of the discrete group $G$. The two cohomologies are related by a comparison map:
\[\EExt^*_{G}(V,W)\to \Ext_{G}^*(V,W)\;.\]
Assume that $k$ has positive characteristic $p$ and that $G$ is defined over the prime subfield of $k$. Restricting a representation $V$ along the Frobenius group morphism $\phi:G\to G$
yields the \emph{twisted representation} $V^{[r]}$. If $k$ is perfect, $\phi$ is an isomorphism in the category of discrete groups, so the comparison map becomes:
\[\EExt^*_{G}(V^{[r]},W^{[r]})\to \Ext^*_{G}(V^{[r]},W^{[r]})\simeq \Ext^*_{G}(V,W)\;.\qquad(*)\]
If $G$ is reductive (e.g. $G=\GL_n(k)$), the left hand-side does not depend on $r$ in low degrees if $r$ is big enough, and it is called the \emph{generic extensions} of $G$. 
A celebrated theorem of Cline, Parshall, Scott and van der Kallen \cite[Main Thm (6.6)]{CPSVdK} shows that if $k$ is a big finite field, $r$ is big enough and $G$ is reductive  and split over the prime subfield, the map $(*)$ is an isomorphism in low degrees (the isomorphism range grows with the size of $k$).

As an application of our main results, we prove a similar result for $\GL_n(k)$ over an infinite perfect field $k$  in theorem \ref{thm-comp-GL}. Namely, if $V$ and $W$ are two finite-dimensional polynomial representations of degree $d$ of $\GL_n(k)$, if $r$ is big enough and if $n$ is big enough, the map $(*)$ is an isomorphism in low degrees
(an explicit isomorphism range is given in our theorem, this range grows with $n$). Analogous results for symplectic and orthogonal groups are given in theorem \ref{thm-comp-OSp}. 

Let us emphasize two points regarding our comparison theorems. Firstly, many perfect fields do not contain big finite subfields, so there is no hope to recover our results from the results of \cite{CPSVdK} by taking colimits over finite subfields. Actually, \cite{CPSVdK} relies on the fact that when $k$ is a big finite field, the polynomials on a $k$-vector space $V$ are a good approximation of discrete maps from $V$ to $k$, namely the natural forgetful map:
$S(V^*)\to \mathrm{Map}(V,k)$
is surjective, and it is an isomorphism in polynomial degrees less than the cardinality of $k$. This fact obviously fails for infinite perfect fields, which makes our results rather unexpected.
Secondly, our comparison results deal with \emph{stable} cohomology, in the sense that $n$ is large. One can wonder what happens when $n$ is small. It is a complicated question, which goes beyond the current understanding of homology of groups, even in the simplest case $W=V=k$ and $G=\GL_n(k)$. In this case, it is well-known that the left hand-side of $(*)$ is zero in positive degrees \cite[II 4.11]{Jantzen}. In sharp contrast, the right hand-side of $(*)$ is mysterious for small values of $n$, and still under investigation, see e.g. \cite{Mirzaii,GKRW}.

\subsection{Organization of the paper}
We start by two sections recalling the basic facts and notations needed for our results. To be more specific, section \ref{sec-recoll} deals with the homological properties of functor categories and can be read independently of the rest of the paper, as a survey of the fundamental techniques of the subject. Most of our results rely not only on homological techniques, but also on \emph{homotopical} techniques, via the introduction of simplicial objects. Section \ref{sec-simplicial} briefly recalls the material needed, in particular various forms of the Hurewicz theorem. 

In section \ref{sec-AP} we prove theorem \ref{thm-AP-type}. Section \ref{sec-excis} presents a general excision result for functor homology, which implies theorem \ref{thm-intro-2} relative to antipolynomial homology. In these two sections, we work over an arbitrary commutative ring $k$.

The remaining sections deal with polynomial homology over a field $k$. Section \ref{sec-prelim-polyn-hom} describe the manipulations which allow to reduce the computation of polynomial functor homology to the case of an infinite \emph{perfect} field $k$, a small additive \emph{$\Fp$-linear} category $\A$, and arguments of $\Tor$ and $\Ext$ of the form $\pi^*F$, where $F$ is strict polynomial and $\pi$ is additive. These manipulations are well known to experts, but scattered in the literature. The only new result is the polynomial excision theorem \ref{thm-poly-excis}. The results of section \ref{sec-prelim-polyn-hom} justify the hypotheses in the theorems dealing with polynomial homology in the subsequent sections. 

In section \ref{sec-add} we prove theorem \ref{thm-intro-3}, which describes the homology of additive functors. The proof can be considered as a appetizer for the proof of the generalized comparison theorem: the two proofs use similar ideas, but the proof of the generalized comparison theorem is significantly more involved. 

The statement and the proof of the generalized comparison theorem occupies sections  \ref{sec-Pol-Fq}, \ref{sec-gen-prelim-comparison} and \ref{sec-generalized}. In section \ref{sec-Pol-Fq} we elaborate on the results of \cite{FFSS} to remove the harsh restriction on the size of the field in the strong comparison theorem \cite[Thm 3.10]{FFSS}. The resulting theorem \ref{thm-verystrong} is the case $\A=\Proj_\Fq$ of our generalized comparison theorem, and it is a key ingredient in the proof of the generalized comparison theorem. 
In section \ref{sec-gen-prelim-comparison} we define a certain comparison map $\Theta_k$ and we study its properties. This map is a second key ingredient in the proof of our generalized comparison theorem. In section \ref{sec-generalized}, we state and prove the generalized comparison theorem. The idea of the proof is roughly to use $\Theta_k$ to reduce the proof to the case $\A=\Proj_\Fp$, already established in section \ref{sec-Pol-Fq}. 

Finally in section \ref{sec-applic}, we develop quick applications of the generalized comparison theorem. Namely, we use it to obtain some explicit computations of functor homology, and we prove our comparison theorems for the cohomology of classical algebraic groups.

\subsection{Some general notations and terminology}
For the convenience of the reader, we gather here some notations and terminology which are used throughout the article. 
\begin{enumerate}[$\bullet$]
\item The letter $k$ denotes a commutative ring, unadorned tensor products are taken over $k$, $k[X]$ denotes the free $k$-module on a set $X$.
\item Categories have a class of objects, and we require that the morphisms between any two objects form a set. We say that a category $\C$ is \emph{small} if the isomorphism classes of objects of $\C$ form a set (those categories are often called `essentially small' in the literature, we prefer to call them `small' to avoid heavy terminology). We say that a category $\C$ is \emph{additive} if it is a $\mathbb{Z}$-category with a zero object and finite biproducts \cite[VIII.2]{ML}.
\item The letter $\A$ denotes a small additive category, $k[\A]$ is the free $k$-category on $\A$, and $\kA=k\otimes_\mathbb{Z}\A$ is the $k$-category obtained by base change. If $k=\mathbb{Z}/n\mathbb{Z}$, the latter is also denoted by $\A/n$.
\item The letter $R$ denotes a ring, and $\Proj_R$ denotes the category of finitely generated projective left $R^\op$-modules and $R^\op$-linear morphisms (or equivalently finitely generated right modules over $R$). When $R$ is a commutative ring, we identify $\Proj_R$ with a full subcategory of $R\Md$ in the obvious way.
\item Given a functor $\phi:\C\to \D$, we use the same notation to denote the induced functor on the opposite categories $\phi:\C^\op\to \D^\op$.
\end{enumerate}

\section{Prerequisites of functor homology}\label{sec-recoll}

\subsection{Functor categories}
We recall some standard notations and terminology from \cite{Mi72}. In particular, the term \emph{$k$-category} refers to a $k$-linear category, that is, a category whose homomorphism sets are equipped with a $k$-module structure and such that the composition is $k$-bilinear. The $k$-functors are the $k$-linear functors, that is the functors whose action on morphisms is given by a $k$-linear map.

Given a small $k$-category $\K$, we denote by $\K\Md$ the category whose objects are the $k$-functors $F:\K\to k\Md$, whose morphisms are the natural transformations between such functors, the composition being the usual composition of natural transformations. 
We let $\Mdd\K=\K^\op\Md$. Throughout the article, we will use the following notation, which makes a typographical distinction between the level of the source category $\K$ and the level of the functor category $\K\Md$. 
\begin{nota}
Objects of $\K$ are denoted by lowercase letters $x,y,\dots$ and homomorphism groups in $\K$ are denoted by $\K(x,y)$. Objects of $\K\Md$ are denoted by uppercase letters $F,G,\dots$ and $k$-modules of homomorphism in $\K\Md$ are denoted by $\Hom_\K(F,G)$.
\end{nota}

\begin{rk}[smallness]
Smallness of $\K$ ensures that the morphisms between two objects $\K\Md$ form a set.
\end{rk}

The category $\K\Md$ has the structure of a $k$-category, and it is a Grothendieck category (i.e. an AB5 abelian category with a set of generators) with enough projectives and injectives. To be more specific, limits and colimits in $\K\Md$ are computed objectwise, in particular a sequence $0\to F'\to F\to F''\to 0$ is exact if and only if the sequence of $k$-modules $0\to F'(x)\to F(x)\to F''(x)\to 0$ is exact for all objects $x$ of $\K$. The representable functors $h^x_\K=\K(x,-)$ yield a set of projective generators, which are called the \emph{standard projectives}. Finally, if $M$ is an injective cogenerator of $k\Md$, the dual functors $\Hom_k(h^x_{\K^\op},M)$, which send an object $y$ to $\Hom_k(h^x_{\K^\op}(y),M)=\Hom_k(\K(y,x),M)$ yield a set of injective cogenerators, called the \emph{standard injectives}. 

We are chiefly interested in the following classical examples of functor categories. The third one, namely, the category of homogeneous strict polynomial functors, will not be used before section \ref{sec-prelim-polyn-hom}.
\begin{ex}{\bf Ordinary functors.} Let $\C$ be a small category, and let $k[\C]$ be the free $k$-category on $\C$, that is, the category with the same objects as $\C$, and such that $k[\C](a,b)=k[\C(a,b)]$ is the free $k$-module on $\C(a,b)$. 
The category $k[\C]\Md$ identifies with the category $\F(\C;k)$ of all functors $F:\C\to k\Md$ and natural transformations between such functors. The latter category is called the category of ordinary functors, the term `ordinary' refering to the fact that there is no $k$-linearity condition on the functors of this category.
\end{ex}

\begin{ex} {\bf Additive functors.} Let $\A$ be a small additive category, and let $\kA$ be the $k$-category obtained by base change. Thus $\kA$ has the same objects as $\A$, and $\kA(a,b)=k\otimes_\mathbb{Z}\A(a,b)$. 
The category $\kA\Md$ identifies with the category $\add(\A;k)$ of all additive functors $F:\A\to k\Md$ and natural transformations.
\end{ex}

\begin{ex}\label{ex-str-fct} {\bf Homogeneous strict polynomial functors.}
Let $\Gamma^d\Proj_k$ denote the \emph{Schur category}. This $k$-category is defined as follows. First, we denote by $\Gamma^d(v)=(v^{\otimes d})^{\Si_d}$ the $d$-th divided power of a finitely generated projective $k$-module $v$. Then $\Gamma^d\Proj_k$ has the same objects as $\Proj_k$ and 
\[\Gamma^d\Proj_k(v,w)=\Gamma^d(\Hom_k(v,w))=\Hom_{k\Si_d}(v^{\otimes d},w^{\otimes d})\] 
(i.e. the module of $k$-linear morphisms from $v^{\otimes d}$ to $w^{\otimes d}$ which are equivariant for the action of the symmetric group $\Si_d$ which permutes the factors of the tensor product). The composition is the usual composition of equivariant morphisms.
The category $\Gamma^d\Proj_k\Md$ is an avatar of representations of Schur algebras. Indeed, let $S(n,d)=\End_{\Gamma^d\Proj_k}(k^n)$ denote the classical Schur algebra as in \cite{Green}. Then evaluation on $k^n$ yields an equivalence of categories $\Gamma^d\Proj_k\Md\simeq S(n,d)\Md$.

In \cite{FS}, Friedlander and Suslin introduce the category  of $d$-homogeneous strict polynomial functors $\Pp_{d,k}$ over a field $k$. As observed in \cite{Pira-Pan}, this category $\Pp_{d,k}$ identifies with the full subcategory category $\Gamma^d\Proj_k\md$ of $\Gamma^d\Proj_k\Md$ on the functors $F$ such that $F(v)$ is finite-dimensional for all $v$. 

In the case of a field, the equivalence of categories between $\Gamma^d\Proj_k\Md$ and modules over Schur algebras restricts to an equivalence of categories $\Gamma^d\Proj_k\md\simeq S(n,d)\md$. Thus, just as in the case of modules over Schur algebras, the $\Ext$ computed in the category $\Gamma^d\Proj_k\md$ are the same as the $\Ext$ computed in the bigger category $\Gamma^d\Proj_k\Md$. Hence working with the former or the latter is rather a matter of taste. In this article, we choose to work in the bigger category $\Gamma^d\Proj_k\Md$, and we call its objects the \emph{$d$-homogeneous strict polynomial functors} over $k$.
\end{ex}

\subsection{Tensor products and $\Tor$ over a $k$-category}\label{subsec-tp}
Given a small $k$-category $\K$ there is a tensor product over $\K$:
$$-\otimes_{\K}-: \Mdd\K\times \K\Md\to k\Md$$
which is defined by the coend formula \cite[IX.6]{ML} (recall that unadorned tensor products are taken over $k$):
$$F\otimes_\K G=\int^{x}F(x)\otimes G(x)\;.$$
Thus, $F\otimes_\K G$ is the quotient module of $\bigoplus_{x\in\mathrm{Ob}(\K)}F(x)\otimes G(x)$ by the relations $F(f)(t)\otimes s=t\otimes G(f)(s)$ for all $t\otimes s\in F(y)\otimes G(x)$ and all $f\in\K(x,y)$.
\begin{ex}
If $R$ is a $k$-algebra, and $\K=*^R$ is the category with one object $*$ with endomorphism ring equal to $R$, then $\K\Md=R\Md$ and $\otimes_\K$ is nothing but the usual tensor product of $R$-modules.
\end{ex}
Usual properties of tensor products over a ring generalize to tensor products over a category. In particular, the tensor product over $\K$ is characterized by the fact that it is $k$-linear and preserves colimits with respect to each of its variables, together with the `Yoneda isomorphisms', natural with respect to $F$, $G$, $x$, $y$:
\begin{align}F\otimes_{\K} h^x_{\K}\simeq F(x)\quad\text{ and }\quad h^y_{\K^{\mathrm{op}}}\otimes_\K G\simeq G(y)\;.\label{eqn-Yoneda-tens}\end{align}
To be more specific, the first Yoneda isomorphism sends the class $\llbracket s\otimes f\rrbracket$ of an element $s\otimes f\in F(y)\otimes \K(x,y)$ to $F(f)(s)\in F(x)$. The second Yoneda isomorphism is given by a similar formula.

Alternatively, the tensor product over $\K$ is characterized by the isomorphism natural with respect to the functors $F$ and $G$ and the $k$-module $M$:
\begin{align}\alpha:\Hom_k(F\otimes_{\K}G,M)\simeq \Hom_\K(G,\Hom_k(F,M))\;,
\label{eqn-Formule-Cartan}\end{align}
where the functor $\Hom_k(F,M)$ is defined by $\Hom_k(F,M)(x)=\Hom_k(F(x),M)$.
To be more explicit, the natural isomorphism $\alpha$ is defined by sending a morphism $f:F\otimes_\K G\to M$ to the natural transformation $\alpha(f)$ such that the $k$-linear map $\alpha(f)_x:G(x)\to \Hom_k(F(x),M)$ sends an element $s$ to 
$s'\mapsto f(\llbracket s\otimes s'\rrbracket)$ where the brackets refer to the class of $s\otimes s'\in F(x)\otimes G(x)$ in the quotient $F\otimes_\K G$.

The tensor product over $\K$ is a left balanced bifunctor, hence can be left derived in  by taking a projective resolution of $F$ or  a projective resolution of $G$. We denote by $\Tor^\K_*(F,G)$ these derived functors. By deriving the adjunction isomorphism \eqref{eqn-Formule-Cartan} we obtain the following result.
\begin{lm}\label{lm-iso-dual}
For all $k$-functors $F$, $G$ and for all injective $k$-modules $M$, the adjunction morphism $\alpha$ derives to a graded isomorphism:
\[\Hom_k(\Tor_*^\K(F,G),M)\xrightarrow[\simeq]{\alpha} \Ext^*_\K(G,\Hom_k(F,M))\;.\]
\end{lm}

\subsection{Restriction functors}\label{subsec-restriction}

If $\phi:\K\to \LL$ is a $k$-functor, there is a $k$-functor called \emph{restriction along $\phi$} and denoted by $\phi^*$:
\[\begin{array}{cccc}
\phi^*:&\LL\Md& \to &\K\Md\\
& F & \mapsto & \phi^*F:=F\circ\phi
\end{array}\;.\] 
In this article, we mainly deal with the following examples of restriction functors. 
\begin{ex}
{\bf Ordinary functors.}
If $\phi:\C\to\D$ is a functor between two small categories, we still denote by $\phi:k[\C]\to k[\D]$ the induced $k$-functor. If 
we identify $k[\C]\Md$ with the category of ordinary functors $\F(\C;k)$, then the restriction functor $\phi^*:k[\D]\Md\to k[\C]\Md$
identifies with the functor $\F(\D;k)\to \F(\C;k)$ given by precomposition by $\phi$.
\end{ex}

\begin{ex} {\bf Additive functors.} If $\phi:\A\to\B$ is an additive functor between two small additive categories, we also denote by $\phi:\kA\to \kB$ the induced $k$-functor. If 
we identify $\kA\Md$ with the category of additive functors $\add(\A;k)$, then the restriction functor $\phi^*:\kB\Md\to \kA\Md$
identifies with the functor $\add(\B;k)\to \add(\A;k)$ given by precomposition by $\phi$.
\end{ex}

\begin{ex} 
{\bf Ordinary versus additive functors.} Let $\pi:k[\A]\to {}_k\A$ be the $k$-functor which is the identity on objects and which is defined on morphisms by $\pi(\sum \lambda_f f)=\sum \lambda_f\otimes f$. Then the restriction functor $\pi^*:\kA\Md\to k[\A]\Md$ identifies with the embedding $\add(\A;k)\hookrightarrow \F(\A;k)$. In order to avoid overloaded notations, the restriction functor $\pi^*$ will often be omitted, i.e. an additive functor will be denoted by the same letter $F$ when viewed as an object of $\kA\Md$ or as an object of $k[\A]\Md$.
\end{ex}

\begin{ex}\label{ex-ordvsstrict}
{\bf Ordinary versus strict polynomial functors.} Let $\gamma^d:k[\Proj_k]\to \Gamma^d\Proj_k$ be the $k$-functor which is the identity on objects and which is defined on morphisms by $\gamma^d(\sum \lambda_f f)=\sum \lambda_f f^{\otimes d}$. Then the restriction functor $\gamma^{d\,*}:\Gamma^d\Proj_k\md\to k[\Proj_k]\Md$ identifies with the forgetful functor $\Pp_{d,k}\to \F(\Proj_k;k)$. It is fully faithful if $k$ is field with at least $d$ elements \cite[Prop 1.4]{FFSS}. In order to avoid overloaded notations, the restriction functor $\gamma^{d\,*}$ will often be omitted, i.e. a strict polynomial functor will be denoted by the same letter $F$ when viewed as an object of $\Gamma^d\Proj_k\Md$ or as an object of $k[\Proj_k]\Md$.
\end{ex}

Restriction along $\phi$ yields $k$-linear morphisms 
$$\phi^*:\Hom_{\LL}(G,F)\to \Hom_{\K}(\phi^*G,\phi^*F)$$
that we also call restriction morphisms. Similarly there are $k$-linear restriction morphisms, natural with respect to $F$ and $G$:
\[\res^\phi:\phi^*F\otimes_\K \phi^*G\to F\otimes_\LL G\;.\]
To be more specific, $\res^\phi$ sends the class of an element $s\otimes s'\in \phi^*F(x)\otimes \phi^*G(x)$ to the class of the same element $s\otimes s'$ now viewed as an element of $F(y)\otimes G(y)$ with $y=\phi(x)$. These two restriction maps derive to morphisms of graded $k$-modules, which are also denoted by the same letters:
\begin{align}
&\phi^*:\Ext^*_{\LL}(G,F)\to \Ext^*_{\K}(\phi^*G,\phi^*F)\;,\label{eqn-res-1}\\
&\res^\phi:\Tor^{\K}_*(\phi^*F,\phi^*G)\to \Tor^{\LL}_*(F,G)\;.\label{eqn-res-2}
\end{align}
These two restriction maps are related by the following proposition.
\begin{pr}\label{prop-Tor-Ext}
For all objects $F$ and $G$ of $\Mdd\K$ and $\K\Md$, for all $k$-functors $\phi:\K\to \LL$, and for all injectives $k$-modules $M$,
there is a commutative square:
\[
\begin{tikzcd}[column sep=large]
\Ext^*_{\LL}(G,\Hom_k(F,M))\ar{d}{\phi^*}\ar{r}{\alpha}[swap]{\simeq}& \Hom_k(\Tor_*^\LL(F,G),M)\ar{d}{\Hom_k(\res^\phi,M)}\\
\Ext^*_{\K}(\phi^*G,\phi^*\Hom_k(F,M))\ar{r}{\alpha}[swap]{\simeq}&\Hom_k(\Tor_*^\K(\phi^*F,\phi^*G),M)
\end{tikzcd}.
\]
\end{pr}
\begin{proof}
Commutativity of the diagram in degree $0$ is a straightforward verification from the definitions. The commutativity in higher degrees follows from the fact that the arrows are all obtained by deriving the arrows in degree zero.
\end{proof}

Several results in the article assert that under certain conditions, the two restriction maps \eqref{eqn-res-1} and \eqref{eqn-res-2} are isomorphisms.
We will use the next corollaries to reduce the proofs of these results to checking that \emph{one} of these maps is an isomorphism.

\begin{cor}\label{cor-memechose}
Given two functors $F$ and $G$ and an integer $i$, the restriction map \eqref{eq-resmap1} below is an isomorphism if the restriction map \eqref{eq-resmap2} is an isomorphism for all injective $k$-modules $M$. 
\begin{align}
&\phi^*:\Ext^i_\LL(G,\Hom_k(F,M))\to \Ext^i_\K(\phi^*G,\phi^*\Hom_k(F,M))\label{eq-resmap2}\\
&\res^\phi:\Tor^{\K}_i(\phi^*F,\phi^*G)\to \Tor^{\LL}_i(F,G)
\label{eq-resmap1}
\end{align}
\end{cor}

The following standard terminology will be often used in the article.
\begin{defi}\label{def-connected-homology}
Let $e$ be a positive integer.
A morphism $f:H^*\to K^*$ between cohomologically graded $k$-modules is called \emph{$e$-connected} if it is an isomorphism in degrees $i<e$ and if it is injective in degree $i=e$. Similarly, a morphism $g:L_*\to M_*$ between homologically graded $k$-modules is called \emph{$e$-connected} if it is an isomorphism in degrees $i<e$ and if it is surjective in degree $i=e$. 

A morphism of graded $k$-modules is $\infty$-connected if it is $e$-connected for all positive integers $e$.
\end{defi}

The following global version of definition \ref{def-connected-homology} will be used in section \ref{sec-excis}.

\begin{defi-prop}\label{cor-memechose2}
 Let $e$ be a positive integer or $\infty$. The restriction functor $\phi^*:\LL\Md\to \K\Md$ is called \emph{$e$-excisive} if it satisfies one of the following equivalent assertions.
\begin{enumerate}
\item[(1)] For all $F,G$, the map \eqref{eqn-res-1} is an isomorphism in degrees $0\le *< e$.
\item[(1$^+$)] For all $F,G$, the map \eqref{eqn-res-1} is $e$-connected.
\item[(2)] For all $F,G$, the map \eqref{eqn-res-2} is an isomorphism in degrees $0\le *< e$.
\item[(2$^+$)] For all $F,G$, the map \eqref{eqn-res-2} is $e$-connected.
\item[(3)] The restriction functor $\phi^*:\LL\Md\to \K\Md$ is fully faithful and for all objects $x$, $y$ of $\LL$:
\[\bigoplus_{0<i< e}\Tor^\K_i(\phi^*h^x_{\LL^\op},\phi^*h^y_{\LL})=0\;.\] 
\end{enumerate}
\end{defi-prop}
\begin{proof}
It suffices to prove that the five assertions are equivalent for all positive integers $e$. 
It is clear that (1$^+$)$\Rightarrow$(1) and (2$^+$)$\Rightarrow$(2), and the converses follow from inspecting the long exact sequences associated to a short exact sequence $0\to K\to P\to G\to 0$ with $P$ projective. 

Let us prove (1)$\Leftrightarrow$(2). We claim that (1) is equivalent to the following assertion.
\begin{enumerate}
\item[(1')] For all standard injectives $F$ and for all $G$, the map \eqref{eqn-res-1} is an isomorphism in degrees $0\le *< e$.
\end{enumerate}
Indeed, it is clear that (1)$\Rightarrow$(1'). Conversely, every $F$ has a coresolution $J$ by direct products of standard injectives. We have two spectral sequences:
\begin{align*}
&E_1^{p,q}=\Ext_{\K}^q(G,J^p))\Rightarrow \Ext^{p+q}_{\K}(G,F)\;,\\
&'E_1^{p,q}=\Ext_{\LL}^q(\phi^*G,\phi^*J^p)\Rightarrow \Ext^{p+q}_{\LL}(\phi^*G,\phi^*F)\;,
\end{align*}
and $\phi^*$ induces a morphism of spectral sequences. If (1') holds then the two spectral sequences have isomorphic first pages, hence isomorphic abutments, hence (1) holds. Now (1') is equivalent to (2) by corollary \ref{cor-memechose}. 

Finally, let us prove (2)$\Leftrightarrow$(3). If (2) holds, then $\phi^*$ is fully faithful by (1), and moreover  
$\Tor^\K_i(\phi^*h^x_{\LL^\op},\phi^*h^y_{\LL})$ is isomorphic to $\Tor^\LL_i(h^x_{\LL^\op},h^y_{\LL})$ for all $0<i< e$, hence it is zero by projectivity of $h^x_{\LL^\op}$, which proves (3). Conversely, if $\phi^*$ is fully faithful then (2) holds for $e=1$, hence $\res^\phi$ is an isomorphism in degree $0$ by (1). If in addition the $\Tor$-vanishing is satisfied, then the map \eqref{eqn-res-2} is an isomorphism in degrees $0\le *< e$ for all projective objects $F$ and $G$, hence for all objects $F$ and $G$ by a d\'ecalage argument, which proves (2).
\end{proof}

\begin{ex}
If $\phi:\K\to \LL$ is full and essentially surjective, then $\phi^*:\LL\Md\to \K\Md$ is $1$-excisive. Indeed, $\phi^*$ is easily checked to be fully faithful.
\end{ex}

\subsection{Adjoint functors in homology}\label{subsec-Kan-homology}

We now recall from \cite[Lm 1.3 and Lm 1.5]{Pira-Pan} the good homological properties that restriction along $\phi:\K\to \LL$ enjoys when $\phi$ admits a right adjoint.
\begin{pr}\label{pr-iso-Ext-explicit}
Let $\phi:\K\leftrightarrows \LL:\psi$ be an adjoint pair and let $u:\id\to \psi\circ\phi$ and $e:\phi\circ\psi\to\id$ denote the unit and the counit of an adjunction. Then the following composition is an isomorphism:
\[\Ext^*_\LL(\psi^*F,G)\xrightarrow[]{\phi^*}\Ext^*_\K(\phi^*\psi^*F,\phi^*G)\xrightarrow[]{\Ext^*_\K(F(u),G)} \Ext^*_\K(F,\phi^*G)\;,\]
whose inverse is given by the composition:
\[\Ext^*_\K(F,\phi^*G)\xrightarrow[]{\psi^*}\Ext^*_\LL(\psi^*F,\psi^*\phi^*G)\xrightarrow[]{\Ext^*_\LL(F,G(e))} \Ext^*_\LL(\psi^*F,G)\;.\]
\end{pr}

There is a similar result for $\Tor$.
\begin{pr}\label{pr-iso-tens-explicit2}
Let $\phi:\K\leftrightarrows \LL:\psi$ be an adjoint pair and let $u:\id\to \psi\circ\phi$ and $e:\phi\circ\psi\to\id$ denote the unit and the counit of an adjunction. Then the following composition is an isomorphism:
\[\Tor_*^\K(\phi^*F,G)\xrightarrow[]{\Tor_*^\K(\phi^*F,G(u))}\Tor_*^\K(\phi^*F,\phi^*\psi^*G) \xrightarrow[]{\res^\phi}\Tor_*^\LL(F,\psi^*G)\;,\]
whose inverse is given by the composition:
\[\Tor_*^\LL(F,\psi^*G)\xrightarrow[]{\Tor_*^\LL(F(e),\psi^*G)}\Tor_*^\LL(\psi^{*}\phi^{*}F,\psi^*G)\xrightarrow[]{\res^\psi}\Tor_*^\K(\phi^{*}F, G)\;.\]
\end{pr}

\begin{ex}\label{ex-sum-diagonal}
If $n$ is a positive integer, the $k$-functor $\Delta:k[\A]\to k[\A^{\times n}]$ such that $\Delta(x)=(x,\dots,x)$ is adjoint on both sides to the $k$-functor $\Sigma:k[\A^{\times n}]\to k[\A]$ such that $\Sigma(x_1,\dots,x_n)=x_1\oplus\cdots\oplus x_n$. The unit and counit for these two adjunctions are given by the morphisms
\begin{align*}
&\diag :a\to a^{\oplus n}\;, && \mathrm{proj}:  (\bigoplus_{1\le i\le n} {a_i},\dots,\bigoplus_{1\le i\le n} a_i)\to (a_1,\dots,a_n)\;,\\
&\summ:a^{\oplus n}\to a  \;, && \mathrm{incl}: (a_1,\dots,a_n)\to (\bigoplus_{1\le i\le n} {a_i},\dots,\bigoplus_{1\le i\le n} a_i)\;,
\end{align*}
such that the coordinate morphisms of $\diag$ and $\rho$ are equal to $\id_a$ and the coordinate morphisms of $\mathrm{proj}$ and $\mathrm{incl}$ are given by the canonical projections and inclusions.
The $\Ext$-isomorphisms 
\begin{align*}
\Ext^*_{k[\A]}(\Delta^*G,F)\simeq \Ext^*_{k[\A^{\times n}]}(G,\Sigma^*F)\;,\\
\Ext^*_{k[\A]}(F,\Delta^*G)\simeq \Ext^*_{k[\A^{\times n}]}(\Sigma^*F,G)\;,
\end{align*}
and the analogue $\Tor$-isomorphisms will be often refered to as \emph{sum-diagonal adjunction isomorphisms}.
\end{ex}

If the functor $\phi:\K\to \LL$ does not have adjoints, the isomorphisms of propositions \ref{pr-iso-Ext-explicit} and \ref{pr-iso-tens-explicit2} are replaced by spectral sequences. To be more explicit, for all integers $i$, let $L_i^\phi, R^i_\phi: \K\Md\to \LL\Md$
denote the functors such that 
\begin{align*}
(L_i^\phi F)(y)= \Tor_i^\K(\phi^*h^y_{\LL^\op}, F)\;,
&&&
(R^i_\phi F)(y)= \Ext^i_\K(\phi^*h^y_{\LL}, F)\;.
\end{align*}
Then $L_0^\phi$ and $R^0_\phi$ are respectively left ajoint and right adjoint to the restriction functor $\phi^*:\LL\Md\to \K\Md$, and the $L_i^\phi$ and $R^i_\phi$ are the derived functors of $L_0^\phi$ and $R^0_\phi$ respectively. The following proposition is proved in the same way as the usual base change spectral sequences \cite[Chap 5, Thm 5.6.6]{Weibel}. 

\begin{pr}\label{pr-base-change-ss}
There are base change spectral sequences: 
\begin{align*}
&E^{s,t}_2 = \Ext^s_\LL(F,R^t_\phi G)\Rightarrow \Ext^{s+t}_\K(F,\phi^*G)\;,\\
&'E^{s,t}_2 = \Ext^s_\LL(L_t^\phi F,G)\Rightarrow \Ext^{s+t}_\K(F,\phi^*G)\;,\\
&''E_{s,t}^2 = \Tor_s^\LL(F', L_t^\phi G)\Rightarrow \Tor_{s+t}^\K(\phi^*F',G)\;.
\end{align*}
\end{pr}

\subsection{External tensor products}\label{subsec-Kunneth}

We recall from \cite{Mi72} the tensor product $\K\otimes\LL$ of two $k$-categories $\K$ and $\LL$. The objects of this tensor category are the pairs $(x,y)$ where $x$ is an object of $\K$ and $y$ is an object of $\LL$, and its $k$-module morphisms are the tensor products $\K(x,x')\otimes \LL(y,y')$. There is an external tensor product operation
\[\boxtimes:\K\Md\times\LL\Md\to \K\otimes\LL\Md\]
which sends a pair $(F,G)$ to the $k$-functor $(F\boxtimes G)(x,y)=F(x)\otimes G(y)$. 
\begin{pr}[K\"unneth]\label{prop-Kunneth-tens}
Assume that $k$ is a field. There is an isomorphism of graded $k$-vector spaces, natural with respect to $F$, $G$, $H$, $K$:
\[\Tor_*^\K(F,H)\otimes\Tor_*^\LL(G,K)\simeq \Tor_*^{\K\otimes\LL}(F\boxtimes G,H\boxtimes K)\;.\]
\end{pr}
\begin{proof}
Since $h_{\K\otimes\LL}^{(x,y)}$ is equal to $h^x_\K\boxtimes h^y_\LL$, there is an isomorphism 
\[(F\boxtimes G)\otimes_{\K\otimes\LL}(h^x_\K\boxtimes h^y_\LL)\simeq F(x)\otimes G(y)\simeq (F\otimes_{\K}h^x_\K)\otimes (G\otimes_\LL h^y_\LL)\;.\]
Tensor products preserve arbitrary direct sums with respect to each variable. So if $P$, resp. $Q$, is a projective resolution of $H$, resp. $K$, the complex $(F\boxtimes G)\otimes_{\K\otimes\LL} (P\boxtimes Q)$ is isomorphic to the complex $(F\otimes_\K P)\otimes (G\otimes_\LL Q)$. Now $P\boxtimes Q$ is a projective resolution of $H\boxtimes K$, hence the result follows from the usual K\"unneth isomorphism for complexes.
\end{proof}
There is a similar result on the level of extension groups. Namely, the tensor product induces a graded morphism:
\begin{align}\Ext^*_\K(F,H)\otimes\Ext^*_\LL(G,K)\to \Ext^*_{\K\otimes\LL}(F\boxtimes G,H\boxtimes K)\;.\label{eq-Kunneth}\end{align}
However $\Hom$ only preserve finite direct sums, so one needs an additional assumption to adapt the proof of proposition \ref{prop-Kunneth-tens} for $\Ext$. One says \cite[Section 2.3]{Pira-Pan} that a functor is \emph{of type ${fp}_\infty$} if it has a projective resolution by finite direct sums of standard projectives (or equivalently if it has a resolution by finitely generated projectives). With this additional hypothesis, one easily proves:
\begin{pr}[K\"unneth]\label{prop-Kunneth-ext}
Let $k$ be a field. Assume that $F$ and $G$ are of type ${fp}_\infty$, or assume that $F$ is of type ${fp}_\infty$ and that $H$ has only finite-dimensional values. Then morphism \eqref{eq-Kunneth} is an isomorphism. 
\end{pr}
\begin{rk}\label{rk-pfinfty}
Being of type ${fp}_\infty$ is a rather strong property, which is usually very hard to check in an elementary way on a given functor. When $\K=\Proj_{\Fq}$, one may prove such properties by using Schwartz's ${fp}_\infty$-lemma \cite[Prop 10.1]{FLS}, or by using that $k[\Proj_\Fq]\Md$ is locally noetherian \cite{SamSn}. When $\K=k[\A]$ over more general additive categories $\A$, local noetherianity often fails but one can use generalizations of Schwartz's ${fp}_\infty$-lemma given in \cite{DTSchwartz}.
\end{rk}
\subsection{$\aleph$-additivization and exact Kan extensions}\label{subsec-completion}
Let $\aleph$ denote a cardinal.
A $k$-category is called \emph{$\aleph$-additive} if it has all \emph{$\aleph$-direct sums}, that is, if all direct sums indexed by sets of cardinality less or equal to $\aleph$ exist. A $k$-functor between two  such categories is called \emph{$\aleph$-additive} if it preserves $\aleph$-direct sums. The following construction will be used in sections \ref{sec-add} and \ref{sec-gen-prelim-comparison}.
\begin{defi}\label{def-kappa-completion}
The \emph{$\aleph$-additivization} of a small $k$-category $\K$ is the category  $\K^\aleph$ whose objects are the families of objects of $\K$ indexed by sets of cardinality less or equal to $\aleph$. Such an object is denoted by a formal direct sum $\bigoplus_{i\in\I}x_i$. The morphisms $f:\bigoplus_{j\in\mathcal{J}}x_j\to \bigoplus_{i\in\mathcal{I}}y_i$ are the `matrices' $[f_{ij}]_{(i,j)\in \I\times\mathcal{J}}$ such that each $f_{ij}\in \K(x_j,y_i)$, and such that for all $j_0$ only a finite number of morphisms $f_{ij_0}$ are nonzero. The composition of morphisms is given by matrix multiplication.

We identify $\K$ with the full $k$-subcategory of $\K^{\aleph}$ on the formal direct sums with only one object. 
\end{defi}

The definition of morphisms in $\K^\aleph$ shows that the formal direct sum $\bigoplus_{i\in\I}x_i$ is the categorical coproduct of the $x_i$ in $\K^\aleph$, and also the categorical product if $\I$ is finite, which justifies the direct sum notation.
The next elementary proposition gathers the basic properties of $\aleph$-additivization.
\begin{pr}\label{pr-elt-prop}
The $k$-category $\K^\aleph$ is small and $\aleph$-additive. Moreover:
\begin{enumerate}
\item\label{item-n1} An object $x$ of $\K^\aleph$ is isomorphic to a finite direct sum of objects of $\K$ if and only if the functor $\K^\aleph(x,-):\K^\aleph\to k\Md$ is $\aleph$-additive. 
\item\label{item-n2} Every $k$-functor $F:\K\to \LL$ whose codomain is a $\aleph$-additive extends to a unique (up to isomorphism) $\aleph$-additive functor $F^\aleph:\K^\aleph\to \LL$.
\end{enumerate}
\end{pr}
\begin{proof}
Let us prove \eqref{item-n1}. The objects of $\K$ (hence their finite direct sums) are $\aleph$-additive by the definition of morphisms in $\K^\aleph$. Conversely, if $x=\bigoplus x_i$ is an object such that $\K^\aleph(x,-)$ is $\aleph$-additive,  the isomorphism $\K^\aleph(x,\bigoplus x_i)\simeq\bigoplus \K^\aleph(x,x_i)$ shows that $\id_x$ factors through a finite direct sum of the $x_i$, hence $x$ is isomorphic to a finite direct sum of objects of $\K$. Now we prove \eqref{item-n2}. For all objects $x=\bigoplus x_i$ we choose a direct sum $\bigoplus F(x_i)$ in $\LL$. The assignment $F^\aleph(x)=\bigoplus F(x_i)$ defines an $\aleph$-additive functor such that $F^\aleph\circ \iota = F$. Uniqueness follows from the fact that given any pair of $\aleph$-additive functors $F',G':\K^\aleph\to \LL$, every natural transformation $\theta$ between their restrictions to $\K$ extends uniquely into a natural transformation $\theta':F'\to G'$. Indeed it suffices to define $\theta'_{\bigoplus x_i}$ as the unique morphism fitting in the commutative square in which the vertical arrows are the canonical isomorphisms:
\[
\begin{tikzcd}
F'(\bigoplus x_i) 
\ar{r}{\theta'_{\bigoplus x_i}} 
& G'(\bigoplus x_i)\\
\bigoplus F(x_i)
\ar{u}{\simeq}\ar{r}{\bigoplus \theta_{x_i}}
&  \bigoplus G(x_i)\ar{u}{\simeq}
\end{tikzcd}\;.
\]
\end{proof}

The following examples are easily checked by using the universal properties of the categories in play. 
\begin{exs}
\begin{enumerate}[1.] 
\item Let $\mathbf{F}_k$ denote the full subcategory of $k\Md$ on the free $k$-modules. Its $\aleph$-additivization $\mathbf{F}_k^\aleph$ is equivalent to $\mathbf{F}_k$ if $\aleph$ is finite, and to the full subcategory of $k\Md$ on free modules of rank less or equal to $\aleph$ if $\aleph$ is infinite.
\item We have canonical equivalences of categories $k[\A^\aleph]\simeq k[\A]^\aleph$ and ${_k\!}(\A^\aleph)\simeq (\kA)^\aleph$.
\end{enumerate}
\end{exs}

\begin{pr}\label{pr-Eadjoint}
Assume that for all pairs $(x,y)$ of objects of $\K$, the $k$-module $\K(x,y)$ belongs to $\mathbf{F}_k^\aleph$. Then for all objects $x$ of $\K$, the functor
$\K^\aleph(x,-):\K^\aleph\to \mathbf{F}_k^\aleph$ has a left adjoint $-\otimes x:\mathbf{F}_k^\aleph\to \K^\aleph$.
\end{pr}
\begin{proof}
Let $v$ be a free $k$-module with basis $(b_i)_{i\in I}$. We let $v\otimes x=\bigoplus_{i\in I}x_i$ where each $x_i$ denotes a copy of $x$. The $k$-linear map $v\mapsto \K^\aleph(x,v\otimes x)$, sending $b_i$ to the canonical inclusion of $x$ as the factor $x_i$ of $v\otimes x$ is initial in $v\downarrow \K^\aleph(x,-)$, and the result follows from 
\cite[Lm 4.6.1]{Riehl}. 
\end{proof}

\begin{pr}\label{pr-Kexact}
Let $x$ be an object of $\K^\aleph$ decomposed as a direct sum $x=\bigoplus_{i\in I} x_i$ of objects of $\K$. Let $\F$ denote the poset of finite direct summands $\bigoplus_{i\in J}x_i$ ordered by the canonical inclusions. The left Kan extension of a $k$-functor $F:\K\to k\Md$ along the inclusion $\iota:\K\hookrightarrow \K^\aleph$ is given by 
\[\Lan_\iota F(x)=\colim_{y\in \F}F(y)\;.\]
Thus the functor $\Lan_\iota: \K\Md\to \K^\aleph\Md$ is exact and restriction along $\iota$ yields isomorphisms
\begin{align*}
& \Ext^*_{\K}(F,\iota^* G)\simeq \Ext^*_{\K^\aleph}(\Lan_\iota F,G)\;,
& \Tor_*^{\K}(\iota^* G,F)\simeq \Tor_*^{\K^\aleph}(G,\Lan_\iota F)\;.
\end{align*}
\end{pr}
\begin{proof}
Since $\F$ is cofinal in the comma category $\K\downarrow x$, we have $\Lan_\iota F(x)=\colim_{y\in \F}F(y)$.
Since $\F$ is filtered, and since filtered colimits of $k$-modules are exact, this implies that $\Lan_\iota$ is exact. Moreover, $\Lan_\iota$ preserves projectives (because it has an exact right adjoint). This implies the $\Ext$ and $\Tor$ isomorphisms.
\end{proof}

\begin{rk}
We shall typically use the material of this section as follows.
One wants to compute $\Ext_\K^* (h^*F,G)$ where $F:k\Md\to k\Md$ preserves colimits and $h=\Hom_\K(x,-):\K\to \mathbf{F}_k$. One would like to use proposition \ref{pr-iso-Ext-explicit} to say that these extensions are isomorphic to $\Ext_{\mathbf{F}_k}^* (F,t^*G)$, where $t=x\otimes_k-:\mathbf{F}_k\to \K$. However we meet the obstruction that $t$ may not exist because the category $\K$ is too small. To bypass this difficulty, one enlarges $\K$. If $\aleph$ is a large enough cardinal, by proposition \ref{pr-Kexact} $\Ext_\K (h^*F,G)$ is isomorphic to $\Ext^*_{\K^\aleph}({h'}^*F,G')$ where $G'$ is any extension of $G$ to $K^\aleph$, and $h'=\K^\aleph(x,-)$. Over $\K^\aleph$, the desired adjoint $t$ exists by proposition \ref{pr-Eadjoint} hence the latter is isomorphic to  $\Ext_{\mathbf{F}_k^\aleph}^* (F,t^*G')$. And finally we preserve $\Ext$ if we restrict to $\mathbf{F}_k$ by proposition \ref{pr-Kexact} again. 
This technique will appear in sections \ref{sec-add} and \ref{sec-gen-prelim-comparison}.
\end{rk}

\section{Simplicial techniques}\label{sec-simplicial}

\subsection{Simplicial objects in abelian categories}\label{subsec-simpl-ab}
We denote by $s\M$ the category of simplicial objects in an abelian category $\M$. 
The Dold-Kan correspondence \cite[Cor 2.3]{GoerssJardine} or \cite[Section 8.4]{Weibel} yields two mutually inverse equivalences of categories
$$\mathbf{N}:s\M\rightleftarrows \mathrm{Ch}_{\ge 0}(\M):\mathbf{K}$$
between the category $s\M$ of simplicial objects in $\M$ and the category of non-negatively graded chain complexes in $\M$. The functor $\mathbf{N}$ is the functor of normalized chains, and the functor $\mathbf{K}$ is the Kan functor. 
Two simplicial morphisms $f,g:X\to Y$ are \emph{homotopic} if the morphisms of chain complexes $\mathbf{N}f$ and $\mathbf{N}g$ are homotopic.  
The Kan functor preserves homotopies.

The \emph{homotopy groups} $\pi_*X$ of an object $X$ of $s\M$ are defined as the homology groups of the normalized chains $\mathbf{N} X$. 

The category $\M$ is called \emph{concrete} if it is equipped with a faithful functor to the category of sets. If $\M$ is concrete, every object $X$ of $s\M$ has an underlying simplicial set and we let $H^\simpl_*(X;k)$ denote the \emph{homology with coefficients in $k$} of the underlying simplicial set of $X$, that is, $H^\simpl_*(X;k)=\pi_*k[X]$.
\begin{rk}\label{rk-ab}
If the faithful functor $\M\to \mathbf{Set}$ is the composition of an exact functor to the category of abelian groups $\M\to \mathbf{Ab}$ and the usual forgetful functor $\mathbf{Ab}\to \mathbf{Set}$, then the homotopy groups of $X$ coincide with the homotopy groups of its underlying simplicial set by \cite[Cor 2.7]{GoerssJardine} or \cite[Thm 8.3.8]{Weibel}.
\end{rk}

Let $e$ be a non-negative integer. An object $X$ of $s\M$ is \emph{$e$-connected} if $\pi_i X=0$ for all $i\le e$. A morphism $f:X\to Y$ is \emph{$e$-connected} if $\pi_i f: \pi_i X\to \pi_i Y$ is bijective for all $i< e$ and surjective for $i=e$. (Thus $f$ is $e$-connected if and only if its homotopy cofiber is $e$-connected).
A morphism $f:X\to Y$ is a \emph{weak equivalence} if it is $e$-connected for all $e\ge 0$.

A \emph{simplicial projective resolution} of an object $X$ of $s\M$ is a weak equivalence $f:P\to X$ where $P$ is degreewise projective in $\M$. As usual, we identify $\M$ with the full subcategory of $s\M$ on the constant simplicial objects. Hence a simplicial projective resolution of an object $X$ of $\M$ is a degreewise projective simplicial object $P$ such that $\pi_iP=0$ for $i>0$, equipped with an isomorphism $\pi_0P\simeq X$. If $\M$ has enough projectives then by the Dold-Kan equivalence every object of $s\M$ has a simplicial resolution, and every morphism between simplicial objects can be lifted to a morphism of simplicial resolutions, unique up to homotopy.

\subsection{Eilenberg MacLane spaces and Hurewicz theorems}
For all abelian groups $A$ and all $n\ge 0$, we denote by $K(A,n)$ any simplicial free abelian group such that $\pi_iK(A,n)=0$ for $i\ne n$ and $\pi_n K(A,n)\simeq A$. Such a simplicial free abelian group is called an \emph{Eilenberg-MacLane space} and is unique up to homotopy equivalence. The study of simplicial abelian groups often reduces to that of Eilenberg-MacLane spaces by the following classical lemma, see e.g. \cite[Prop 2.20]{GoerssJardine}. We impose that Eilenberg-MacLane spaces are degreewise free abelian groups by definition in order to have genuine maps rather than zig-zags in this lemma.
\begin{lm}\label{lm-decomp}
For all simplicial abelian groups $A$, there is a weak equivalence (unique up to homotopy)
\[\prod_{i\ge 0} K(\pi_i A,i)\to A\;.\]
Moreover for all morphisms of simplicial abelian groups $f:A\to B$, let $K(\pi_i f,i): K(\pi_i A,i)\to K(\pi_i B,i)$ denote a lift of $\pi_i f: \pi_iA\to \pi_i B$ to the level of the simplicial resolutions. Then the following diagram commutes up to homotopy:
\[
\begin{tikzcd}[column sep=huge]
\prod_{i\ge 0} K(\pi_i A,i)\ar{d}\ar{r}{\prod K(\pi_i f,i)}& \prod_{i\ge 0} K(\pi_i B,i)\ar{d}\\
A\ar{r}{f}& B
\end{tikzcd}\;.
\]
\end{lm}
If $A$ is a simplicial abelian group, the morphism of simplicial sets $A\to \mathbb{Z}[A]$ induces a natural morphism of graded abelian groups
\[h_*:\pi_*A\to H_*^\simpl(A;\mathbb{Z})\]
called the \emph{Hurewicz morphism}. The following well-known proposition recalls the classical Hurewicz theorems in the context of simplicial abelian groups (injectivity of $h_*$ as well as the fact that no hypothesis on fundamental groups is needed for the relative version are specific to the abelian group setting).
\begin{pr}[Classical Hurewicz Theorems]\label{pr-Hur-classical}Let $e$ be a nonnnegative integer. 
\begin{enumerate}
\item[(1)] (Absolute theorem)
The Hurewicz map $h_*$ is split injective. Moreover, if $A$ is $e$-connected then $h_i$ is an isomorphism for $i\le e+1$. 
\item[(2)] (Relative theorem)
Every $e$-connected morphism of simplicial abelian groups $f:A\to B$ induces an $e$-connected map $H_*^\simpl(f;\mathbb{Z}):H_*^\simpl(A;\mathbb{Z})\to H_*^\simpl(B;\mathbb{Z})$.  
\end{enumerate}
\end{pr}
\begin{proof}
(1) The canonical morphism of abelian groups $\mathbb{Z}[A]\to A$ yields a retract of $h_*$. The isomorphism is given by \cite[III Thm 3.7]{GoerssJardine}.
(2) Since simplicial groups are fibrant simplicial sets \cite[I Lm 3.4]{GoerssJardine}, any weak equivalence between simplicial groups yields a homotopy equivalence of simplicial sets \cite[II Thm 1.10]{GoerssJardine}, hence it induces an isomorphism in homology. Therefore, lemma \ref{lm-decomp} and the K\"unneth theorem reduce the proof of the isomorphism to the case where $A$ and $B$ are Eilenberg-MacLane spaces, with nonzero homotopy groups placed in the same degree $i$. If $i<e$, $f$ is $e$-connected if and only if it is a weak equivalence, hence if and only if it induces an isomorphism in homology. If $i\ge e$, then $A$ and $B$ are $e$-connected hence the result follows from (1).  
\end{proof}

For our purposes, we need a $k$-local version of the absolute Hurewicz theorem. We shall derive it from the following presumably well-known property of Eilenberg-MacLane spaces, which we have not found in the literature -- though the case of a prime field $k$ is of course given by the  classical calculations of Cartan \cite{Cartan}. 

\begin{lm}\label{lm-EML-vanish}
Let $k$ be a commutative ring, let $A$ be an abelian group. If $k\otimes_\mathbb{Z}A=0$ and $\Tor_1^{\mathbb{Z}}(k,A)=0$, then $H_i^\simpl(K(A,n),k)=0$ for all positive integers $n$ and $i$.
\end{lm}
\begin{proof}
We say that an abelian group $A$ is $k$-negligible if $\Tor_1^{\mathbb{Z}}(k,A)=0=k\otimes_{\mathbb{Z}}A$. 

We first take $n=1$. The lemma holds if $A$ is a $k$-negligible torsion-free group because
$H_i^{\simpl}(K(A,1);k)\simeq  k\otimes_\mathbb{Z}\Lambda^i_{\mathbb{Z}}(A)= \Lambda_k^i(k\otimes_{\mathbb{Z}}A)=0.$
The lemma also holds if $A$ is a $k$-negligible torsion group. Indeed, if $A$ is finite, the lemma holds by a direct computation. If $A$ is infinite, then $A$ is the filtered union of all its finite subgroups $A_\alpha$. And since any subgroup of a $k$-negligible torsion group is $k$-negligible, we have:
$H_i^{\simpl}(K(A,1);k)=\colim_\alpha H_i^{\simpl}(K(A_\alpha,1);k)=0$.
Now let $A$ be an arbitrary abelian group with torsion subgroup $A_{\mathrm{tors}}$. If $A$ is $k$-negligible, then so are $A_{\mathrm{tors}}$ and $A/A_{\mathrm{tors}}$. So the lemma holds for $A$ as a consequence of the Hochschild-Serre spectral sequence of the fibration $K(A_{\mathrm{tors}},1)\to K(A,1)\to K(A/A_{\mathrm{tors}},1)$.

Assume now that $n>1$. Then $K(A,1)\otimes_{\mathbb{Z}} K(\mathbb{Z},n-1)$ is an Eilenberg Mac Lane space $K(A,n)$. Thus $H^\simpl_*(K(A,n))$ is the abutment of the spectral sequence of the bisimplicial $k$-module $M_{pq}= k[K(\mathbb{Z},n-1)_q\otimes_{\mathbb{Z}}K(\pi,1)_p]$. Let us choose $K(\mathbb{Z},n-1)$ such that it is free of finite rank $r(q)$ in each degree $q$ (e.g. take the image of the complex $\mathbb{Z}[-n]$ by the Kan functor). 
Then for $q$ fixed, the simplicial $k$-module $M_{\bullet q}$ is isomorphic to $k[K(A^{\times r(q)},1)]$. Thus the simplicial spectral sequence of $M_{pq}$ can be rewriten as:
\[E^1_{pq}=H^\simpl_p\big(K(A^{\times r(q)},1)\big)\Longrightarrow H^\simpl_{p+q}(K(A,n))\;.\]
The first page is zero by the case $n=1$, whence the result.
\end{proof}

\begin{pr}[$k$-local absolute Hurewicz theorem]\label{pr-Hur}
Let $A$ be a simplicial abelian group, let $k$ be a commutative ring and let $e$ be a non-negative integer. Assume that for $0<i \le e$ and for $0<j<e$ we have $k\otimes_{\mathbb{Z}}\pi_iA=0=\mathrm{Tor}^{\mathbb{Z}}(k,\pi_jA)$. Then 
\begin{enumerate}
\item[(1)] $H^{\simpl}_0(A;k)=k[\pi_0 A]$;
\item[(2)] $H^{\simpl}_i(A;k)=0$ for $0<i\le e$;
\item[(3)] $H^{\simpl}_{e+1}(A;k)$ contains the following $k$-module as a direct summand: 
$$k[\pi_0 A]\otimes_{k}\left(k\otimes_{\mathbb{Z}}\pi_{e+1} A\,\oplus \,\mathrm{Tor}^{\mathbb{Z}}(k,\pi_{e}A)\right)\;.$$
\end{enumerate}
\end{pr}
\begin{proof}
Lemma \ref{lm-decomp} and the K\"unneth theorem reduce the proof to the case of an Eilenberg-MacLane space $A$. Assume that the nonzero homotopy group of $A$ is placed in degree $i$. If $i=0$, the result holds by a direct computation. If $0<i<e$ then the result follows from lemma \ref{lm-EML-vanish}. If $i\ge e$, the result follows from the classical absolute Hurewicz theorem of proposition \ref{pr-Hur-classical} together with the universal coefficient theorem which says that the graded $k$-module 
$H_*^\simpl(A;k)$ is (non canonically) isomorphic to $k\otimes_{\mathbb{Z}} H_*^\simpl(A;\mathbb{Z})\oplus 
\Tor^{\mathbb{Z}}(k,H^\simpl_{*-1}(A;\mathbb{Z}))$.
\end{proof}
\begin{cor}\label{cor-Hur}
The $k$-modules $H^{\simpl}_i(A;k)$ vanish for $0<i\le e$ if and only if $k\otimes_{\mathbb{Z}}\pi_iA$ and $\Tor^\mathbb{Z}_1(k,\pi_jA)$ vanish for $0<i\le e$ and $0<j<e$. 
\end{cor}

\subsection{Functors of simplicial objects}

Assume that $\M$ is an abelian category. Evaluating a functor $F:\M\to s(k\Md)$ on a simplicial object $M$ in $\M$ yields a bisimplicial object $F(M_p)_q$, and we denote by $F(M)$ the associated diagonal simplicial $k$-module. This construction is natural with respect to $F$ and $M$.

We shall say that a natural transformation $f:F\to F'$ is \emph{$e$-connected} if for all $M$ in $\M$, the morphism of simplicial $k$-modules $F(M)\to F'(M)$ is $e$-connected. 
\begin{rk}
If $\M$ is small, the category $k[\M]\Md$ is well-defined. In that case the functors $F:\M\to s(k\Md)$ (resp. the natural transformations between such functors) identify with the simplicial objects in $k[\M]\Md$ (resp. the morphisms between such simplicial objects), and the definition of $e$-connectedness given here coincides with the one given in section \ref{subsec-simpl-ab}.
\end{rk}

\begin{pr}\label{prop-simpl-connected}
Let $\M$ be an abelian category of global dimension $0$. 
For all $e$-connected natural transformations $f:F\to F'$ and for all $e$-connected simplicial morphisms $g:M\to M'$, the induced morphism $F(M)\to F'(M')$ is $e$-connected.
\end{pr}
\begin{proof}
We prove that $f(M):F(M)\to F'(M)$ and $F'(g):F'(M)\to F'(M')$ are  $e$-connected. 
The $e$-connectedness of $f(M)$ follows from the spectral sequence \cite[IV, section 2.2]{GoerssJardine}, which is natural with respect to $F$:
$$E_{m,n}^1(F)= \pi_m (F(M_n))\Rightarrow \pi_{m+n}(F(M))\;.$$

Let us prove the $e$-connectedness of $F'(g)$. We choose a small additive subcategory $\M'\subset \M$ which contains the objects $M_n$ and $M'_n$ for all $n\ge 0$. Let $\pi:P\to F'$ be a simplicial projective resolution in $s(\M'\Md)$. We have a commutative square of simplicial $k$-modules 
$$\begin{tikzcd}
P(M)\ar{r}{P(g)}\ar{d}{\pi(M)}& P(M')\ar{d}{\pi(M')}\\
F(M)\ar{r}{F(g)}& F(M')
\end{tikzcd}$$
whose vertical arrows are weak equivalences by the preceding paragraph, so it suffices to prove that $P(g)$ is $e$-connected. By using the spectral sequence natural with respect to $M$:  
$$F_{m,n}^1(M)= \pi_n(P_m(M))\Rightarrow \pi_{m+n}(P(M))\;,$$
the proof reduces further to showing that for all projective objects $Q$ in $\M'\Md$, the morphism of simplicial $k$-modules $Q(g)$ is $e$-connected. Since $Q$ is a direct summand of a direct sum of standard projectives, we can assume that $Q=k[\M(x,-)]$ with $x\in \M'$. Since $\M$ has global dimension zero, the functor $\M(x,-)$ is exact, in particular $\M(x,g)$ is $e$-connected. Hence, the result follows from the relative Hurewicz theorem of proposition \ref{pr-Hur-classical}.
\end{proof}

\section{Homology of bifunctors of AP-type}\label{sec-AP}

In this section we prove theorem \ref{thm-AP-type} from the introduction. The theorem is stated over an infinite perfect field in the introduction, but actually works over an arbitrary commutative ring, see theorem \ref{thm-AP-type-bis} below. So we fix a commutative ring $k$ and a small additive category $\A$. 

Before stating the theorem, we first come back to the definition of bifunctors of AP-type.
Recall from \cite[p.18]{Mi72} that an \emph{ideal} of the category $\A$ is a subfunctor of $\A(-,-):\A^\op\times \A\to \mathbf{Ab}$. Given such an ideal $\I$, we can form the additive quotient $\A/\I$ of $\A$, with the same objects as $\A$ and with morphisms $({\A/\I})(x,y)=\A(x,y)/\I(x,y)$. We let $\pi_\I:\A\to \A/\I$ denote the additive quotient functor. The following definitions are introduced in \cite{DTV}. 
\begin{defi}
An additive category $\B$ is \emph{$k$-trivial} if for all objects $x$ and $y$ the abelian group $\B(x,y)$ is finite and such that $k\otimes_\mathbb{Z}\B(x,y)=0$.  
An ideal $\I$ of $\A$ is \emph{$k$-cotrivial} if $\A/\I$ is $k$-trivial.
A functor $F:\A\to k\Md$ is \emph{antipolynomial}  if there is a $k$-cotrivial ideal $\I$ of $\A$ such that $F$ factors though $\pi_\I:\A\to \A/\I$. 
\end{defi}
We also need the polynomial functors introduced by Eilenberg and Mac Lane in \cite{EML}. 
An object $F$ of $k[\A]\Md$ is \emph{polynomial} if its $d$-th cross-effect $\cross_{d}F$ vanishes for some nonnegative integer $d$. This $d$-th cross-effect is an object of $k[\A^{\times d}]\Md$, which is a certain direct summand of the 
functor $F(x_1\oplus\cdots\oplus x_d)$. If $\cross_d F=0$ then $\cross_e F=0$ for all $e>d$. Hence there is a biggest integer $d$ such that $\cross_{d}F\ne 0$,s which is called the  \emph{Eilenberg-Mac Lane degree} of $F$ which we denote by $\deg_{\EML}F$.  
We refer the reader to \cite[Section 1]{DTV} for a detailed description of polynomial functors and further references. We will not use the explicit expression of cross effects, but we will rather rely on the following classical properties of polynomial functors. (These properties are explained in \cite[Section 1]{DTV}, except the $\Tor$ vanishing, which can be obtained from the $\Ext$-vanishing and the isomorphism of lemma \ref{lm-iso-dual}.)
\begin{pr} \label{prop-propfctpol}
The full subcategory $k[\A]\Md_{< d}$ of $k[\A]\Md$ on the polynomial functors of degree less than $d$ is stable under subobjects, extensions, arbitrary direct sums and products. 
Moreover, if $F$ and $F'$ are polynomial with $\deg_\EML F<d$ and $\deg_\EML F'<d$, then for all objects $F_i$ of $k[\A]\Md$ and $F'_i$ of $\Mdd k[\A]$ satisfying $F_i(0)=0=F_i'(0)$, we have
\begin{align}
\Ext^*_{k[\A]}(F_1\otimes\cdots\otimes F_d,F)=0=\Ext^*_{k[\A]}(F,F_1\otimes\cdots\otimes F_d)\;,\\
\Tor_*^{k[\A]}(F_1'\otimes\cdots\otimes F_d',F)=0=\Tor_*^{k[\A]}(F',F_1\otimes\cdots\otimes F_d)\;.
\end{align}
\end{pr}

\begin{defi}
A bifunctor $B:\A\times\A\to k\Md$ is of \emph{antipolynomial-polynomial type} (AP-type) if  for all objects $x$ of $\A$ the functor $y\mapsto B(y,x)$ is antipolynomial and the functor $y\mapsto B(x,y)$ is polynomial.
\end{defi}

The main result of the section is the following theorem, which is theorem \ref{thm-AP-type} of the introduction.

\begin{thm}\label{thm-AP-type-bis}Let $k$ be a commutative ring and let $B$, $C$, $B'$ be three bifunctors of AP-type, with $B'$ contravariant in both variables. Restriction along the diagonal $\Delta:\A\to \A\times\A$ yields isomorphisms:
\begin{align*}
&\Ext^*_{k[\A\times\A]}(B,C)\simeq \Ext^*_{k[\A]}(\Delta^*B,\Delta^*C)\;,\\
&\Tor_*^{k[\A]}(\Delta^*B',\Delta^*C)\simeq \Tor_*^{k[\A\times \A]}(B',C)\;.
\end{align*}
\end{thm}

\begin{rk}
The Hom-isomorphism of theorem \ref{thm-AP-type-bis} was already known before. It is included in \cite[Prop 4.9]{DTV} and was proved there by another method.
\end{rk}

The remainder of the section is devoted to the proof of theorem \ref{thm-AP-type-bis}. 
If $B':\A^\op\times\A^\op\to k\Md$ is of AP-type, then for all injective $k$-modules $M$, the bifunctor $\Hom_k(B',M)=\Hom_k(-,M)\circ B'$ is also of AP-type. 
Therefore, by proposition \ref{prop-Tor-Ext}, it suffices to prove the $\Ext$-isomorphism of theorem \ref{thm-AP-type-bis}.
We shall prove this $\Ext$-isomorphism in two steps. We first reduce the proof to bifunctors $B$ `of special-AP-type', that is, bifunctors of the form $B(x,y)=A(x)\otimes P(y)$ for some particular antipolynomial functor $A$ (see definition below). In a second step, we establish the isomorphism for these bifunctors of special AP-type.

\subsubsection*{Step 1: Reduction to bifunctors of special-AP-type}

Given an ideal $\I$ of $\A$ and a positive integer $d$, we denote by $\C_{\I,d}$ the full subcategory of $k[\A\times\A]\Md$ whose objects are the bifunctors $B$ such that:
\begin{enumerate}[i)]
\item $B$ factors through $\pi_\I\times\mathrm{id}:\A\times\A\to (\A/\I)\times \A$, and
\item for all $x$, the functor $y\mapsto B(x,y)$ is polynomial of degree less than $d$. 
\end{enumerate}
The subcategory $\C_{\I,d}$ of $k[\A\times\A]\Md$ is stable under limits and colimits. Stability under colimits ensures that any object $B$ of $k[\A\times\A]\Md$ has a largest subobject $B_{\I,d}$ belonging to $\C_{\I,d}$. 

\begin{lm}\label{lm-exhaust}
If $B$ is a bifunctor of AP-type then 
$B=\bigcup_{\I,d} B_{\I,d}$, where $\I$ runs over the set of $k$-cotrivial ideals of $\A$ and $d$ runs over the set of positive integers. 
\end{lm}
\begin{proof}
We fix two objects $x,y$ of $\A$. Let $d$ be the degree of $t\mapsto B(x,t)$ and let $\I$ be a $k$-cotrivial ideal such that $s\mapsto B(s,y)$ factors through $\A/\I$. To prove the lemma, it suffices to show that the inclusion $B_{\I,d}\hookrightarrow B$ induces an equality $B_{\I,d}(x,y)=B(x,y)$.

Let $B_{d}(a,-)$ be the largest subfunctor of $B(a,-)$ of degree less than $d$. Any map $f:a\to b$ induces a map $B_{d}(a,-)\to B_{d}(b,-)$, so that the functors $B_{d}(a,-)$ assemble into a bifunctor $B_{d}:\A\times\A\to k\Md$ which is a subfunctor of $B$, polynomial of degree less of equal to $d$ with respect to its first variable. By construction $B_{d}(x,y)=B(x,y)$. 
Similarly, let $B_\I(-,b)$ be the largest subfunctor of $B(-,b)$ factorizing through $\A/\I$. These functors assemble into a bifunctor $B_\I:\A\times\A\to k\Md$  factorizing through $\A/\I\times\A$. By construction $B_\I(x,y)=B(x,y)$. 
Since $B_{\I}\cap B_d\subset B_{\I,d}$ we finally obtain that $B_{\I,d}(x,y)=B(x,y)$.
\end{proof}

An object $B$ of $k[\A\times\A]\Md$ is \emph{of special-AP-type} if there is an object $z$ of $\A$, a $k$-cotrivial ideal $\I$ and a polynomial functor $F$ in $k[\A]\Md$ such that 
\[B(x,y)= k[\A/\I(z,x)]\otimes F(y)\;.\]

\begin{lm}\label{lm-resolution}
Let $B$ be an object of $k[\A\times\A]\Md$ such that $B=\bigcup_{\I,d} B_{\I,d}$, where $\I$ runs over the the set of $k$-cotrivial ideals of $\A$ and $d$ over the set of positive integers. Then $B$ has a resolution $Q$ whose terms are direct sums of bifunctors of special-AP-type.
\end{lm}
\begin{proof}
It suffices to prove that every $B_{\I,d}$, is a quotient of a direct sum of bifunctors of special-AP-type. By definition $B_{\I,d}=(\pi_\I\times\mathrm{id}_\A)^*B'$ for some bifunctor $B':\A/\I\times\A\to k\Md$ such that each $B'_z(-):=B'(z,-)$ is polynomial of degree less than $d$. The standard resolution of $B'$ \cite[section 17]{Mi72} yields an epimorphism $\bigoplus_{z} h_{k[\A/\I]}^z\boxtimes B'_z\to B'$, where the sum is indexed by a set of representatives $z$ of isomorphism classes of objects of $\A/\I$. The result follows by restricting this epimorphism along $\pi_\I\times\mathrm{id}_\A$.
\end{proof}

\begin{pr}\label{prop-red-step2}
Let $C$ be an arbitrary object of $k[\A\times\A]\Md$. If the map:
\[\Delta^*:\Ext^*_{k[\A\times\A]}(B,C)\to \Ext^*_{k[\A]}(\Delta^*B,\Delta^*C)\]
is an isomorphism for all bifunctors $B$ of special-AP-type then it is an isomorphism for all bifunctors $B$ of AP-type. 
\end{pr}
\begin{proof}
By lemmas \ref{lm-exhaust} and \ref{lm-resolution},  $B$ has a resolution $Q$ by direct sums of bifunctors of special AP-type. We have two spectral sequences:
\begin{align*}
&E_1^{p,q}=\Ext_{k[\A\times \A]}^q(Q_p,C)\Rightarrow \Ext^{p+q}_{k[\A\times\A]}(B,C)\;,\\
&'E_1^{p,q}=\Ext_{k[\A]}^q(\Delta^*Q_p,\Delta^*C))\Rightarrow \Ext^{p+q}_{k[\A]}(\Delta^*B,\Delta^* C)\;,
\end{align*}
and $\Delta^*$ induces a morphism of spectral sequences. So it suffices to prove that $\Delta^*$ is an isomorphism when $B=Q_p$, hence when $B$ is a functor of special-AP-type. 
\end{proof}

\subsubsection*{Step 2: Proof for bifunctors of special-AP-type} The proof relies on three vanishing lemmas. Lemmas \ref{lm-ssstd} and \ref{lm-vanish-pol} are quite general, and we will also use them later in the article, in the proof of proposition \ref{pr-vanish-Kuhn} and theorem \ref{thm-poly-excis}.

\begin{lm}\label{lm-ssstd}
Let $\C$ and $\D$ be two small categories and let $B$ and $C$ be two objects of $k[\C\times\D]\Md$. There is a first quadrant spectral sequence 
\[E^{p,q}_2=\Ext^p_{k[\D^\op\times \D]}(k[\D];E_\C^q)\Rightarrow \Ext^{p+q}_{k[\C\times\D]}(B,C)\,\]
where $k[\D]$ and $E_\C^q$ are the objects of $k[\D^\op\times \D]\Md$ respectively defined by
\begin{align*} &k[\D](x,y)=k[\D(x,y)]\;, && E_\C^q(x,y)=\Ext^q_{k[\C]}(B(-,x),C(-,y))\;.
\end{align*}
In particular, if $E_\C^*=0$ then $\Ext^{*}_{k[\C\times \D]}(B,C)=0$.
\end{lm}
\begin{proof}
We use the notation $\Hom_{k[C]}(B,C):=E^0_\C$.
There is an isomorphism, natural with respect to $B$, $C$, and the object $D$ of $k[\D^\op\times\D]\Md$:
\[\Hom_{k[\D^\op\times \D]}(D,\Hom_{k[\C]}(B,C))\simeq \Hom_{k[\C\times\D]}(B\otimes_{k[\D^\op]}D,C)\;.\]
This isomorphism is the functor analogue of \cite[IX.2 Prop 2.2]{CE}, and we may construct it as follows. Firstly, there is a natural isomorphism when $D$ is a standard projective. Indeed $D(x,y)=k[\D(x,c)]\otimes k[\D(d,y)]$ and the the two sides are naturally isomorphic to $E_\C^0(c,d)$ by the Yoneda lemma. Now the isomorphism extends to every functor $D$ by taking a projective presentation of $D$.

Thus we have two spectral sequences converging to the same abutment (the construction of these spectral sequences is exactly the same as the one given for categories of modules in \cite[XVI.4]{CE}):
\begin{align*}
&\mathrm{I}^{p,q}=\Ext^p_{k[\D^\op\times\D]}(D,E^q_\C)\Rightarrow H^{p+q}\\
&\mathrm{II}^{p,q}= \Ext^p_{k[\C\times\D]}(\Tor^{k[\D^\op]}_q(B,D),C)\Rightarrow H^{p+q}\;.
\end{align*}
If $D=k[\D]$ then for all $x$, $k[\D](-,x)$ is a projective object of $k[\D^\op]\Md$, hence the functor $\Tor^{k[\D^\op]}_q(B,k[\D])$ is zero for positive $q$, and the Yoneda isomorphism  \eqref{eqn-Yoneda-tens} shows that for $q=0$ this functor is isomorphic to $B$. Thus the second spectral sequence collapses at the second page and $H^*=\Ext^*_{k[\C\times\D]}(B,C)$. Hence the first spectral sequence gives the result.
\end{proof}

The next two lemmas use the notion of a reduced functor. A functor of $k[\A]\Md$ is \emph{reduced} if it satisfies $F(0)=0$. In general, we denote by $F^{\red}$ the \emph{reduced part} of a functor $F$ in $k[\A]\Md$. Thus $F^\red$ is a reduced functor such that $F\simeq F(0)\oplus F^\red$ in $k[\A]\Md$. This decomposition is natural with respect to $F$, in particular we have a decomposition of $\Ext$ (since the full subcategory of constant functors is equivalent to $k\Md$, the $\Ext$ between $F(0)$ and $G(0)$ can be indifferently computed in $k[\A]\Md$ or $k\Md$):
$$\Ext^*_{k[\A]}(F,G)\simeq\Ext^*_k(F(0),G(0))\oplus \Ext^*_{k[\A]}(F^\red,G^\red)$$
and a similar decomposition for $\Tor$. 

\begin{lm}\label{lm-vanish-pol}
Let $G:\A^\op\to \mathbb{Z}\Md$ be such that $\Tor^\mathbb{Z}_1(k,G(x))=0$ and $k\otimes_{\mathbb{Z}}G(x)=0$ for all objects $x$ of $\A$, and let $k[G]$ denote its composition with the $k$-linearization functor $k[-]$. For all functors $H$ and for all reduced polynomial functors $F$ and $F'$ we have:
$$\Tor^{k[\A]}_*(k[G]^\red\otimes H,F)=0 = \Ext^*_{k[\A^\op]}(k[G]^\red\otimes H,F')\;. $$
\end{lm}
\begin{proof}
We prove the $\Tor$-vanishing, the proof of the $\Ext$-vanishing is similar. The functor $k[G]$ is isomorphic to $k[G(0)]\otimes k[G^\red]$, hence to a direct sum of copies of $k[G^\red]$. Therefore, to prove the vanishing, we may assume that $G$ is reduced.  

Let $K$ be the kernel of the augmentation $\epsilon:k[G]\to k$, such that $\epsilon(\sum\lambda_f[f])=\sum \lambda_f$. Then $k[G]^\red=K$.
By lemma \ref{lm-EML-vanish}, the hypotheses on $G(x)$ imply that this abelian group has trivial homology with coefficients in $k$. Thus the reduced normalized bar construction of $k[G(x)]$ yields an exact sequence 
$$ \dots\to K^{\otimes i+1}\to K^{\otimes i}\to \dots \to K\to 0$$ 
in $\Mdd k[\A]$. Since $K$ has $k$-free values, this complex becomes a \emph{split} exact complex of $k$-modules after evaluation on every object $x$ or $\A$. Therefore, if we tensor this sequence by $K^{\otimes r-1}\otimes H$ for a positive integer $r$, we obtain an exact complex with associated hypercohomology spectral sequence:
$$E^1_{s,t}(r)=\Tor_t(K^{\otimes s+r+1}\otimes H,F)\Rightarrow \Tor_t^{k[\A]}(K^{\otimes r}\otimes H,F)\;.$$
Now we can prove that $\Tor_*^{k[\A]}(K^{\otimes r}\otimes H,F)=0$ for all positive integers $r$. Since $K^{\otimes r}\otimes H$ is the direct sum of $K^{\otimes r}\otimes H^\red$ and $K^{\otimes r}\otimes H(0)$, this is true if $r> d$ by lemma \ref{lm-EML-vanish}. Now if this is true for a given $r$, then $E^1_{*,*}(r-1)=0$, hence $\Tor_*^{k[\A]}(K^{\otimes r-1}\otimes H,F)=0$. The result follows.
\end{proof}

\begin{lm}\label{lm-annul}
Let $A$, $P$, and $F$ be three objects of $k[\A]\Md$. Assume that $A$ is antipolynomial and that $P$ is polynomial. Then $\Ext^*_{k[\A]}(A\otimes F,P)= 0$ if $A$ is reduced, and $\Ext^*_{k[\A]}(P\otimes F,A)= 0$ if $P$ is reduced.
\end{lm}

\begin{proof}
We prove the first $\Ext$-vanishing, the proof of second one is similar. 
By sum-diagonal adjunction $\Ext^*_{k[\A]}(A\otimes F,P)$ is isomorphic to 
$\Ext^*_{k[\A\times\A]}(A\boxtimes F,P)$, where $B(x,y)=P(x\oplus y)$ (see example \ref{ex-sum-diagonal}). For all $x$, the functor $x\mapsto P(x\oplus y)$ is polynomial, hence by lemma \ref{lm-ssstd}, it suffices to prove that $\Ext^*_{k[\A]}(A,P)=0$ for all reduced antipolynomial functors $A$ and for all polynomial functors $P$. Since $A$ is reduced, we may as well assume that $P$ is reduced. 
Since $A=\pi^*_\I A'$ for some $k$-cotrivial ideal $\I$ and some functor $A'$ in $k[\A/\I]\Md$, the first base change spectral sequence of proposition \ref{pr-base-change-ss} shows that is suffices to prove that $\Ext^*_{k[\A]}( k[\A/\I(-,x)],P)=0$ for all $x$ in $\A$. The latter follows from  lemma \ref{lm-vanish-pol} (with $G=\A/\I(-,x)$, $H=k$ and $F=P$).  
\end{proof}

The next proposition finishes the proof of theorem \ref{thm-AP-type-bis}.
\begin{pr}
Let $B$ be a bifunctor of special-AP-type, and let $C$ be a bifunctor of AP-type. Restriction along the diagonal yields an isomorphism:
\[\Delta^*:\Ext^*_{k[\A\times\A]}(B,C)\xrightarrow[]{\simeq} \Ext^*_{k[\A]}(\Delta^*B,\Delta^*C)\;.\]
\end{pr}
\begin{proof}
We have $B=A\boxtimes F$ and $\Delta^*B=A\otimes F$, with $F$ polynomial of degree less than $d$, and $A=k[\A/\I(z,-)]$ for a $k$-cotrivial ideal $\I$ and an object $z$ of $\A$.

We have $A\simeq k\oplus A^\red$. We also have an isomorphism $F(x\oplus y)\simeq F(x)\oplus G(x,y)$ for some bifunctor $G$. Since $A(x\oplus y)\simeq A(x)\otimes A(y)$, we have a decomposition 
\[(A\otimes F)(x\oplus y)\simeq \big(A(x)\otimes F(y)\big) \;\oplus \; X(x,y)\;\oplus Y(x,y)\]
where the bifunctors $X$ and $Y$ are defined by
\begin{align*}
&X(x,y):= A(y)\otimes G(x,y)\;,&&Y(x,y):= A(x)\otimes A^\red(y)\otimes F(x\oplus y)\;.
\end{align*}
Moreover, let $q:A\boxtimes F \oplus X\oplus Y \twoheadrightarrow A\boxtimes F$ denote the canonical projection. Then we readily check from the expression of the adjunction isomorphism $\alpha$ given in example \ref{ex-sum-diagonal} that the following diagram commutes:
\[
\begin{tikzcd}[column sep=large]
\Ext^*_{k[\A\times\A]}(A\boxtimes F,C)\ar{d}[swap]{\Ext^*_{k[\A\times\A]}(q,C)}\ar{r}{\Delta^*} & \Ext^*_{k[\A]}(A\otimes F,\Delta^*C)\ar{dl}{\alpha}[swap]{\simeq}\\
\Ext^*_{k[\A\times\A]}(A\boxtimes F \oplus X\oplus Y,C)
\end{tikzcd}\;.
\] 
Therefore, it suffices to prove that there is no nonzero $\Ext$ between $X\oplus Y$ and $C$.  

But for all objects $y$ and $y'$, there is no nonzero $\Ext$ between $X(x,-)$ and $C(x,-)$ by lemma \ref{lm-annul} since $C(-,y')$ is antipolynomial and $X(-,y)$ is polynomial and such that $X(0,y)=0$. Hence there is no nonzero $\Ext$ between $X$ and $C$ by lemma \ref{lm-ssstd}. Similarly, there is no nonzero $\Ext$ between $Y$ and $C$, whence the result.
\end{proof}

\section{Excision in functor homology}\label{sec-excis}

The main purpose of this section is to prove theorem \ref{thm-intro-2} in the introduction, in a version which is valid over a commutative ring $k$, see corollary \ref{cor-proof-thm-ex-intro}. In fact, we will derive theorem \ref{thm-intro-2} from a general `excision theorem', which is a functor homology analogue of the excision theorem of Suslin and Wodzicki in $K$-theory \cite{Suslin-Wodz}, see remark \ref{rk-excis} for further explanations relative to this analogy. We finish the section by a quick computational application to antipolynomial homology.

\subsection{The excision theorem}

\begin{thm}[excision]\label{thm-magique-general}Let $k$ be a commutative ring and let $\phi:\A\to \B$ be an additive functor between two small additive categories, such that $\phi^*:k[\B]\Md\to k[\A]\Md$ is fully faithful.  
For all positive integers $e$, the following assertions are equivalent.
\begin{enumerate}[(1)]
\item\label{item-1-ex} The restriction functor $\phi^*$ is $e$-excisive. 
\item\label{item-2-ex} For all objects $x$ and $y$ of $\B$ we have:
\begin{align*}
\bigoplus_{0\le i< e}k \otimes_{\mathbb{Z}}\pi_iA(x,y) =0 \text{ and }
\bigoplus_{0\le i< e-1}\mathrm{Tor}^{\mathbb{Z}}(k,\pi_iA(x,y))=0\;,
\end{align*}
where $A(x,y)$ is the simplicial abelian group $A(x,y):=Q^x\otimes_\A \phi^*h_\B^y$, where $Q^x$ is a projective simplicial resolution of the functor $\phi^*h_{\B^\op}^x$ in $\Mdd\A$.
\end{enumerate}
\end{thm}
\begin{rk}
A key point in theorem \ref{thm-magique-general} is that it connects two different contexts. Namely, condition \eqref{item-1-ex} deals with the category $k[\A]\Md$ of \emph{all} functors, while the simplicial abelian group $A(x,y)$ appearing in the technical condition \eqref{item-2-ex} are constructed from the category $\A\Md$ of \emph{additive} functors from $\A$ to the category of abelian groups $\mathbb{Z}\Md$.
\end{rk}

In order to prove theorem \ref{thm-magique-general} we need the following general lemma. It is well-known to experts, but we do not know any written reference for it.

\begin{lm}\label{lm-kcrochet-tens}
Let $A:\A^\mathrm{op}\to \mathbb{Z}\Md$ and $B:\A\to \mathbb{Z}\Md$ be two additive functors, and let $k[A]$ and $k[B]$ denote the composition of these functors with the $k$-linearization functor $k[-]$. There is an isomorphism of $k$-modules, natural with respect to $A$ and $B$:
$$k[A]\otimes_{k[\A]} k[B]\simeq k[A\otimes_\A B]\;.$$
\end{lm}
\begin{proof}
For all objects $x$ of $\A$, we let $\theta_{A,B,x}:k[A(x)]\otimes k[B(x)]\to k[A\otimes_{\A}B]$ be the $k$-linear map
such that $\theta_{A,B,x}(s\otimes t)=\llbracket s\otimes t\rrbracket$ for all $s$ in $A(x)$ and all $t\in B(x)$. (The brackets refer to the class of an element of $A(x)\otimes B(x)$ in the quotient $A\otimes_{\A}B$.) These maps $\theta_{A,B,x}$ induce a $k$-linear map, natural in $A$ and $B$:
\[\theta_{A,B}:k[A]\otimes_{k[\A]}k[B]\to k[A\otimes_{\A}B]\;.\]

Assume that $B=\A(x,-)$, hence $k[B]$ is a standard projective in $k[\A]\Md$. One checks on the explicit formulas that the composition of $\theta_{A,B}$ with the $k$-linearization of the Yoneda isomorphism $A\otimes_{\kA}\A(x,-)\simeq A(x)$ is equal to the Yoneda isomorphism $k[A]\otimes_{k[\A]}k[\A(x,-)]\simeq k[A(x)]$. Thus $\theta_{A,\A(x,-)}$ is an isomorphism. Since every finitely generated projective object of $\A\Md$ is a direct summand of a standard projective, it follows that $\theta_{A,B}$ is an isomorphism for all finitely projective functor $B$ in $\A\Md$.

Every projective functor $\A\Md$ is the filtered colimit of its finitely generated projective subfunctors. Since both the source and the target of $\theta_{A,B}$, viewed as functors of $B$, preserve filtered colimits of monomorphisms, we obtain that $\theta_{A,B}$ is an isomorphism for all projectives $B$. 

Now let $B$ be arbitrary and let $P\to B$ be a projective simplicial resolution of $B$ in $\A\Md$. Then we have a commutative square of simplicial $k$-modules in which the top row is an isomorphism, the bottom row features constant simplicial $k$-modules and the vertical arrows are induced by the simplicial maps $P\to B$:
\[
\begin{tikzcd}
k[A]\otimes_{k[\A]}k[P]\ar{d}\ar{r}{\Theta_{A,P}}[swap]{\simeq}& k[A\otimes_{\A}P]\ar{d}\\
k[A]\otimes_{k[\A]}k[B]\ar{r}{\Theta_{A,B}}& k[A\otimes_{\A}B]
\end{tikzcd}
\;.\]
By right exactness of tensor products and by the relative Hurewicz theorem of proposition \ref{pr-Hur-classical}, the vertical morphisms are isomorphisms in $\pi_0$. Hence $\Theta_{A,B}$ is an isomorphism.
\end{proof}

\begin{proof}[Proof of theorem \ref{thm-magique-general}]
Lemma \ref{lm-kcrochet-tens} yields an isomorphism of simplicial $k$-modules:
\[k[A(x,y)]\simeq k[Q^x]\otimes_{k[\A]}k[\phi^*h_\B^y]\;. \]
We can give a homological interpretation of $k[A(x,y)]$ by inspecting the right hand-side of this isomorphism. Indeed we have:
\begin{align*} &k[\phi^*h_\B^y]=k[\B(y,\phi(-))]=\phi^*h^y_{k[\B]}\;, &&&& k[\phi^*h_{\B^\op}^x]=k[\B(\phi(-),x)]=\phi^*h^x_{k[\B^\op]}\;,
\end{align*}
and $k[Q^x]$ is a simplicial projective resolution of $\phi^*h^x_{k[\B^\op]}$ by the relative Hurewicz theorem of proposition \ref{pr-Hur-classical} (with $e=\infty$). Whence an isomorphism:
\begin{align}\pi_*k[A(x,y)]\simeq \Tor_*^{k[\A]}(\phi^*h^x_{k[\B^\op]},\phi^*h^y_{k[\B]})\;.\label{eqn-iso-excis}
\end{align}
Therefore, it follows from the $k$-local Hurewicz theorem of corollary \ref{cor-Hur} that the vanishing of $\pi_ik[\A(x,y)]$ for $0<i< e$ is equivalent to condition \eqref{item-2-ex} of theorem \ref{thm-magique-general}. Thus, the result follows from the isomorphism \eqref{eqn-iso-excis} and proposition \ref{cor-memechose2}.
\end{proof}

\subsection{Some special cases of the excision theorem}
Condition \eqref{item-2-ex} of theorem \ref{thm-magique-general} looks quite technical. We now investigate concrete situations in which condition \eqref{item-2-ex}, hence theorem \ref{thm-magique-general}, can be reformulated in a nicer way. 

As a first example, condition \eqref{item-2-ex} is always satisfied if the category $\B$ is such that $\B(x,x)\otimes_{\mathbb{Z}} k=0$ for all $x$, hence in this case theorem \ref{thm-magique-general} takes the following form.

\begin{thm}\label{thm-magique-csq1}
Let $\A$ and $\B$ be small additive categories, and let $\phi:\A\to \B$ be a full and essentially surjective additive functor. Assume that for all $x$, $\B(x,x)\otimes_\mathbb{Z}k=0$. Then the restriction functor $\phi^*:k[\B]\Md\to k[\A]\Md$
is $\infty$-excisive.
\end{thm}

\begin{proof}
Let $\C$ denote the following full subcategory of abelian groups:
\begin{itemize}
\item if $\operatorname{char} k\ne 0$, the objects of $\C$ are the groups on which multiplication by $\operatorname{char} k$ is invertible;
\item if $\operatorname{char} k=0$, the objects of $\C$ are the torsion groups whose elements have orders invertible in $k$.
\end{itemize}
This subcategory is stable under kernels, cokernels and direct sums. Moreover, for all $A\in \C$ we have $\Tor^\mathbb{Z}_1(A,k)=0=A\otimes_\mathbb{Z} k$.

Now, the hypothesis on $\B$ implies that $\B(x,y)\in\C$ for all $(x,y)$, thanks to lemma~\ref{exoann} below (applied to the rings $k$ and $\B(y,y)$, using that $\B(x,y)$ is a $\B(y,y)$-module). Thus, for all standard projectives $h^a_{\A^\op}$ in $\Mdd\A$, the abelian group $h^a_{\A^\op}\otimes_\A \phi^*h^y_\B=\B(y,\phi(a))$ belongs to $\C$. Therefore, $A(x,y)$ is a simplicial group in $\C$. In particular, its homotopy groups belong to $\C$. Thus the second assertion of theorem \ref{thm-magique-general} is satisfied for all $e$, whence the result.
\end{proof}

\begin{lm}\label{exoann} Let $R$ and $S$ be rings such that $R\otimes_\mathbb{Z}S=0$. Let us denote $r:=\operatorname{char} R$ and $s:=\operatorname{char} S$. Then $(r,s)\ne (0,0)$. Moreover, if $r\ne 0$, then $r$ belongs to $S^\times$.
\end{lm}

\begin{proof} If a tensor product of abelian groups is zero, at least one of them is torsion, whence $(r,s)\ne (0,0)$. If $r\ne 0$, then $\mathbb{Z}/r$ is a direct summand of the additive group of $R$, whence $\mathbb{Z}/r\otimes_\mathbb{Z}S=0$, which implies $r\in S^\times$.
\end{proof}

The next two corollaries are direct consequences of theorem \ref{thm-magique-csq1}. Corollary \ref{cor-proof-thm-ex-intro} provides a proof of theorem \ref{thm-intro-2} from the introduction. 

\begin{cor}
Let $\A$ be a small additive category and let $n$ be an integer invertible in $k$. The restriction functor $\phi^*:k[\A/n]\Md\to k[\A]\Md$ induced by the quotient functor $\phi:\A\to \A/n$ is $\infty$-excisive.
\end{cor}

\begin{cor}\label{cor-proof-thm-ex-intro}
Let $\I$ be a $k$-cotrivial ideal of a small additive category $\A$. The restriction functor $\pi^*:k[\A/\I]\Md\to k[\A]\Md$ induced by the quotient functor $\pi:\A\to \A/\I$ is $\infty$-excisive.
\end{cor}

Under some favorable assumptions on $\A$ and $\B$, theorem \ref{thm-magique-general} can also be reformulated in terms of categories of additive functors. 

\begin{defi}
Let $k$ be a commutative ring. We say that an additive category $\C$ is \emph{$k$-torsion-free} if $\mathrm{Tor}^{\mathbb{Z}}(k,\C(x,y))=0$ for all objects $x$, $y$ of $\C$. 
\end{defi}

\begin{thm}\label{thm-magique2} Let $\phi:\A\to \B$ be an additive functor between two small additive categories, such that $\phi^*:k[\B]\Md\to k[\A]\Md$ is fully faithful.  Assume that that $\A$ and $\B$ are both $k$-torsion free. Then the following assertions are equivalent.
\begin{enumerate}[(1)]
\item The functor $\phi^*:k[\B]\Md\to k[\A]\Md$ is $e$-excisive.
\item The functor $\phi^*:\kB\Md\to \kA\Md$ is $e$-excisive.
\end{enumerate}
\end{thm}
\begin{proof}
We first claim that for all objects $x,y$ and all integers $i$ we have a short exact sequence, where $A(x,y)$ is the simplicial group defined in theorem \ref{thm-magique-general}:
$$0\to k\otimes_\mathbb{Z}\pi_iA(x,y)\to \pi_i(k\otimes_kA(x,y))\to \mathrm{Tor}^{\mathbb{Z}}(k,\pi_{i-1}A(x,y))\to 0\;.$$
Indeed, if $\B$ is $k$-torsion-free, then for all objects $a$ of $\A$,  $\Tor^\mathbb{Z}_1(k,-)$ vanishes on the abelian group $h^a_{\A^{\mathrm{op}}}\otimes_\A \phi^*h^y_{\B^{\mathrm{op}}}\simeq \B(y,\phi(a))$, hence on the abelian groups $A(x,y)_q$
for all $q\ge 0$, and the short exact sequence is given by the universal coefficient theorem \cite[XII Thm 12.1]{MLHom}. Thus, the second assertion of theorem \ref{thm-magique-general} is equivalent to the vanishing of the homotopy groups of $k\otimes_\mathbb{Z} A(x,y)\simeq (k\otimes_\mathbb{Z}Q^x)\otimes_{\kA}(\phi^*h^y_{\kB})$ in degrees $0<i<e$.

Next, we claim that $k\otimes_\mathbb{Z}Q^x$ is a simplicial resolution of $k\otimes_\mathbb{Z}\phi^*h^x_{\B^\op}=k\otimes_{\mathbb{Z}}\B(\phi(-),x)$ in $\Mdd\kA$. If $\A$ is $k$-torsion-free, then $\Tor^\mathbb{Z}_1(k,-)$ vanishes on the objects of $Q^x$, and also on $\pi_0Q^x$ because $\B$ is $k$-torsion-free. Thus the claim follows from the universal coefficient theorem \cite[XII Thm 12.1]{MLHom}. 
As a consequence, proposition \ref{cor-memechose2} tells us that the vanishing of the homotopy groups of $k\otimes_{\mathbb{Z}}A(x,y)$ is equivalent to $\phi^*:\kB\Md\to \kA\Md$ being $e$-excisive. 
\end{proof}

\begin{rk}\label{rk-excis}
The above theorem is a functor homology analogue of Suslin-Wodzicki's excision theorem in rational algebraic $K$-theory \cite{Suslin-Wodz} (see also \cite{Suslin-Kexcision} for the non-rational case). Indeed, the second assertion in theorem \ref{thm-magique2} is a natural generalization of the `$H$-unital' condition which governs excision in $K$-theory.

To be more specific, if $I$ is a two-sided ideal of a ring $R$, and if we consider $\A=\Proj_R$, $\B=\Proj_{R/I}$ and $\phi=-\otimes_R R/I$, then the second assertion of theorem \ref{thm-magique2} is easily seen to be equivalent to 
$$\text{(3)}\quad 
\mathrm{Tor}^{R\otimes_{\mathbb{Z}} k}_i((R/I)\otimes_{\mathbb{Z}} k,(R/I)\otimes_\mathbb{Z} k)=0 \text{ for $0<i<e$.}
$$ 
(To prove the equivalence, use proposition \ref{cor-memechose2} and the fact that $\kA\Md$ is equivalent to $R\otimes_\mathbb{Z}k$ by the Eilenberg-Watts theorem.)
In the situation considered in \cite{Suslin-Wodz}, that is, if  
$R=\mathbb{Z}\oplus I$ where $I$ is a ring without unit (seen as an ideal in the unital ring $R$ constructed by adding formally a unit to $I$) and $k=\mathbb{Q}$, the $\Tor$ appearing in assertion (3) can be computed with a bar complex, hence assertion (3) is equivalent to $R$ being $H$-unital. 
\end{rk}

\subsection{An application to the computation of antipolynomial homology}

The next proposition is a consequence of Kuhn's structure results \cite{Ku-adv}, and corollary \ref{cor-vanish-Kuhn} is the consequence for antipolynomial homology that one immediately deduces from corollary \ref{cor-proof-thm-ex-intro}. 
\begin{pr}\label{pr-vanish-Kuhn}
If $R$ is a finite semi-simple ring and if $k$ is a field of characteristic zero, the $k$-vector spaces $\Ext^i_{k[\Proj_R]}(F,G)$ and $\Tor_i^{k[\Proj_R]}(F,G)$ vanish in positive degrees $i$ for all functors $F$, $G$. 
\end{pr} 
\begin{proof}
The main result of \cite{Ku-adv} says that $k[\Proj_\Fq]\Md$ is equivalent to the infinite product $\prod_{n\ge 0}k[\GL_n(\Fq)]\Md$, which implies the vanishing result when $R$ is a finite field. Assume now that $R$ is a finite simple ring. Then $R$ is isomorphic to a matrix ring $M_n(\Fq)$ and $\Proj_R$ is therefore equivalent to $\Proj_\Fq$ by Morita theory, which implies that the vanishing result holds for finite simple rings. Finally, assume that $R$ is a finite semi-simple ring. Then $R$ is isomorphic to a product $R_1\times\cdots\times R_n$ of simple rings, hence $\Proj_R$ is equivalent to $\Proj_{R_1}\times\cdots\times \Proj_{R_n}$. The vanishing result can then be retrieved from the vanishing result for simple rings by iterated uses of the spectral sequence of lemma \ref{lm-ssstd}.
\end{proof}
\begin{cor}\label{cor-vanish-Kuhn}
Let $k$ be a field of characteristic zero, and let $F,F',G$ be three functors from $\A$ to $k\Md$, with $F$ contravariant. If there is a finite semi-simple ring $R$ such that these three functors factor through $\Proj_R$, then for all $i>0$ we have:
\begin{align*}
&\Ext^i_{k[\A]}(F',G)=0\;, &&  \Tor_i^{k[\A]}(F,G)=0\;.
\end{align*}
\end{cor}

\section{Preliminaries on polynomial homology}\label{sec-prelim-polyn-hom}

We use the term `polynomial homology' as a shorthand for the computation of $\Tor$ and $\Ext$ over $k[\A]$ between polynomial functors. Sections \ref{sec-prelim-polyn-hom} to \ref{sec-generalized} deal with the computation of polynomial homology.  
We will assume that $k$ is a field and we will focus on the polynomial functors of the form
\begin{align}T_F:=\pi_1^*F_1\otimes \cdots\otimes \pi_n^*F_n\label{eq-tp}\end{align}
where the $F_i$ are strict polynomial functors over $k$ and the $\pi_i$ are additive functors from $\A$ to $k$-vector spaces.
The purpose of this short section is to make preliminary remarks on polynomial homology, which justify and explain the assumptions of our theorems in sections \ref{sec-add} and \ref{sec-generalized}. All the material presented in this section is more or less standard, the only new result being the polynomial analogue of excision given in theorem \ref{thm-poly-excis}.

\subsection{The size of the field $k$}\label{subsec-size}

We can most often assume that the field $k$ is as big as we want, in particular infinite and perfect. Indeed, for all field extensions $k\to K$ proposition \ref{prop-Kunneth-tens} yields a base change isomorphism:
\begin{align}
&\Tor^{k[\A]}_*(F,G)\otimes K\simeq \Tor^{K[\A]}_*(F\otimes K,G\otimes K)\;.
\end{align}
We have similar situation for $\Ext$, however one needs suitable finiteness assumptions (namely $F$ is ${fp}_\infty$, or $k\to K$ is a finite extension of fields, see proposition \ref{prop-Kunneth-ext}) to ensure that the following map is an isomorphism
\begin{align}
&\Ext_{k[\A]}^*(F,G)\otimes K\to \Ext_{K[\A]}^*(F\otimes K,G\otimes K)\;.\label{base-change-Ext}
\end{align}

\subsection{Additive functors $\pi_i$ with infinite dimensional values}

Recall that over an infinite field, the category of $d$-homogeneous polynomial functors is a full subcategory of the category $k[\Proj_k]\Md$. In other words, strict polynomial functors  are functors from \emph{finite-dimensional} vector spaces to all vector spaces, hence the meaning of $\pi^*_i F_i = F_i\circ \pi_i$ is clear only when $\pi_i(x)$ is finite-dimensional for all $x$. In general we use the following definition.
\begin{defi}\label{defi-comp}
Let $\overline{F}:k\Md\to k\Md$ denote the left Kan extension to all vector spaces of a functor $F:\Proj_k\to k\Md$. That is, $\overline{F}(v)$ is the colimit of the vector spaces $F(u)$ taken over the poset of finite-dimensional subspaces $u\subset v$ ordered by inclusion. For all additive functors $\pi:\A\to k\Md$ we define $\pi^*F$ as the composition 
\[\pi^*F:=\overline{F}\circ \pi\;.\]
\end{defi}

The following lemma follows from the fact that $\overline{F}$ is defined by taking filtered colimits and that limits and colimits in functor categories are computed objectwise.
\begin{lm}\label{lm-evalinfty}
Sending a strict polynomial functor $F$ to the composition $\pi^*F$ yields an exact and colimit preserving functor 
\[\pi^*:\Gamma^d\Proj_k\Md\to k[\A]\Md\;.\]
As a consequence, we have induced maps on the level of functor homology:
\begin{align*}
&\Ext^*_{\Gamma^d\Proj_k}(F,G)\to \Ext^*_{k[\A]}(\pi^*F,\pi^*G)\;,\\
&\Tor_*^{k[\A]}(\pi^*F,\pi^*G)\to \Tor_*^{\Gamma^d\Proj_k}(F,G)\;.
\end{align*}
\end{lm}

\begin{rk}\label{rk-ambiguous}
The notation $\pi^*F$ used in definition \ref{defi-comp} is compact and it extends the classical notation for composition. However, we shall be careful about the following phenomenon. If $\phi:k\Md\to k\Md$ is an additive functor, we often denote by the same letter $\phi:\Proj_k\to k\Md$ its restriction to finite-dimensional vector spaces. Now the functor $\overline{F}\circ \phi\circ \pi$ need not be be isomorphic to $\overline{F\circ \phi}\circ \pi$ -- though these two functors do coincide if $\phi$ preserves filtered colimits of monomorphisms of vector spaces or if $\pi$ has finite-dimensional values. As a consequence, $\pi^*(\phi^* F)$ might have two different meanings, depending on the fact that we consider $\phi:k\Md\to k\Md$ or its restriction to $\Proj_k$. For this reason, we shall cautiously avoid iterating the notation of definition \ref{defi-comp} and we turn back to notations with compositions whenever there is a risk of ambiguity.
\end{rk}

\subsection{Reducing the number of factors in tensor products}

Computation of $\Ext$ and $\Tor$ between tensor products
of the form \eqref{eq-tp} can always be reduced to the case where there is only one factor in the tensor products. The reduction procedure is well-known (at least to the experts) and we briefly explain it 
here.
We consider two cases, according to the characteristic of the field $k$.
\begin{defi}\label{defi-char-large}
We say that the characteristic of the field $k$ is \emph{large with respect to the tensor product \eqref{eq-tp}} if each $F_i$ is a $d_i$-homogeneous strict polynomial functor such that $d_i!$ is invertible in $k$.
\end{defi}

Notice that according to definition \ref{defi-char-large}, characteristic zero is large.
The following well-known fact is a consequence of classical Schur-Weyl duality for Schur algebras, together with the fact \cite[Thm 3.2]{FS} that the category of $d$-homogeneous strict polynomial functors is equivalent to the category of modules over the Schur algebra $S(n,d)$, if $n\ge d$.

\begin{lm}\label{lm-SWD}
Assume that $d!$ is invertible in the field $k$. Then for all $d$-homogeneous strict polynomial functor $F$ there is a $k[\Si_d]$-module $M$ and an isomorphism, natural with respect to the vector space $v$ 
\[F(v)\simeq v^{\otimes d}\otimes_{k[\Si_d]}M\;.\]
\end{lm}

Assume that the characteristic of $k$ is large with respect to the tensor product \eqref{eq-tp} and let $\Si_{\mathbf{d}}=\Si_{d_1}\times\cdots\times \Si_{d_n}$. Lemma \ref{lm-SWD} yields a $k[\Si_{\mathbf{d}}]$-module $N$ such that 
\[T_F\simeq (\pi_1^{\otimes d_1}\otimes\cdots\otimes\pi_n^{\otimes d_n})\otimes_{k[\Si_{\mathbf{d}}]}N\]
where the action of $\Si_{\mathbf{d}}$ on $\pi_1^{\otimes d_1}\otimes\cdots\otimes\pi_n^{\otimes d_n}$ is given by permuting the factors of the tensor product. If the characteristic of $k$ is also large with respect to another tensor product $T_G:=\rho_1^*G_1\otimes\cdots \otimes\rho^*_mG_m$ where each $G_i$ is a $e_i$-homogeneous strict polynomial functor then we have a similar isomorphism:
\[T_G\simeq (\rho_1^{\otimes e_1}\otimes\cdots\otimes\rho_m^{\otimes e_m})\otimes_{k[\Si_{\mathbf{e}}]}M\;.\]
The assumption on the characteristic also ensures that $k[\Si_{\mathbf{d}}]$ and $k[\Si_{\mathbf{e}}]$ are semisimple, hence we have an isomorphism: 
\begin{align}
\Tor_*^{k[\A]}(T_F, T_G)\simeq \mathbb{T}_* \otimes_{k[\Si_{\mathbf{e}}]\otimes k[\Si_{\mathbf{e}}]} (N\otimes M)\label{eq-calcul-gde-char}
\end{align}
where $\mathbb{T}_*$ denotes the right $k[\Si_{\mathbf{d}}]\otimes k[\Si_{\mathbf{e}}]$-module (with action of $\Si_{\mathbf{d}}$ and $\Si_{\mathbf{e}}$ induced by permuting the factors of the tensor product in the first argument of $\Tor$ and in the second argument of $\Tor$ respectively):
\[\mathbb{T}_*:=\Tor_*^{k[\A]}(\pi_1^{\otimes d_1}\otimes\cdots\otimes \pi_n^{\otimes d_n},\rho_1^{\otimes e_1}\otimes\cdots\otimes\rho_n^{\otimes e_m})\;.\]
Similarly $\Ext^*_{k[\A]}(T_F, T_G)$ can be computed from 
\[\mathbb{E}^*:=\Ext^*_{k[\A]}(\pi_1^{\otimes d_1}\otimes\cdots\otimes \pi_n^{\otimes d_n},\rho_1^{\otimes e_1}\otimes\cdots\otimes\rho_n^{\otimes e_m})\;.\]
Thus it remains to compute $\mathbb{T}_*$ and $\mathbb{E}^*$. This can be achieved by the standard technique using sum-diagonal adjunction isomorphisms (see example \ref{ex-sum-diagonal}) and K\"unneth formulas (see section \ref{subsec-Kunneth}). Some special instances of this computation can be found in the literature, see e.g. \cite[Thm 1.8]{FF} or \cite[Prop 5.4]{TouzeENS}. 
The general formula is not harder to prove but it is combinatorially slightly more involved. We give the result for $\mathbb{T}_*$ (and leave its proof as an exercise to the reader). 
\begin{pr}\label{prop-calcul-bimod}
Let $d=d_1+\cdots + d_n$ and $e=e_1+\dots+e_m$. If $d\ne e$ then $\mathbb{T}_*$ is zero in all degrees. If $d=e$, let
\begin{align*}\alpha:\{1,\dots,d\}\twoheadrightarrow \{1,\dots,n\}\text{ and }
\beta:\{1,\dots,d\}\twoheadrightarrow \{1,\dots,m\}
\end{align*}
be the nondecreasing surjective maps
such that $\alpha^{-1}(i)$ has cardinal $d_i$ and $\beta^{-1}(i)$ has cardinal $e_i$ for all $i$. There is an isomorphism of graded vector spaces:
\begin{align*}
\mathbb{T}_*\simeq\bigoplus_{\sigma\in\Si_d}\mathbf{T}^\sigma_*\;,\text{ with }\;\mathbf{T}^\sigma_* =
\bigotimes_{1\le i\le d}\Tor_*^{k[\A]}(\pi_{\alpha\sigma(i)},\rho_{\beta(i)})\;.
\end{align*}
The action of $(\tau,\mu)\in\Si_{\mathbf{d}}\times\Si_{\mathbf{e}}$ on the right hand side of this isomorphism can be described as follows. If $t=t_1\otimes\cdots\otimes t_n\in \mathbf{T}^\sigma_*$ then $(\tau,\mu)\cdot t$ equals
\[ \epsilon t_{\mu(1)}\otimes\cdots\otimes t_{\mu(n)} \in \mathbb{T}_*^{\tau^{-1}\sigma\mu}\]
where $\epsilon\in\{\pm1\}$ is the Koszul sign such that $t_1\cdots t_n = \epsilon t_{\mu(1)}\cdots t_{\mu(n)}$ in the free graded commutative algebra generated by $t_1,\dots,t_n$.
\end{pr}

There is a similar result for $\Ext$, provided the $\pi_i$ are of type ${fp}_{\infty}$ -- this assumption is needed for the K\"unneth formula, cf. proposition \ref{prop-Kunneth-ext} and remark \ref{rk-pfinfty}. Isomorphism \eqref{eq-calcul-gde-char} and proposition \ref{prop-calcul-bimod} lead us to the following conclusion.

\begin{conclusion}
If the characteristic of the field $k$ is large with respect to the tensor products $T_F$ and $T_G$, then the computation of $\Tor^{k[\A]}_*(T_F,T_G)$ reduces to the computation of $\Tor^{k[\A]}_*(\pi,\rho)$ where $\pi$ and $\rho$ are some of the additive functors appearing in the definition of $T_F$ and $T_G$. A similar reduction to additive functors holds for the computation of $\Ext$-groups, provided the additive functors $\pi_i$ appearing in the definition of $T_F$ are of type ${fp}_\infty$. (See remark \ref{rk-pfinfty} and the article \cite{DTSchwartz} for more details on the  $\mathrm{fp}_\infty$ condition).
\end{conclusion}

If the characteristic of the field $k$ is not large with respect to $T_F$ and $T_G$, we may not be able to reduce ourselves to the computation of $\Ext$ an $\Tor$ between additive functors. But in principle, we can still reduce the computations to something simpler. Namely sum-diagonal adjunction yields an isomorphism 
\[ \Tor_*^{k[\A]}(T_F,T_G)\simeq \Tor_*^{k[\A^{\times m}]}(T_F^{\boxplus^m},\rho_1^*G_1\boxtimes\cdots \boxtimes \rho_m^*G_m)\;,\]
where $T_F^{\boxplus^m}$ is the functor such that 
\[T_F^{\boxplus^m}(x_1,\dots,x_m)=T_F(x_1\oplus\cdots\oplus x_m)=\bigotimes_{1\le i\le n}F_i(\pi_i(x_1)\oplus\cdots\oplus \pi_i(x_m))\]
Each $F_i(v_1\oplus\cdots\oplus v_m)$ has a finite filtration (e.g. its Loewy filtration) whose layers are direct sums of tensor products of the form $L_1(v_1)\otimes\cdots\otimes L_m(v_m)$, in which the $L_i$ are strict polynomial functors. Thus $T_F^{\boxplus^m}$ has a finite filtration whose layers are direct sums of functors of the form $T_{H_1}\boxtimes\cdots\boxtimes T_{H_m}$, where each $T_{H_j}$ is a tensor product of the form \eqref{eq-tp}:
\[T_{H_j}=\pi_1^*H_{1,j}\otimes\cdots\otimes \pi_n^*H_{n,j}\;.\]
Therefore, the K\"unneth formula of proposition \ref{prop-Kunneth-tens} reduces the computation of $\Tor_*^{k[\A]}(T_F,T_G)$ to the computation of the graded $k$-modules
\[\Tor_*^{k[\A]}(T_{H_j},\rho_j^*G_j)\]
In other words, we have reduced the computation of $\Tor_*^{k[\A]}(T_F,T_G)$ to a similar computation, where the tensor product in the right argument of $\Tor$ is now a tensor product with only one factor. (Admittedly, this reduction may be hard to work out in practice since it involves a computation of the filtration of $T_F^{\boxplus^m}$ and the study of the associated long exact sequences in $\Tor$.) 

A similar reasoning then allows to reduce the number of factors of the tensor product in the left hand side argument of $\Tor$, and we obtain the following conclusion.
\begin{conclusion}
In principle, the computation of $\Tor_*^{k[\A]}(T_F,T_G)$ can be reduced to the computation of $\Tor$-groups of the form $\Tor^{k[\A]}_*(\pi^*H,\rho^*K)$, where $H$, $K$ are strict polynomial functors and $\pi$ and $\rho$ are some of the additive functors used in the definition of $T_F$ and $T_G$. A similar reduction holds for the computation of $\Ext$-groups under some suitable ${fp}_\infty$ assumptions.
\end{conclusion}

\subsection{Simplifying the source category $\A$}
If $k$ is a field of positive characteristic $p$, let $\I$ be a ideal of $\A$ contained in $p\A$. Then every additive functor $\A\to k\Md$ factors through $\A/\I$, hence every tensor product of the form \eqref{eq-tp} factors through $\A/\I$. 
The next theorem is a polynomial analogue of the excision theorem.
Under good hypotheses on $\I$, it reduces the computation of polynomial homology over $\A$ to the computation of polynomial homology over $\A/\I$.  
A typical use of this theorem is when $\I(x,y)$ is the abelian subgroup of all the elements of $\A(x,y)$ whose orders are finite and invertible in $k$.
\begin{thm}[Polynomial excision]\label{thm-poly-excis}
Let $k$ be an arbitrary commutative ring, and let $\I$ be an ideal of $\A$ such that $k\otimes_\mathbb{Z}\I(x,y)=0=\Tor_1^\mathbb{Z}(k,\I(x,y))$ for all $x$ and $y$ in $\A$. Then for all polynomial functors $F$ in $k[\A/\I]\Md$ and for all functors $G$ in $k[\A/\I]\Md$ and $G'$ in $\Mdd k[\A/\I]$, restriction along  $\pi:\A\twoheadrightarrow \A/\I$ yields isomorphisms
\begin{align*}
& \Ext^*_{k[\A/\I]}(G,F)\simeq \Ext^*_{k[\A]}(\pi^*G,\pi^*F)\;,&&
& \Tor_*^{k[\A/\I]}(G',F)\simeq \Tor_*^{k[\A]}(\pi^*G',\pi^*F)\;.\\
\end{align*}
\end{thm}
\begin{proof}
We prove the $\Ext$-isomorphism, the proof of the $\Tor$-isomorphism is similar. Since $\pi^*:k[\A/\I]\Md\to k[\A]\Md$ is full and faithful, we only have to check that for all $x$ in $\A$ the functor $\pi^*h^x_{\A/\I}=k[\A/\I(x,-)]$ is $\Hom_{k[\A]}(-,F)$-acyclic. 

If $A$ is an abelian group and $B$ is a subgroup of $A$, then $k[A]$ is a free $k[B]$-module, and $k\otimes_{k[B]}k[A]\simeq k[A/B]$. Therefore the normalized bar construction yields an exact complex, where $K\subset k[B]$ is the augmentation ideal of $k[B]$:
$$\dots\to \underbrace{K^{\otimes i}\otimes k[A]}_{\deg i}\to K^{\otimes i-1}\otimes k[A] \to \dots\to \underbrace{k[A]}_{\deg 0}\to \underbrace{k[A/B]}_{\deg -1}\to 0\;.$$
For all objects $x$ of $\A$, we consider this complex with $B=\I(x,-)$ and $A=\A(x,-)$. The resulting complex is projective in degree $0$, and it is $\Hom_{k[\A]}(-,F)$-acyclic  in positive degrees by lemma \ref{lm-vanish-pol}. Therefore, its degree $-1$ term, which is nothing but $\pi^*h^x_{\A/\I}$, is $\Hom_{k[\A]}(-,F)$-acyclic by a standard dimension shifting argument.
\end{proof}

In some cases, $\A/\I$ is $\Fp$-linear. If not, one can at least hope to obtain information on the polynomial homology over $k[\A/\I]$ from the polynomial homology over $k[\A/p]$ via the base change spectral sequences of proposition \ref{pr-base-change-ss}. In the sequel of the article, we bound ourselves to the study of polynomial homology over an $\Fp$-linear source category.

\section{The homology of additive functors}\label{sec-add}

The purpose of this section is to prove theorem \ref{thm-intro-3} of the introduction, which compares functor homology over $\A$ with functor homology over $k[\A]$, when $k$ is an infinite imperfect field of positive characteristic $p$. 

Although they do not appear in the statement of theorem \ref{thm-intro-3}, Frobenius twist functors do play an important role in its proof, and generic homology of strict polynomial functors is hidden behind the graded vector spaces $E^*_\infty$ and $T^\infty_*$. So we begin by recollections on Frobenius twists and generic homology in subsection \ref{subsec-frobgen}. In section \ref{subsec-homadd} we give a more precise statement of theorem \ref{thm-intro-3}, revealing the role of Frobenius twists and generic homology, and we prove this statement in section \ref{subsec-homadd-proof}.  

\subsection{Frobenius twists and generic cohomology}\label{subsec-frobgen}

Let $k$ be a perfect field (non necessarily infinite in this subsection) of positive characteristic $p$.
For all integers $r$ and all $k$-vector spaces $v$ we denote by $^{(r)}v$ the $k$-vector space which equals $v$ as an abelian group, with action of $k$ given by 
\[\lambda\cdot x:= \lambda^{p^{-r}}x\;.\]
We note that $^{(r)}-$ is an additive endofunctor of $k$-vector spaces which preserves dimension. 
Moreover $^{(0)}v= v$ and $^{(s)}({}^{(r)}v)={}^{(s+r)}v$, hence  $^{(r)}-$ is a self-equivalence of the category of $k$-vector spaces, with inverse $^{(-r)}-$.
\begin{nota}\label{nota-twist-ord}
If $L$ is a perfect field and $F$ is an object of $k[\Proj_L]\Md$ we denote by $F^{(r)}$ the composition of $F$ and $^{(r)}-:\Proj_L\to \Proj_L$.
\end{nota}

When $r> 0$, the functor $^{(r)}-: \Proj_k\to k\Md$ is the underlying ordinary functor of a certain strict polynomial functor.
Indeed, let $\mathrm{sym}:S^{p^r} \to \otimes^{p^r}$ be the symmetrization morphism. This is a morphism of $p^r$-homogeneous strict polynomial functors such that for all finite-dimensional $k$-vector spaces $v$:
\[
\begin{array}{cccc}
\mathrm{sym}:& S^{p^r}(v)&\to & v^{\otimes p^r}\\
& x_1\cdots x_{p^r} & \mapsto & \sum_{\sigma\in\Si_{p^r}}x_{\sigma(1)}\otimes\cdots\otimes x_{\sigma(p^r)}
\end{array}.
\]
The natural morphism ${}^{(r)}v\to S^{p^r}(v)$ which maps $x$ to $x^{p^r}$ identifies $^{(r)}v$ with the kernel of $\mathrm{sym}$. Hence $^{(r)}v$ is actually the underlying functor of a $p^r$-homogeneous strict polynomial functor, namely the kernel of $\mathrm{sym}$, which is called the \emph{$r$-th Frobenius twist functor} and which is denoted by $I^{(r)}$.
The following notation was introduced in \cite{FS}, it is the strict polynomial functor analogue of notation \ref{nota-twist-ord}.
\begin{nota}\label{nota-twist}
For all $d$-homogeneous strict polynomial functors $F$, we denote by $F^{(r)}$ the $dp^r$-homogeneous strict polynomial functor $F^{(r)}:= F\circ I^{(r)}$.
\end{nota}
\begin{rk}
The definition of composition of strict polynomial functors is the obvious one if we think of strict polynomial functors in the way they are defined in \cite{FS}. If we use the description of strict polynomial functors as $k$-linear functors with domain the Schur category as in example \ref{ex-str-fct}, then $F^{(r)}$ is the restriction of $F$ along the $k$-linear functor $\Gamma^{dp^r}\Proj_k\to \Gamma^d\Proj_k$ which sends a vector space $v$ to its Frobenius twist $v^{(r)}$, and whose action on morphisms is induced by the verschiebung map 
$$\mathbf{v}:\Gamma^{dp^r}\Hom_k(u,v)\to \Gamma^{d}({}^{(r)}\Hom_k(u,v))\simeq \Gamma^{d}(\Hom_k({}^{(r)}u,{}^{(r)}v))\;.$$
\end{rk}

One checks that $(I^{(r)})^{(s)}= I^{(r+s)}$ for all positive integers $r,s$. This formula extends to all non-negative integers $r$ and $s$ if we define $I^{(0)}$ as the $1$-homogeneous functor such that $I^{(0)}(v)=v$. 

\begin{rk}
The strict polynomial functor $I^{(0)}$ is known under many different names, namely we have isomorphisms of $1$-homogeneous strict polynomial functors $I^{(0)}\simeq S^1\simeq \Lambda^1\simeq \Gamma^1\simeq \otimes^1$. The functor is also commonly denoted by the letter $I$, this simpler notation being consistent with notation \ref{nota-twist}, i.e. $I^{(r)}=I\circ I^{(r)}$.
\end{rk}


The next result was first established in \cite[Cor 1.3 and Cor 4.6]{FFSS}. 

\begin{defi-prop}\label{pdef-gen-Ext}
Let $F$ and $G$ be two $d$-homogeneous strict polynomial functors. The maps given by precomposition by $I^{(1)}$:
\[\Ext^i_{\Gamma^{dp^r}\Proj_k}(F^{(r)},G^{(r)})\to \Ext^i_{\Gamma^{dp^{r+1}}\Proj_k}(F^{(r+1)},G^{(r+1)})\]
are always injective, and they are isomorphisms if $i<2p^r$. The stable value is called the vector space of \emph{generic extensions of degree $i$} and denoted by $\Ext^i_\gen(F,G)$:
\begin{align*}
\Ext^i_\gen(F,G)&:=\colim_{r}\Ext^i_{\Gamma^{dp^r}\Proj_k}(F^{(r)},G^{(r)})\simeq \Ext^i_{\Gamma^{dp^r}\Proj_k}(F^{(r)},G^{(r)})\text{ if $r\gg 0$.}
\end{align*}
\end{defi-prop}

\begin{ex}\label{ex-FS-key} The generic cohomology of the simplest case $F=G=I$ is a key computation in \cite{FS}. We have \cite[Thm 4.5]{FS}:
\[\Ext^i_{\Gamma^{p^r}\Proj_k}(I^{(r)},I^{(r)})=
\begin{cases}
k & \text{if $i$ is even and $i<2p^r$,}\\
0 & \text{otherwise.} 
\end{cases}
\]
Therefore, $\Ext^i_{\gen}(I,I)$ equals $k$ in even degrees and $0$ in odd degrees.
\end{ex}

We refer the reader to \cite{Touze-Survey} for a survey of these generic extensions and some formulas to compute them (which simplify and generalize \cite{FFSS}). See also section \ref{subsec-sample}. 
We can define generic $\Tor$ in the same fashion as generic $\Ext$. In order to compute $\Tor$, we need objects of $\Mdd\Gamma^d\Proj_k$ that we call \emph{contravariant $d$-homogeneous strict polynomial functors}. By applying proposition \ref{prop-Tor-Ext}, we can dualize the $\Ext$ situation, and we obtain the next statement.

\begin{defi-prop}\label{pdef-gen-Tor}
Let $F$ and $G$ be two $d$-homogeneous strict polynomial functors, with $F$ contravariant. The maps given by precomposition by $I^{(1)}$
\[\Tor^i_{\Gamma^{dp^{r+1}}\Proj_k}(F^{(r+1)},G^{(r+1)})\to \Tor^i_{\Gamma^{dp^{r}}\Proj_k}(F^{(r)},G^{(r)})\]
are always surjective, and they are isomorphisms if $i<2p^r$. The stable value is called the vector space of \emph{generic torsion of degree $i$}
and denoted by $\Tor^\gen_i(F,G)$:
\begin{align*}\Tor_i^\gen(F,G)&:=\lim_{r}\Tor_i^{\Gamma^{dp^r}\Proj_k}(F^{(r)},G^{(r)})\simeq \Tor_i^{\Gamma^{dp^r}\Proj_k}(F^{(r)},G^{(r)})\text{ for $r\gg 0$.}
\end{align*}
\end{defi-prop}

\subsection{The main result on the homology of additive functors}\label{subsec-homadd} We now state a more precise form of theorem \ref{thm-intro-3} of the introduction, which makes explicit the role of Frobenius twists and generic homology. Let $E_r^*=\Ext^*_{\Gamma^{p^r}\Proj_k}(I^{(r)},I^{(r)})$, and let
\[E_\infty^*=\Ext^*_\gen(I,I)=\colim_r E_r^*\;.\]
The graded vector space $E_\infty^*$ is described in example \ref{ex-FS-key}: it is one-dimensional in even degrees, and zero in odd degrees. For all even nonnegative integers $i$, we fix a basis vector $e_\infty(i)$ of $E_\infty^i$.

We explain now how the cohomology class $e_\infty(i)$ determines a cohomology class $e_\rho(i)\in \Ext^i_{k[\A]}(\rho,\rho)$ for all additive functors $\rho:\A\to k\Md$. 
Firstly, since the category $\A$ is small, there is a cardinal $\aleph$ such that $\rho$ has values in the category $\Proj_k^{\aleph}$ of vector spaces of dimension less or equal to $\aleph$. 
For all nonnegative integers $r$, the functor $\rho$ can then be written as the composition:
\[ \A\xrightarrow[]{{}^{(-r)}\rho} \Proj_k^{\aleph} \xrightarrow[]{{}^{(r)}-} k\Md\;,\]
where ${}^{(-r)}\rho= ({}^{(-r)}-)\circ\rho$ stands for the composition of $\rho$ and the $(-r)$-th Frobenius twist. Secondly, we consider the map obtained  by composing the forgetful functor from strict polynomial to ordinary functors as in example \ref{ex-ordvsstrict}, the $\Ext$-isomorphism of proposition \ref{pr-Kexact} provided by restriction along the inclusion $k[\Proj_k]\hookrightarrow k[\Proj_k^\aleph]$, and restriction along ${}^{(-r)}\rho:\A\to \Proj_k^\aleph$: 
\begin{equation}
E_r^i\to \Ext^i_{k[\Proj_k]}({}^{(r)}-, {}^{(r)}-)\simeq \Ext^i_{k[\Proj_k^\aleph]}({}^{(r)}-, {}^{(r)}-)\xrightarrow[]{({}^{(-r)}\rho)^*}
\Ext^i_{k[\A]}(\rho,\rho)\;.
\label{eq-compo-add}
\end{equation}
Now we denote by $e_r(i)$ the vector basis of $E_r^i$ representing $e_\infty(i)$ in the colimit, and we define $e_\rho(i)$ to be the image of $e_r(i)$ by the previous map \eqref{eq-compo-add}.
\begin{lm}
The class $e_\rho(i)$ does not depend on the choice of $r$ and $\aleph$.
\end{lm}
\begin{proof}
If $\beth$ is a cardinal greater or equal to $\aleph$ and if $s\ge r$, we have a commutative diagram in which the two vertical arrows on the left are induced by pullback along $(s-r)$-th Frobenius twists, the third vertical arrow is induced by restriction along the $(r-s)$-th Frobenius twist and the inclusion $k[\Proj_k^\aleph]\hookrightarrow k[\Proj_k^\beth]$:
\[\begin{tikzcd}
E_r^i\ar{r}\ar{d}{\simeq}&\Ext^i_{k[\Proj_k]}({}^{(r)}-, {}^{(r)}-)\ar{d}{\simeq}&\ar{l}[swap]{\simeq}\Ext^i_{k[\Proj_k^\aleph]}({}^{(r)}-, {}^{(r)}-)\ar{r}{({}^{(-r)}\rho)^*}
&\Ext^i_{k[\A]}(\rho,\rho)\\
E_s^i\ar{r}&\Ext^i_{k[\Proj_k]}({}^{(s)}-, {}^{(s)}-)&\ar{l}[swap]{\simeq}\Ext^i_{k[\Proj_k^\beth]}({}^{(s)}-, {}^{(s)}-)\ar{r}{({}^{(-s)}\rho)^*}\ar{u}[swap]{\simeq}
&\Ext^i_{k[\A]}(\rho,\rho)\ar[equal]{u}
\end{tikzcd}\;.\]
\end{proof}

For all even integers $i$ and all integers $j$, we denote by $\Upsilon_{ij}$ the composition
\[\Upsilon_{ij}:\Ext^j_{\kATor}(\pi,\rho)\otimes E_\infty^i\to \Ext^{j}_{k[\A]}(\pi,\rho)\otimes E^i_\infty\to \Ext^{i+j}_{k[\A]}(\pi,\rho)\]
where the first map is induced by the forgetful functor $\kA\Md\to k[\A]\Md$ and the second one sends $e\otimes e_\infty(i)$ to the Yoneda splice $e_\rho(i)\circ e$.
The maps $\Upsilon_{ij}$ assemble into a graded $k$-linear map
\begin{align}
\Upsilon: \Ext^*_{\kATor}(\pi,\rho)\otimes E^*_\infty\to \Ext^*_{k[\A]}(\pi,\rho)\;.\label{eq-Upsilon}
\end{align}
The following result is the main result of this section, and our improved form of the $\Ext$-isomorphism of theorem \ref{thm-intro-3}.
\begin{thm}\label{thm-main-additive}
Let $k$ be an infinite perfect field of positive characteristic $p$, and let $\A$ be a small additive category, which we assume to be $\Fp$-linear.  For all additive functors  $\rho,\pi:\A\to k\Md$, the map $\Upsilon$ defined in equation \eqref{eq-Upsilon} is an isomorphism of graded $k$-vector spaces.
\end{thm}

Before proving theorem \ref{thm-main-additive}, we observe that theorem \ref{thm-main-additive} implies not only the $\Ext$-isomorphism in theorem \ref{thm-intro-3} but also the $\Tor$-isomorphism therein. Hence the whole of theorem  \ref{thm-intro-3} is actually a direct consequence of theorem \ref{thm-main-additive}.
\begin{cor}\label{cor-mainadditive-tor}
Let $T^\infty_*$ denote the graded vector space equal to $k$ in even non-negative degrees and to $0$ in the other degrees. There is a graded isomorphism, natural with respect to $\pi$ and $\rho$: 
\[\Tor_*^{\kATor}(\pi,\rho)\otimes T^{\infty}_*\simeq  \Tor_*^{k[\A]}(\pi,\rho)\;.\]
\end{cor}
\begin{proof}[Proof of corollary \ref{cor-mainadditive-tor}]
The graded vector spaces $\Hom_k(\Tor_*^{\K}(\pi,\rho),M)$ are naturally isomorphic to $\Ext^*_\K(\rho,\Hom_k(\pi,M))$ for $\K=\kA$ or $\K=k[\A]$ and for all vector spaces $M$. So theorem \ref{thm-main-additive} yields an isomorphism, natural in $\pi$, $\rho$ and $M$:
\begin{align}\Hom_k(\Tor_*^{\kATor}(\pi,\rho),M)\otimes E_\infty^*\simeq \Hom_k(\Tor_*^{k[\A]}(\pi,\rho),M)\label{eq-dual-iso}\end{align}
Note that $E_\infty^*$ is the $k$-linear graded dual of $T^\infty_*$. Since $E^\infty_*$ is degreewise finite-dimensional, there is a canonical isomorphism: 
\begin{align}\Hom_k(\Tor_*^{\kATor}(\pi,\rho),M)\otimes E_\infty^*\simeq \Hom_k(\Tor_*^{\kATor}(\pi,\rho)\otimes T_*^\infty,M)\;.\label{eq-dual-iso2}
\end{align}
Thus the right hand sides of \eqref{eq-dual-iso2} and  \eqref{eq-dual-iso} are naturally isomorphic with respect to $M$, whence the result.
\end{proof}

\subsection{Proof of theorem \ref{thm-main-additive}}\label{subsec-homadd-proof} 
We fix a small category $\A$ and an infinite perfect field $k$ of positive characteristic.
In order to prove that the map $\Upsilon$ is an isomorphism, we proceed in three steps. 
\begin{enumerate}[1.]
\item In lemma \ref{lm-red-proj}, we reduce the proof to the case of $\pi=k\otimes_{\mathbb{Z}}\A(a,-)$. This step relies on a standard spectral sequence argument.
\item In lemma \ref{lm-red-Fp}, we reduce the proof further to $\A=\Proj_\Fp$. This step uses an adjoint of $\A(a,-):\A\to \Fp\Md$. Since the adjoint does not necessarily exist, we must first replace $\A$ by a larger additive category $\A^\aleph$.
\item Finally, we prove the case $\A=\Proj_\Fp$ and $\pi=k\otimes_{\mathbb{Z}}\Hom_\Fp(a,-)$ in lemma \ref{lm-finish-proof}, as an application of the computations of \cite{FLS} and \cite{FS}.
\end{enumerate}
\begin{lm}\label{lm-red-proj}
Fix $\rho$ in $\kA\Md$. If the morphism \eqref{eq-Upsilon} is an isomorphism for $\pi=k\otimes_{\mathbb{Z}}\A(a,-)$ for all objects $a$ of $\A$, then it is an isomorphism for all $\pi$.
\end{lm}
\begin{proof}
Assume that the morphism \eqref{eq-Upsilon} is an isomorphism for $\pi=k\otimes_{\mathbb{Z}}\A(a,-)$ for all objects $a$ of $\A$. 
The source and the target of $\Upsilon$, regarded as functors of the variable $\pi$, turn direct sums into products. Since every projective object of $\kA\Md$ is a direct summand of a direct sum of functors of the form $k\otimes_{\mathbb{Z}}\A(a,-)$, this implies that the map \eqref{eq-Upsilon} is an isomorphism whenever $\pi$ is projective in $\kA\Md$. 

Now let $\pi$ be an arbitrary object of $\kA\Md$, let $P$ be a projective resolution of $\pi$ in $\kA\Md$, and let $Q$ be an injective resolution of $\rho$ in $k[\A]\Md$. We consider the bicomplexes:
\begin{align*}
C^{pq}=\Hom_{\kATor}(P_p,\rho)\otimes E_\infty^q\;,&&& D^{pq}=\Hom_{k[\A]}(P_p,Q^q)\;.
\end{align*}
Here we consider $E_\infty^*$ as a complex with zero differential, hence the second differential of $C$ is zero. We have two associated spectral sequences:
\begin{align*}
&E^{pq}_1(C)=\Hom_{\kATor}(P_p,\rho)\otimes E^q_\infty\Rightarrow  (\Ext_{\kATor}(\pi,\rho)\otimes E_\infty)^{p+q}\;,\\
&E^{pq}_1(D)=\Ext^q_{k[\A]}(P_p,\rho)\Rightarrow  \Ext_{k[\A]}^{p+q}(\pi,\rho)\;.
\end{align*}
For all even integers $q$, we choose a cycle $z(q)$ representing $e_\rho(q)$ in the complex $\Hom_{k[\A]}(\pi,Q)$. Then the morphism of bicomplexes $\Phi^{pq}:C^{pq}\to D^{pq}$ such that $\Phi^{pq}(f\otimes e_\infty(q))=z(q)\circ f$ induces a morphism of spectral sequences $E(\Phi)$.
By construction, the morphisms
\begin{align*}
&E^{p,*}_1(\Phi):\Hom_{\kATor}(P_p,\rho)\otimes E^*_\infty\to \Ext^*_{k[\A]}(P_p,\rho)\\
&\Tot(\Phi): \Ext^*_{\kATor}(\pi,\rho)\otimes E_\infty^*\to \Ext_{k[\A]}^{*}(\pi,\rho)
\end{align*}
are both equal to $\Upsilon$. Thus $E_1(\Phi)$ is an isomorphism (since the morphism \eqref{eq-Upsilon} is an isomorphism on projective objects of $\kA\Md$), which implies that $\Tot(\Phi)$ is an isomorphism.
\end{proof}

\begin{lm}\label{lm-red-Fp}
If theorem \ref{thm-main-additive} holds for $\A=\Proj_\Fp$ then it holds for all small additive $\Fp$-linear categories $\A$. 
\end{lm}
\begin{proof}
Let us fix a small additive $\Fp$-linear category $\A$. By lemma \ref{lm-red-proj}, it suffices to prove that the map \eqref{eq-Upsilon} is an isomorphism when $\pi=k\otimes_{\mathbb{Z}}\A(a,-)$. 

Let $\aleph$ be a cardinal larger than the cardinal of $\A(x,y)$ for all $x$ and $y$, and let $\A^\aleph$ be the $\aleph$-additivization of the $\Fp$-category $\A$, as in definition \ref{def-kappa-completion}. Let $\pi'=k\otimes_{\mathbb{Z}}\A^\aleph(a,-)$ and let $\rho':\A^\aleph\to k\Md$ be an arbitrary extension of $\rho$. We have a commutative diagram in which the vertical arrows are induced by restriction along the inclusions $\A\hookrightarrow \A^\aleph$:
\begin{equation}
\begin{tikzcd}
\Ext^*_{\kATor^\aleph}(\pi',\rho')\otimes E^*_\infty \ar{r}{\Upsilon}\ar{d}{\simeq} & \Ext^*_{k[\A^\aleph]}(\pi',\rho')\ar{d}{\simeq}\\
\Ext^*_{\kATor}(\pi,\rho)\otimes E^*_\infty \ar{r}{\Upsilon} & \Ext^*_{k[\A]}(\pi,\rho)
\end{tikzcd}\;.\label{dgm-enlarge}
\end{equation}
The explicit formula for Kan extensions given in proposition \ref{pr-Kexact} shows that $\pi'$ is the left Kan extension of $\pi$ (regarded as an object of $\kA\Md$ or as an object of $k[\A]\Md$), hence the vertical arrows are isomorphisms. Thus it suffices to check that the upper $\Upsilon$ is an isomorphism.

Proposition \ref{pr-Eadjoint} gives an adjoint pair $\A^\aleph(a,-):\A^\aleph\leftrightarrows \Proj_{\Fp}^\aleph:a\otimes -$.
One can write $\pi'$ as the composition of the functor $\A^\aleph(a,-)$ with the functor $I':\Proj_{\Fp}^\aleph\to \Proj_{k}^\aleph$ such that $I'(v)=k\otimes_{\Fp}v$. Hence we have adjunction isomorphisms
\begin{align*}
&\Ext^*_{\kATor^\aleph}(\pi',\rho')\simeq \Ext^*_{\Proj_{\Fp}^\aleph}(I',\rho'(a\otimes -))\;,\\
&\Ext^*_{k[\A^\aleph]}(\pi',\rho')\simeq \Ext^*_{k[\Proj_{\Fp}^\aleph]}(I',\rho'(a\otimes -))\;.
\end{align*}
These adjunction isomorphisms are given by evaluation on $a\otimes -$ and restriction along the unit of adjunction $v\to \A^\aleph(a,a\otimes v)$ (see the beginning of section \ref{subsec-Kan-homology}). Thus they fit into a commutative square:
\begin{equation}
\begin{tikzcd}
\Ext^*_{\kATor^\aleph}(\pi',\rho')\otimes E^*_\infty \ar{r}{\Upsilon}\ar{d}{\simeq} & \Ext^*_{k[\A^\aleph]}(\pi',\rho')\ar{d}{\simeq}\\
\Ext^*_{{}_k\Proj_{\Fp}^\aleph}(I',\rho'(a\otimes -))\otimes E^*_\infty \ar{r}{\Upsilon} & \Ext^*_{k[\Proj_{\Fp}^\aleph]}(I',\rho'(a\otimes -))
\end{tikzcd}\;.\label{dgm-adj}
\end{equation}
Thus in order to prove lemma \ref{lm-red-Fp}, it suffices to prove that the lower $\Upsilon$ in diagram \eqref{dgm-adj} is an isomorphism.

But the lower $\Upsilon$ is an isomorphism because we assume that theorem \ref{thm-main-additive} holds for $\A=\Proj_\Fp$. Indeed, $I'=k\otimes_{\mathbb{Z}}\Hom_\Fp(\Fp,-)={}_k(\Proj_{\Fp}^\aleph)(\Fp,-)$, hence diagram \eqref{dgm-enlarge} with $\A$, $\pi'$, $\rho'$ respectively taken as $\Proj_\Fp$, $I'$, $\rho'(a\otimes -)$, shows that the lower $\Upsilon$ in diagram \eqref{dgm-adj} is an isomorphism.
\end{proof}
The next lemma finishes the proof of theorem \ref{thm-main-additive}.
\begin{lm}\label{lm-finish-proof}
Theorem \ref{thm-main-additive} holds when $\A=\Proj_\Fp$.
\end{lm}
\begin{proof}
If $\A=\Proj_\Fp$, there is an equivalence of categories
\[k\Md\simeq {}_k\A\Md\]
which sends a $k$-vector space $u$ to the functor $v\mapsto v\otimes_{\Fp}u$. Thus, the additive functors $\pi$ and $\rho$ are direct sums of copies of the functor $t:v\mapsto v\otimes_{\Fp}k$, hence it suffices to check that \eqref{eq-Upsilon} is an isomorphism when $\pi=\rho=t$. Moreover, we have \[\Ext^j_{\kA}(t,t)\simeq \Ext^j_k(k,k) = 
\begin{cases}
0 &\text{if $j>0$,}\\
k\id_t &\text{ if $j=0$.}
\end{cases}\] 
Hence, going back to the definition of the map \eqref{eq-Upsilon} it suffices to check that for $r$ big enough, the composition of the forgetful map and of the restriction along the functor $^{(-r)}t=t$: 
\[\Ext^i_{\Gamma^{p^r}\Proj_k}(I^{(r)},I^{(r)})\to \Ext^i_{k[\Proj_k]}(I^{(r)},I^{(r)})\to \Ext^i_{k[\Proj_\Fp]}(t,t)\]
is an isomorphism. 
The latter fact follows from the computations in \cite{FLS} and \cite{FS}. To be more specific, we have a commutative square 
\[\begin{tikzcd}
\Ext^i_{\Gamma^{p^r}\Proj_k}(I^{(r)},I^{(r)})\ar{r}& \Ext^i_{k[\Proj_k]}(I^{(r)},I^{(r)})\ar{r}& \Ext^i_{k[\Proj_\Fp]}(t,t)\\
k\otimes_\Fp\Ext^i_{\Gamma^{p^r}\Proj_\Fp}(I^{(r)},I^{(r)})\ar{u}{\simeq}\ar{rr}&& k\otimes_\Fp\Ext^i_{\Fp[\Proj_\Fp]}(I,I)\ar{u}{\simeq} 
\end{tikzcd}\]
where the bottom horizontal arrow is induced by the forgetful functor, the vertical isomorphism on the left is the base change isomorphism for strict polynomial functors \cite[2.7]{SFB} and the vertical isomorphism on the right is the base change morphism \eqref{base-change-Ext} from section \ref{subsec-size} (the latter is an iso since $I$ is $fp_\infty$ by \cite[Prop 10.1]{FLS}). Now the top horizontal line of the square is an isomorphism because the bottom arrow is an isomorphism, as we can see it by comparing the computations of \cite{FLS} and \cite{FS} or by using \cite[Thm 3.10]{FFSS}.
\end{proof}

\section{Polynomial homology over $\Proj_{\Fq}$}\label{sec-Pol-Fq}
The purpose of this section is to compute polynomial homology over the additive category $\Proj_\Fq$ in terms of the generic homology of strict polynomial functors recalled in section \ref{subsec-frobgen}. The main result is theorem \ref{thm-verystrong}, which can be seen as the special case $\A=\Proj_\Fq$ of the generalized comparison theorem established in section \ref{sec-generalized}. This special case is a key ingredient for the proof of the generalized comparison theorem.

The section is organized as follows. In the first three subsections we review the strong comparison theorem 
\cite[Thm 3.10]{FFSS} and we formulate it in a form which is better adapted to our purposes.  Then, in section \ref{subsec-strong-small}, we elaborate on the techniques of \cite{FFSS}, and we succeed in removing the assumption on the size of the field $\Fq$ in the strong comparison theorem. This leads us to theorem \ref{thm-verystrong}. 

Throughout the section, $k$ denotes a (non-necessarily infinite) perfect field of positive characteristic $p$ containing a finite subfield $\Fq$ with $q$ elements and we let 
\[t:\Proj_\Fq\to \Proj_k\] 
denote the additive functor given by extensions of scalars: $t(v)=k\otimes_\Fq v$.

\subsection{The strong comparison theorem}\label{subsec-strong-comparison}
Recall from lemma \ref{lm-evalinfty} the exact functor:
\[t^*:\Gamma^d\Proj_k\Md\to k[\Proj_\Fq]\Md\;.\]
induced by forgetful functor from strict polynomial functors to ordinary functors and by restriction along the base change functor $t:\Proj_\Fq\to \Proj_k$. If $q=p^r$, there are canonical isomorphisms $t^*I^{(nr)}\simeq t$ in $k[\Proj_\Fq]\Md$, which sends an element $\lambda\otimes x$ in  ${}^{(nr)}(k\otimes_\Fq v)$ to the element $\lambda^{p^{nr}}\otimes x$ in $k\otimes_\Fq v$. So if $F$ and $G$ are $d$-homogeneous strict polynomial functors, we have canonical isomorphisms: 
\begin{align*}
&t^*F\simeq t^*(F^{(nr)})\;, &&&t^*(G^{(nr)})\simeq t^*G\;.
\end{align*}
If $n$ is big enough, by combining these isomorphisms with the morphism of $\Ext$ induced by $t^*$ we obtain a graded $k$-linear map:
\begin{align}
\Ext^i_{\gen}(F,G)\simeq 
\Ext^i_{\Gamma^{dp^{nr}}\Proj_k}(F^{(nr)},G^{(nr)})
\to 
\Ext^i_{k[\Proj_\Fq]}(t^*F,t^*G)
\;.\label{eq-forget-comparison}
\end{align} 

The next result follows from the strong comparison theorem \cite[Thm 3.10]{FFSS}.

\begin{thm}\label{thm-strong-comparison}
Let $k$ be an infinite perfect field containing a finite subfield with $q$ elements, and let $F$ and $G$ be two $d$-homogeneous strict polynomial functors. If $q\ge d$, the map \eqref{eq-forget-comparison} is an isomorphism in all degrees $i$.
\end{thm}
\begin{proof}
Theorem \ref{thm-strong-comparison} slightly generalizes the strong comparison theorem of \cite{FFSS} in two ways. Firstly, contrarily to \cite{FFSS}, we do not assume that $k=\Fq$. Secondly, we work with $\Gamma^d\Proj_k\Md$ rather than with the category $\Gamma^d\Proj_k\md$, i.e. we allow our strict polynomial functors to have infinite-dimensional values. 

We overcome these two technical points as follows. The standard projective objects of $\Gamma^d\Proj_k\Md$ are the divided power functors $\Gamma^{d,s}=\Gamma^d(\Hom_k(k^s,-))$ and the standard injectives are the symmetric power functors $S^{d,s}=S^d(k^s\otimes -)$. These two kinds of functors commute with base change: there are canonical isomorphisms
\[t^*\Gamma^{d,s}(v)\simeq\Gamma^{d,s}_\Fq(v)\otimes_\Fq k \text{ and } t^*S^{d,s}(v)\simeq S^{d,s}_\Fq(v)\otimes_\Fq k\]
where the indices $\Fq$ indicate their counterparts in the category $\Gamma^d\Proj_\Fq\Md$ of strict polynomial functors over $\Fq$. Thus the morphism \eqref{eq-forget-comparison} fits into a commutative diagram
\[
\begin{tikzcd}
\Ext^i_{\Gamma^d\Proj_\Fq}(\Gamma^{d,s\,(nr)}_\Fq, S^{d,s\,(nr)}_\Fq )\otimes_\Fq k\ar{r}{\simeq}\ar{d}& \Ext^i_{\Gamma^d\Proj_k}(\Gamma^{d,s\,(nr)}, S^{d,s\,(nr)})\ar{d}{\eqref{eq-forget-comparison}}\\
\Ext^i_{\Fq[\Proj_\Fq]}(\Gamma^{d,s}_\Fq, S^{d,s}_\Fq )\otimes_\Fq k\ar{r}{\simeq}& \Ext^i_{k[\Proj_\Fq]}(t^*\Gamma^{d,s}, t^*S^{d,s})
\end{tikzcd}
\]
where the upper horizontal isomorphism is the base change functor for strict polynomial functors \cite[2.7]{SFB}, and the lower horizontal isomorphism is induced by tensoring by $k$ over $\Fq$ (that tensoring by $k$ yields an isomorphism follows from the K\"unneth formula of proposition \ref{prop-Kunneth-ext} and the fact that $\Gamma^{d,s}_\Fq$ is $\mathrm{fp}_\infty$ by Schwartz's $\mathrm{fp}_\infty$ lemma \cite[Prop 10.1]{FLS}), and the vertical morphism on the left hand side is induced by the forgetful functor from strict polynomial functors to ordinary functors. This latter morphism is an isomorphism if $n\gg 0$ by the strong comparison theorem \cite[Thm 3.10]{FFSS}, so that our morphism \eqref{eq-forget-comparison} is an isomorphism when $F$ is a standard projective and $G$ is a standard injective. 
Next, we observe that the source and the target of morphism \eqref{eq-forget-comparison} both turn sums into products when viewed as functors of $F$, and they both turn products into products when viewed as functors of $G$, which implies that morphism \eqref{eq-forget-comparison} is an isomorphism when $F$ is an arbitrary projective strict polynomial functor and $G$ is an arbitrary injective strict polynomial functors (possibly with infinite-dimensional values). This implies that \eqref{eq-forget-comparison} is an isomorphism when $F$ and $G$ are arbitrary objects of $\Gamma^d\Proj_k\Md$ by a standard spectral sequence argument.
\end{proof}

We have a similar situation with generic $\Tor$. Namely, if $F$ is a $d$-homogeneous contravariant strict polynomial functors, and if $G$ is a $d$-homogeneous strict polynomial functor, restriction along $t$ (as in lemma \ref{lm-evalinfty}) together with the isomorphisms $t^*F\simeq t^*F^{(nr)}$ and $t^*G\simeq t^*G^{(nr)}$ induce a morphism:
\begin{align}\Tor_i^{k[\Proj_\Fq]}(t^*F,t^*G)\to \Tor_i^{\Gamma^{dp^{nr}}\Proj_k}(F^{(nr)},G^{(nr)})\simeq \Tor_i^{\gen}(F,G)\;.\label{eq-forget-comparison-Tor}
\end{align}
The next corollary follows from theorem \ref{thm-strong-comparison} and proposition \ref{prop-Tor-Ext}.
\begin{cor}\label{cor-strong-thm}
If $q\ge d$, the map \eqref{eq-forget-comparison-Tor} is an isomorphism in all degrees $i$.
\end{cor} 

\subsection{Recollections of non-homogeneous strict polynomial functors}\label{subsec-nonhomogeneous}

Non-homogeneous strict polynomial functors are used in the generalizations of the strong comparison theorem \ref{thm-strong-comparison} that we give in sections \ref{subsec-strong-inhomogeneous} and \ref{subsec-strong-small}.
We abuse notations (see remark \ref{rk-abuse} below) and we denote by $\Gamma\Proj_k\Md$ the category of strict polynomial functors of bounded degree over a field $k$. 
This category is defined by 
\[\Gamma\Proj_k\Md=\bigoplus_{d\ge 0} \Gamma^d\Proj_k\Md\;.\]
Thus a strict polynomial functor of bounded degree $F$ is simply defined as a family of $d$-homogeneous strict polynomial functors $F_d$, which are called the $d$-homogeneous components of $F$, and all the $F_d$ are zero but a finite number of them. 
The highest $d$ such that $F_d\ne 0$ is called the \emph{degree of $F$}, and denoted by $\stdeg F$. 
We have $F=\bigoplus_{d\ge 0}F_d$. Morphisms of strict polynomial functors preserve these decompositions into homogeneous components. More generally we have (only  finitely many terms of the sum are nonzero):
\[\Ext^*_{\Gamma\Proj_k}(F,G)=\bigoplus_{d\ge 0}\Ext^*_{\Gamma^d\Proj_k}(F_d,G_d)\;.\]
We define generic extensions by:
\[\Ext^*_\gen(F,G)=\bigoplus_{d\ge 0}\Ext^*_\gen(F_d,G_d)\;.\]
\begin{rk}
Let $\Gamma\Proj_k\md$ denote the full subcategory of $\Gamma\Proj_k\Md$ on the functors $F$ such that $F(v)=\bigoplus_{d\ge 0}F(v)$ has finite dimension for all $v$. Then  $\Gamma\Proj_k\md$ identifies with the category $\Pp_k$ introduced in \cite{FS}. The inclusion $\Gamma\Proj_k\md\hookrightarrow \Gamma\Proj_k\Md$ induces an isomorphism on $\Ext$, so that working with the former category or the latter is largely a matter of taste.
\end{rk}

The forgetful functor from homogeneous strict polynomial functors to ordinary functors described in section \ref{subsec-restriction} extends to the non-homogeneous case. Namely, we have a forgetful functor:
\[\gamma^*:\Gamma\Proj_k\Md = \bigoplus_{d\ge 0}\Gamma^d\Proj_k\Md\xrightarrow[]{\sum\gamma^{d\,*}}k[\Proj_k]\Md\;.\]
If $k$ is an infinite field, this forgetful functor is fully faithful, and 
an ordinary functor $F$ with finite-dimensional values is the image of a strict polynomial functor of degree $d$ if and only if the coordinate functions of the maps
$\Hom_k(v,w)\to \Hom_k(F(v),F(w))$, $f\mapsto F(f)$, are polynomials of degree $d$.

Most often we will omit $\gamma^*$ from the notations, and simply denote by $F$ the underlying ordinary functor of a strict polynomial functor $F$. 

\begin{rk}\label{rk-ordvsstrict}
The underlying ordinary functor of a strict polynomial functor of bounded degree $F$ is always polynomial in the sense of Eilenberg and Mac Lane, used in section \ref{sec-AP}. Thus $F$ has a degree $\stdeg F$ and an Eilenberg-Mac Lane degree $\deg_\EML F$.  We have $\stdeg F\ge \deg_\EML F$, but the inequality may be strict. For example $\stdeg I^{(r)}=p^r$ and $\deg_\EML I^{(r)}=1$. More detailed relations between these two notions of degree can be found in \cite{TouzeFund}. 
\end{rk}

Similarly there is a category $\Mdd\Gamma\Proj_k$ of contravariant strict polynomial functors of bounded degree and we have a similar decompositions (with finitely many nonzero terms in the direct sum)
 \[\Tor_*^{\Gamma\Proj_k}(F,G)=\bigoplus_{d\ge 0}\Tor_*^{\Gamma^d\Proj_k}(F_d,G_d)\;,\quad \Tor_*^{\gen}(F,G)=\bigoplus_{d\ge 0}\Tor_*^{\gen}(F_d,G_d)\;.\]
\begin{rk}\label{rk-abuse}
Despite its notation, the category $\Gamma\Proj_k\Md$ is not a category of $k$-linear functors from some category $\Gamma\Proj_k$ to $k$-modules. However this abuse of notation emphasizes the fact the properties of the category $\Gamma\Proj_k\Md$ are very close to those of the categories $\Gamma^d\Proj_k\Md$. It also allows compact notations for $\Ext$ and $\Tor$, with the fictious category $\Gamma\Proj_k$ as a decoration.
\end{rk}

\subsection{Strong comparison without homogeneity}\label{subsec-strong-inhomogeneous}
If $F$ and $G$ are two strict polynomial functors of bounded degrees, we define a comparison map  
\begin{align}
\begin{array}[b]{c}
\Ext^*_{\gen}(F,G)=\\
\displaystyle\bigoplus_{d\ge 0}\Ext^*_{\gen}(F_d,G_d)
\end{array}
\to 
\bigoplus_{d\ge 0}\Ext^*_{k[\Proj_\Fq]}(t^*F_d,t^*G_d)\to 
\Ext^*_{k[\Proj_\Fq]}(t^*F,t^*G)
\label{eqn-strong-inhom}
\end{align}
where the map on the left hand side is the  direct sum of the comparison maps \eqref{eq-forget-comparison} used in theorem \ref{thm-strong-comparison}
while the map on the right hand side is the canonical inclusion into  
\[\Ext^*_{k[\Proj_k]}(t^*F,t^*G)=\bigoplus_{d,e\ge 0}\Ext^*_{k[\Proj_k]}(t^*F_d,t^*G_e)\;.\]
We will often refer to morphism \eqref{eqn-strong-inhom} as the \emph{strong comparison map}. The next result extends the strong comparison theorem \ref{thm-strong-comparison} to the non-homogeneous case.
\begin{thm}\label{thm-strong-inhom}
Let $k$ be an infinite perfect field containing a finite subfield with $q$ elements, and let $F$ and $G$ be two strict polynomial functors, with degrees less or equal to $q$. Then the map \eqref{eqn-strong-inhom} is an isomorphism. 
\end{thm}
\begin{proof}
The first map of the composition \eqref{eqn-strong-inhom} is an isomorphism by the strong comparison theorem \ref{thm-strong-comparison}. Thus it suffices to prove that $\Ext^*_{k[\Proj_k]}(t^*F_d,t^*G_e)=0$ as soon as $d\ne e$, which follows from the vanishing lemma \ref{lm-homogeneity-vanish} below.
\end{proof}

We shall explain the elementary vanishing result used in theorem \ref{thm-strong-inhom} in a general context, in order to use it again later.
Let $\A$ be a small additive category. We assume that $\A$ is $\FF$-linear, over a subfield $\FF\subset k$. In the next lemma, we say that a functor $F$ of $k[\A]\Md$ is \emph{$d$-homogeneous} (with respect to the field $\FF$) if $F(\lambda f)=\lambda^dF(f)$ for all morphisms $f$ in $\A$ and all $\lambda\in\FF$. 

\begin{lm}\label{lm-homogeneity-vanish}
Let $\FF$ be a subfield of $k$, and let $d\ne e$ be two non-negative integers such that $\mathrm{card}\,\FF\ge d,e$. If $F$ and $G$ are two objects of $k[\A]\Md$ which are respectively $d$-homogeneous and $e$-homogeneous, then $\Ext^*_{k[\A]}(F,G)=0$.
\end{lm}
\begin{proof}
Let $F$ and $G$ be two arbitrary objects of $k[\A]\Md$. 
Since $\A$ is $\FF$-linear, every element of $\FF$ yields a natural transformation $\lambda_F\in \End_{k[\A]}(F)$ whose component at $x$ equals $F(\lambda \id_x)$. Thus $\Ext^*_{k[\A]}(F,G)$ has an $\FF$-$\FF$-bimodule structure given by $\lambda\cdot [\xi]\cdot \mu = [\mu_G\circ \xi\circ \lambda_F]$, where $-\circ\lambda_F$ is the pullback of an extension along $\lambda_F$ and $\lambda_G\circ-$ is the pushout of an extension along $\mu_G$. Moreover, for all morphisms $f:H\to K$ in $k[\A]\Md$ we have $f\circ \lambda_H=\mu_G\circ f$, which implies that the two $\FF$-module structures coincide: $\lambda_F\cdot [\xi]=[\xi]\cdot \lambda_G$.
Assume now that $F$ is $d$ homogeneous and $G$ is $e$-homogeneous. Then $\lambda_F=\lambda^d\id_F$ and $\lambda_G=\lambda^e\id_G$. Thus for all extensions $[\xi]$ we have $\lambda^d[\xi]=\lambda_F\cdot [\xi]=[\xi]\cdot\lambda_G=\lambda^e[\xi]$.
Since the cardinal of $\FF$ is greater or equal to $d$ and $e$, the maps $\lambda\mapsto \lambda^d$ and $\lambda\mapsto \lambda^e$, seen as maps from $\FF$ to $k$, are not equal. Hence the equality $\lambda^d[\xi]=\lambda^e[\xi]$ implies that $[\xi]=0$.
\end{proof}

As before, this result can be dualized. Namely, if
$F$ and $G$ are two strict polynomial functors of bounded degrees, respectively contravariant and covariant, we have a comparison map (where the first map is the canonical projection and the second map is given by the direct sum of the comparison maps \eqref{eq-forget-comparison-Tor})
\begin{align}
\Tor_*^{k[\Proj_k]}(t^*F,t^*G)\to 
\bigoplus_{d\ge 0}\Tor_*^{k[\Proj_k]}(t^*F_d,t^*G_d)\to 
\begin{array}{c}
\Tor^\gen_*(F,G)=\\
\displaystyle\bigoplus_{d\ge 0}\Tor^\gen_*(F_d,G_d)
\end{array}\;.
\label{eqn-strong-inhom-Tor}
\end{align}
We will often refer to morphism \eqref{eqn-strong-inhom-Tor} as the \emph{strong comparison map (for $\Tor$)}. 
By using proposition \ref{prop-Tor-Ext} we deduce the following result from theorem \ref{thm-strong-inhom}. 
\begin{cor}\label{cor-strong-inhom-Tor}
Let $k$ be an infinite perfect field containing a finite subfield with $q$ elements, and let $F$ and $G$ be two strict polynomial functors, with degrees less or equal to $q$. Then the map \eqref{eqn-strong-inhom-Tor} is an isomorphism. 
\end{cor}

\subsection{Strong comparison over small fields}\label{subsec-strong-small}
We are now going to generalize the strong comparison theorem \ref{thm-strong-inhom} to the case when $q=p^r$ is not big enough with respect to the degrees of $F$ and $G$. Our approach, in particular lemma \ref{lm-chgt-base-finitefield} and proposition \ref{pr-interm-strong}, is inspired by the proof of \cite[Thm 6.1]{FFSS}.
\begin{nota}\label{nota-twist1}
Let $L$ be a perfect field. For all positive integers $a$ and $s$ and for all $L$-vector spaces $v$, we let $^{(a|s)}v={}^{(0)}v\oplus {}^{(a)}v\oplus\cdots \oplus {}^{(\,(s-1)a\,)}v$.
For all functors $F$ in $k[\Proj_L]\Md$, we denote by $F^{(a|b)}$ the composition of $F$ with the functor $^{(a|s)}-:\Proj_L\to \Proj_L$.
\end{nota}

Assume that $\Fq$ is a finite field with $q=p^r$ elements and that $\Fq\subset L$ is an extension of perfect fields. Let $\tau:\Proj_\Fq\to \Proj_L$ be the associated extension of scalars. 
For all positive integers $s$ we define two morphisms in $k[\Proj_\Fq]\Md$ :
\begin{align}
&\tau^*F\to \tau^*(F^{(r|s)})\label{eq-map1}\\
&\tau^*(G^{(rs|s)})\to \tau^*G\label{eq-map2}
\end{align}
as follows.
Firstly, if $a$ is divisible by $r$ then for all integers $i$ and for all $\Fq$-vector spaces $v$ there is a canonical isomorphism of $L$-vector spaces $^{(ia)}\tau(v)\simeq \tau(v)$ which sends an element $\lambda\otimes x\in {}^{(ia)}\tau(v)$ to the element $\lambda^{p^{ai}}\otimes x\in \tau(v)$. By taking the direct sum of these isomorphisms, we obtain a canonical isomorphism 
$^{(a|s)}\tau(v)\simeq \tau(v)^{\oplus s}$.
Secondly, we let $\mathrm{diag}:\tau\to \tau^{\oplus s}$ and $\mathrm{sum}: \tau^{\oplus s}\to \tau$ denote the morphisms whose restrictions to the components of $\tau^{\oplus s}$ are all equal to the identity of $\tau$. Then we define the morphisms \eqref{eq-map1} and \eqref{eq-map2} as the compositions:
\begin{align*}
&\tau^*F\xrightarrow[]{F(\diag)} (\tau^{\oplus s})^*F\simeq \tau^*(F^{(r|s)})\;,&&&
&\tau^*(G^{(rs|s)})\simeq (\tau^{\oplus s})^*G \xrightarrow[]{G(\summ)} \tau^*G\;.
\end{align*}

Now, restriction along $\tau:\Proj_\Fq\to \Proj_L$ and naturality with respect to the morphisms \eqref{eq-map1} and \eqref{eq-map2} yield a graded $k$-linear map:
\begin{equation}
\Ext^*_{k[\Proj_L]}(F^{(r|s)},G^{(rs|s)})\to \Ext^*_{k[\Proj_\Fq]}(\tau^*F,\tau^*G)\;.\label{eqn-interm-iso}
\end{equation} 
\begin{lm}\label{lm-chgt-base-finitefield}
If $\Fq\subset L$ is an extension of fields of degree $s^2$ and $q=p^r$, then for all $F$ and $G$ in $k[\Proj_L]\Md$, the map \eqref{eqn-interm-iso} is an isomorphism.
\end{lm}
\begin{proof}
The strategy of the proof is as follows. Since $\Fq\subset L$ is an extension of degree $s^2$, we can find an intermediate field $K$ such that $\Fq\subset K\subset L$ is a sequence of extensions of fields of degree $s$. We are going to convert $\Ext$ over $k[\Proj_L]$ into $\Ext$ over $k[\Proj_\Fq]$ in two steps, by using the effect on $\Ext$ of the restrictions of scalars $\Proj_L\to \Proj_K$ and $\Proj_K\to \Proj_\Fq$, and their adjoints (with the help of proposition \ref{pr-iso-Ext-explicit}), and we are going to check that the two-steps isomorphism obtained coincides with the explicit map \eqref{eqn-interm-iso}. 

In the first step, we express extensions over $k[\Proj_L]$ in terms of extensions over $k[\Proj_K]$. For this purpose, we consider the adjoint pair \cite[Prop 3.1]{FFSS}
$$\rho':\Proj_L\leftrightarrows \Proj_K:\tau'$$ 
where $\tau'$ is the extension of scalars and $\rho'$ the restriction of scalars associated to the extension $K\subset L$. Of course, $\tau'=L\otimes_K-$ is left  adjoint of $\rho'$, but we consider it here as the \emph{right} adjoint (this is possible because the extension $K\subset L$ has finite degree). By proposition \ref{pr-iso-Ext-explicit} the adjoint pair $(\rho',\tau')$ induces an isomorphism for all $H$ in $k[\Proj_L]\Md$ 
\begin{align}
\Ext^*_{k[\Proj_L]}(H,{\rho'}^*{\tau'}^*G)\simeq \Ext^*_{k[\Proj_K]}({\tau'}^*H,{\tau'}^*G)\;.\label{eqn-isoproof2}
\end{align}
Moreover, there is an isomorphism of $L$-vector spaces $\phi_v:L\otimes_K v\simeq \bigoplus_{0\le i<s} {}^{(irs)}v$ natural with respect to the $L$-vector space $v$. This isomorphism is given by sending $\lambda\otimes x$ to $\sum_{0\le i<s}\lambda^{p^{-rsi}}x$. Therefore we have an isomorphism in $k[\Proj_L]\Md$:
\begin{align}
G(\phi^{-1}): G^{(rs|s)}\simeq {\rho'}^*{\tau'}^*G\;.
\label{eqn-isoproof1}
\end{align}
By combining the isomorphisms \eqref{eqn-isoproof2} and \eqref{eqn-isoproof1} we obtain an isomorphism
\begin{align}
\Ext^*_{k[\Proj_L]}(H,G^{(rs|s)})\xrightarrow[]{\simeq}
\Ext^*_{k[\Proj_K]}({\tau'}^*H,{\tau'}^*G)\;.
\label{eqn-isoproof12}
\end{align}
To finish this first step, we are going to give a more explicit expression of the isomorphism \eqref{eqn-isoproof12}. Recall from proposition \ref{pr-iso-Ext-explicit} that the isomorphism \eqref{eqn-isoproof2} is induced by restriction along $\tau'$ and by the map $({\tau'}^*G)(\epsilon)$ where $\epsilon$ is the counit of the adjunction $\rho'\dashv\tau'$. By \cite[Prop 3.1]{FFSS}, this counit of adjunction $\epsilon_u: L\otimes_K u\to u$ is given by $\epsilon_u(\lambda\otimes x)= T(\lambda)x$, where $T(\lambda)=\sum_{0\le i<s}\lambda^{p^{rsi}}$ is the trace of $\lambda$. Thus for all $K$-vector spaces $u$ we have a commutative square of $L$-vector spaces, in which the upper horizontal arrow is the canonical isomorphism:
\[
\begin{tikzcd}
\bigoplus_{0\le i<s}{}^{(-rsi)}\tau'(u)\ar{r}{\simeq} & 
\bigoplus_{0\le i<s}\tau'(u)\ar{d}{\summ}\\
\tau'(\rho'(\tau'(u)))\ar{u}{\phi_{\tau'(u)}}\ar{r}{\tau'(\epsilon_u)}&
\tau'(u)
\end{tikzcd}\;.
\]
It follows that the isomorphism \eqref{eqn-isoproof12} equals the following composition: 
\[\Ext^*_{k[\Proj_L]}(H,G^{(rs|s)})\xrightarrow[]{\tau'{}^*} \Ext^*_{k[\Proj_K]}(\tau'{}^*H,\tau'{}^*(G^{(rs|s)}))\to 
\Ext^*_{k[\Proj_K]}(\tau'{}^*H,\tau'{}^*G)\]
where the last map is induced by the morphism ${\tau'}^*(G^{(rs|s)})\simeq ({\tau'}^{\oplus s})^*G\xrightarrow[]{G(\summ)}G$.

In the second step, we express extensions over $k[\Proj_K]$ in terms of extensions over $k[\Proj_\Fq]$. For this purpose, we consider a pair of adjoints, in which $\tau''$ and $\rho''$ are the extension of scalars and the restriction of scalars associated to the extension $\Fq\subset K$, and $\tau''$ is this time seen as a left adjoint: 
$$\tau'':\Proj_\Fq\leftrightarrows \Proj_K:\rho''\;.$$ 
Since $\rho''$ is right adjoint to $\tau''$, proposition \ref{pr-iso-Ext-explicit} yields an isomorphism for all $K$ in $k[\Proj_K]\Md$:
\begin{align}
\Ext^*_{k[\Proj_K]}({\rho''}^* {\tau}^*F,K)\simeq \Ext^*_{k[\Proj_\Fq]}(\tau^*F,{\tau''}^*K)\;.
\label{eqn-isoproof4}
\end{align}
Moreover, the isomorphism of $L$-vector spaces $\psi_v:L\otimes_\Fq v\simeq \bigoplus_{0\le i<s} {}^{(ri)}(L\otimes_K v)$ given by $\psi_v(\lambda \otimes x)=\sum_{0\le i<s}\lambda^{p^{-ri}}x$ is natural with respect to the $K$-vector space $v$, hence it induces an isomorphism in $k[\Proj_K]\Md$:
\begin{align}
F(\psi): {\rho''}^* {\tau}^*F \simeq {\tau'}^*(F^{(r|s)}) \;.
\label{eqn-isoproof3}
\end{align}
Combining the isomorphisms \eqref{eqn-isoproof4} and \eqref{eqn-isoproof3}, we obtain an isomorphism:
\begin{align}
\Ext^*_{k[\Proj_K]}({\tau'}^*(F^{(r|s)}),K)\xrightarrow[]{\simeq}\Ext^*_{k[\Proj_\Fq]}(\tau^*F,{\tau''}^*K)\;.
\label{eqn-isoproof34}
\end{align}
To finish the second step or the proof, we give an explicit expression of the isomorphism \eqref{eqn-isoproof34}. Recall from proposition \ref{pr-iso-Ext-explicit} that the isomorphism \eqref{eqn-isoproof2} is induced by restriction along $\tau''$ and by $(\tau^*F)(\eta)$, where $\eta$ is the unit of the adjunction $\tau''\dashv\rho''$. This unit of adjunction $\eta_v: v\to K\otimes_\Fq v$ is given by $\eta_v(x)=1\otimes x$, whence a commutative diagram of $L$-vector spaces, in which the lower horizontal arrow is the canonical isomorphism:
\[
\begin{tikzcd}
\tau(\rho''(\tau''(u)))\ar{d}{\psi_{\tau''(u)}}& \tau(u)\ar{l}[swap]{\tau(\eta_u)}\ar{d}{\diag}\\
\bigoplus_{0\le i<s}{}^{(ri)}\tau'(\tau''(u))& \bigoplus_{0\le i<s}\tau(u)\ar{l}[swap]{\simeq}
\end{tikzcd}\;.
\] 
This shows that the isomorphism \eqref{eqn-isoproof34} is equal to the composition 
\begin{align*}
\Ext^*_{k[\Proj_K]}({\tau'}^*(F^{(r|s)}),K)\xrightarrow[]{\tau''{}^*} \Ext^*_{k[\Proj_\Fq]}({\tau}^*(F^{(r|s)}),\tau''{}^*K)\to \Ext^*_{k[\Proj_\Fq]}(\tau^*F,{\tau''}^*K)
\end{align*}
with last map induced by the morphism $\tau^*F\xrightarrow[]{F(\diag)} (\tau^{\oplus s})^*F\simeq {\tau''}^*{\tau'}^*(F^{(r|s)})$.

Thus, the graded morphism \eqref{eqn-interm-iso} is the composition of the maps \eqref{eqn-isoproof12} (with $H=F^{(r,s)}$) and  \eqref{eqn-isoproof34} (with $K={\tau'}^*G$), hence it is an isomorphism.
\end{proof}

\begin{nota}\label{nota-twist2} We extend notation \ref{nota-twist1} to strict polynomial functors. 
For all positive integers $r$ and $s$, we denote by $I^{(r|s)}$ the (non-homogeneous) strict polynomial functor of degree $p^{(s-1)r}$ defined by:
\begin{align*}
I^{(r|s)}:=I^{(0)}\oplus I^{(r)}\oplus \cdots\oplus I^{(\,(s-1)r\,)}.
\end{align*}
For all strict polynomial functors $F$ we let $F^{(r|s)}=F\circ I^{(r|s)}$. 
If $\stdeg F=d$ then $\stdeg(F^{(r|s)})=p^{(s-1)r}d$.
\end{nota}
\begin{rk}
The definition of composition of strict polynomial functors is the obvious one if we think of strict polynomial functors in the way they are defined in \cite{FS}. If we use the description of strict polynomial functors as families of $k$-linear functors as we pretend to do it, then composition can be defined as follows. First we can consider $F(v_0\oplus \cdots \oplus v_{s-1})$ as a strict polynomial functor of $s$ variables as in \cite[Section 3.2]{TouzeFund}. Then we precompose each variable $v_i$ by the Frobenius twist $I^{(ir)}$. The strict polynomial functor $F^{(r|s)}$ is then defined as the evaluation of the resulting strict polynomial functor with $s$ variables on the $s$-tuple $(v,\dots,v)$.
\end{rk}

For all strict polynomial functors $F$ and $G$ we define a 
morphism $\Xi_k$ of graded $k$-vector spaces as the composition of the strong comparison map \eqref{eqn-strong-inhom} together with map induced by the morphisms $t^*F\to t^*(F^{(r|s)})$ and $t^*(G^{(rs|s)})\to t^*G$ constructed in \eqref{eq-map1} and \eqref{eq-map2} with $L$ and $\tau$ replaced by $k$ and $t$:
\begin{equation}
\begin{tikzcd}
\Ext^*_\gen(F^{(r|s)},G^{(rs|s)})\ar{r}{\eqref{eqn-strong-inhom}}\ar[dashed]{rd}[swap]{\Xi_k}&\Ext^*_{k[\Proj_\Fq]}(t^*(F^{(r|s)}),t^*(G^{(rs|s)}))\ar{d}\\
& \Ext^*_{k[\Proj_\Fq]}(t^*F,t^*G)
\end{tikzcd}\;.
\label{eq-comparison-q-petit}
\end{equation} 
\begin{pr}\label{pr-interm-strong}
Assume that $k$ contains a finite field $\Fq$ of cardinal $q=p^r$, and let $s$ be a positive integer. Then for all strict polynomial functors $F$ and $G$ of degrees less or equal to $q^s$, the map \eqref{eq-comparison-q-petit} is an isomorphism in all degrees.
\end{pr}
\begin{proof}
We first claim that it suffices to prove the result when $k$ contains a subfield $L$ with $q^{s^2}$ elements. Indeed, let $k\to K$ be a finite extension of fields and let $\tau:k\Md\to K\Md$ be the extension of scalars. By \cite[Section 2]{SFB} there is an exact $k$-linear base change functor 
\[-_K:\Gamma\Proj_k\Md \to \Gamma\Proj_K\Md\] 
such that for all strict polynomial functors $F'$ over $k$ there are canonical isomorphisms of functors $\tau^*F'_K\simeq \tau\circ F'$. Moreover this base change functor induces an isomorphism on the level of $\Ext$ (See \cite[cor 2.7]{SFB} for the case of functors with finite-dimensional values. The proof extends to arbitrary functors when $k\to K$ is a finite extension of fields). 
There is a commutative square 
\[
\begin{tikzcd}[column sep = large]
\Ext^*_\gen(F^{(r|s)}_K,G^{(rs|s)}_K)\ar{r}{\Xi_K}&\Ext^*_K[\Proj_\Fq](t^*\tau^*F_K,t^*\tau^*G_K)\\
K\otimes\Ext^*_\gen(F^{(r|s)},G^{(rs|s)})\ar{r}{K\otimes\Xi_k}\ar{u}{\simeq}&K\otimes \Ext^*_K[\Proj_\Fq](t^*F,t^*G)\ar{u}{\simeq}
\end{tikzcd}
\]
in which the vertical isomorphism on the left hand side is induced by the base change functor $-_K$ and the vertical isomorphism on the right hand side is induced by $\tau$ (see the map \eqref{base-change-Ext} in section \ref{sec-prelim-polyn-hom}) and by the isomorphisms $H_K\circ \tau\circ t\simeq \tau\circ H\circ t$, for $H=F$ or $G$. Therefore $\Xi_k$ is an isomorphism if an only if $\Xi_K$ is an isomorphism. Thus, up to replacing $k$ by a finite extension $K$, we may assume that our field $k$ contains a subfield $L$ with $q^{s^2}$ elements.

We denote by $t':\Proj_L\to \Proj_k$ the extension of scalars associated to the extension of fields $L\subset k$. Then $\Xi_k$ is an isomorphism because we can rewrite it as the composition of three isomorphisms:
\[
\begin{tikzcd}[column sep = large]
\Ext^*_\gen(F^{(r|s)},G^{(rs|s)})\ar{r}{\Xi_k}\ar{d}{\simeq}& \Ext^*_{k[\Proj_\Fq]}(t^*F,t^*G)\\
\Ext^*_{k[\Proj_L]}({t'}^*(F^{(r|s)}),{t'}^*(G^{(rs|s)}))\ar{r}{\simeq}&\Ext^*_{k[\Proj_L]}(({t'}^*F)^{(r|s)},({t'}^*G)^{(rs|s)})\ar{u}{\simeq}
\end{tikzcd}\;.
\]
To be more specific, the vertical map on the left hand side is the strong comparison map \eqref{eqn-strong-inhom} relative to the finite field $L$. By our assumptions on $s$, the degrees of $F^{(r|s)}$ and $G^{(rs|s)}$ are less or equal to the cardinal of $L$, hence this map is an isomorphism by theorem \ref{thm-strong-inhom}. To define the lower horizontal map, we first observe that for all integers $i$ there is a canonical isomorphism $^{(ir)}t'(v)\simeq t'({}^{(ir)}v)$ which sends an element $\lambda\otimes x \in {}^{(ir)}(k\otimes_L v)$ to the element $\lambda^{p^{ir}}\otimes x\in k\otimes_L {}^{(ir)}v$. These canonical isomorphisms induce isomorphisms of functors ${t'}^*(F^{(r|s)})\simeq ({t'}^*F)^{(r|s)}$ and ${t'}^*(G^{(rs|s)})\simeq ({t'}^*G)^{(rs|s)}$, and the lower horizontal map is induced by these isomorphisms. Finally, the vertical map on the right hand side is the isomorphism provided by lemma \ref{lm-chgt-base-finitefield}.
\end{proof}

Now we introduce a variant of the map $\Xi_k$ which will be better adapted to our later purposes. With this new map $\Xi_k'$, direct sums of Frobenius twists only appear inside $F$ in the generic extensions. To be more specific, we let $\Xi_k'$ be the composition
of the strong comparison map \eqref{eqn-strong-inhom} and of the map induced by the morphism $t^*F\to F^{(r|s^2)}$ as in \eqref{eq-map1} (with $L$, $\tau$ and $s$ replaced by $k$, $t$ and $s^2$) and by the isomorphism of functors $t^*(G^{(rs^2-rs)})\simeq t^*G$ induced by the canonical isomorphism $^{(rs^2-rs)}t(v)\simeq t(v)$:
\begin{equation}
\begin{tikzcd}
\Ext^*_\gen(F^{(r|s^2)},G^{(rs^2-rs)})\ar{r}{\eqref{eqn-strong-inhom}}\ar[dashed]{rd}[swap]{\Xi_k'}&\Ext^*_{k[\Proj_\Fq]}(t^*(F^{(r|s^2)}),t^*(G^{(rs^2-rs)}))\ar{d}\\
& \Ext^*_{k[\Proj_\Fq]}(t^*F,t^*G)
\end{tikzcd}\;.
\label{eqn-verystrong}
\end{equation} 
Next theorem subsumes the strong comparison theorem \ref{thm-strong-inhom}. To be more specific, on recovers theorem \ref{thm-strong-inhom} by taking $s=1$ in the statement. 
\begin{thm}\label{thm-verystrong}
Let $k$ be a perfect field containing a finite subfield with $q=p^r$ elements, and let $s$ be a positive integer. Assume that $F$ and $G$ are two strict polynomial functors with degrees less or equal to $q^s$. Then the map \eqref{eqn-verystrong} is a graded isomorphism.
\end{thm}
\begin{proof}
We shall use strict polynomial multifunctors, as in \cite[Section 3]{SFB}, \cite[Section 2]{TouzeClassical} or \cite[Section 3.2]{TouzeFund}. To be more specific, we consider the category of strict polynomial multifunctors of $n$ variables:
\[\Gamma(\Proj_k^{\times n})\Md = \bigoplus_{d\ge 0}\Gamma^d (\Proj_k^{\times n})\Md\;.\]

The operation of precomposition by Frobenius twist extends to the multivariable setting, namely given a strict polynomial multifunctors $F$ and an $n$-tuple of non-negative integers $\underline{r}=(r_1,\dots,r_n)$ we let $F^{(\underline{r})}$ denote the strict polynomial multifunctor such that 
$$F^{(\underline{r})}(v_1,\dots,v_n)=F({}^{(r_1)}v_1,\dots,{}^{(r_n)}v_n)\;.$$
Precomposition by Frobenius twists yield a morphism on $\Ext$:
$$-\circ I^{(\underline{m})}:\Ext^i_{\Gamma(\Proj_k^{\times n})}(F^{(\underline{r})}, G^{(\underline{r})})\to 
\Ext^i_{\Gamma(\Proj_k^{\times n})}(F^{(\underline{r}+\underline{m})}, G^{(\underline{r}+\underline{m})})$$
which is an isomorphism provided that all the integers $r_i$ are big enough (with respect to $i$, $F$ and $G$). Indeed, by a spectral sequence argument, it suffices to check the result when $F$ is a standard projective and $G$ is a standard injective. In this case, 
$F(v_1,\dots,v_n)=F_1(v_1)\otimes\cdots\otimes F_n(v_n)$ for some standard projective strict polynomial functors $F_i$, and $G(v_1,\dots,v_n)=G_1(v_1)\otimes\cdots\otimes G_n(v_n)$ for some standard injective strict polynomial functors $G_i$, hence the isomorphism follows from the $\Ext$-isomorphism for functors with one variable and the K\"unneth formula. 

There is a forgetful functor $\gamma^*:\Gamma(\Proj_k^{\times n})\Md\to k[\Proj_k^{\times n}]\Md$ and the sum-diagonal adjunction lifts to the setting of strict polynomial functors.

Now in order to prove theorem \ref{thm-verystrong}, we observe that we may choose strict polynomial multifunctors $F'$, $G'$, $F''$ and $G''$ such that there is a commutative diagram, with $n\gg 0$:
\[
\begin{tikzcd}
\Ext^{i}_{\Gamma \Proj_k^{s^2}}(F',G')\ar{r}{\simeq}[swap]{(*)}& \Ext^i_{\Gamma\Proj_k}\Big((F^{(r|s^2)})^{(nr)},(G^{(rs-r)})^{(nr)}\Big)\ar{dd}{\Xi'_k}\\
\Ext^{i}_{\Gamma \Proj_k^{s^2}} (F'',G'')\ar{d}{\simeq}[swap]{(**)}\ar{u}[swap]{\simeq}{-\circ I^{(\underline{m})}}&\\
\Ext^i_{\Gamma \Proj_k}\Big((F^{(r|s)})^{(nr)},(G^{(rs|s)})^{(nr)}\Big)\ar{r}{\Xi_k} & \Ext^i_{k[\Proj_\Fq]}(t^*F,t^*G)
\end{tikzcd}
\]
To be more specific, the strict polynomial multifunctors $F'$, $G'$, $F''$ and $G''$ of the $s^2$ variables $v_{ij}$, $0\le i,j< s$, are respectively given by 
\begin{align*}
&F'(\dots,v_{ij},\dots)=F\Big(\bigoplus_{0\le i,j< s} {}^{(nr+ri+rsj)}v_{ij}\Big),\\
&G'(\dots,v_{ij},\dots)=G\Big(\bigoplus_{0\le i,j< s} {}^{(nr+rs^2-rs)}v_{ij}\Big),\\
&F''(\dots,v_{ij},\dots)=F\Big(\bigoplus_{0\le i,j< s} {}^{(nr+ri)}v_{ij}\Big),\\
&G''(\dots,v_{ij},\dots)=G\Big(\bigoplus_{0\le i,j< s} {}^{(nr+rsj)}v_{ij}\Big).
\end{align*}
The $s^2$-tuple $\underline{m}$ is given by $m_{ij}=rs^2-rs-rsj$ and $-\circ I^{(\underline{m})}$ is an isomorphism because $n$ is big enough. The maps $(*)$ and $(**)$ are given by sum-diagonal adjunction. To be more explicit, the map $(*)$ is given by setting $v_{ij}=v$ for all $i$ and $j$, and by composing the resulting extensions of strict polynomial functors of the variable $v$ by the morphism $G(\summ')$, where 
$$\summ': {}^{(rs^2-rs+nr)}v^{\oplus s^2}\to {}^{(rs^2-rs+nr)}v$$ 
is the morphism which restricts to the identity of ${}^{(rs^2-rs+nr)}v$ on each summand of ${}^{(rs^2-rs+nr)}v^{\oplus s^2}$. Similarly, the map $(**)$ is given by setting $v_{ij}=v$ for all $i$ and $j$, and by composing the resulting extensions of strict polynomial functors of the variable $v$ by the morphisms $F(\diag'')$ and $G(\summ'')$ where each of the morphisms
\begin{align*}
&\diag'':\bigoplus_{0\le i<s}{}^{(ri+nr)}v\to\bigoplus_{0\le i<s}{}^{(ri+nr)}v^{\oplus s}   \\
&\summ'':\bigoplus_{0\le j<s}{}^{(rsj+nr)}v^{\oplus s}\to  \bigoplus_{0\le j<s}{}^{(rsj+nr)}v
\end{align*}
restricts to identity morphisms between any two summands with the same number of Frobenius twists.

Under the hypotheses of theorem \ref{thm-verystrong} the map $\Xi_k$ is an isomorphism by proposition \ref{pr-interm-strong}, hence $\Xi'_k$ is an isomorphism by commutativity of the above diagram.
\end{proof}

\begin{rk}
Theorem \ref{thm-verystrong} shows in particular that $\Ext^*_\gen(F^{(r|s^2)},G^{(rs^2-rs)})$ does not depend on $s$ when $s$ is big enough (i.e. greater of equal to the degrees of $F$ and $G$). This stabilization phenomenon can be seen directly within the framework of strict polynomial functors. 
To be more specific, set $c=2rs$. Then $(G^{(rs^2-rs)})^{(c)} = G^{(r(s+1)^2-r(s+1))}$, and $(F^{(r|s^2)})^{(c)}$ is a direct summand of $F^{(r|(s+1)^2)}$. Thus precomposition by $I^{(c)}$ and projection onto the direct summand $(F^{(r|s^2)})^{(c)}$ yields a graded morphism fitting into a commutative triangle:
\[
\begin{tikzcd}
\Ext^*_\gen(F^{(r|s^2)},G^{(rs^2-rs)})
\ar{r}{\Xi_k'}\ar{r}\ar[d] & 
\Ext^*_{k[\Proj_\Fq]}(t^*F,t^*G)\\
\Ext^*_\gen(F^{(r|(s+1)^2)},G^{(r(s+1)^2-r(s+1))})\ar{ru}[swap]{\Xi_k'}&
\end{tikzcd}\;.
\]
The cokernel of the vertical map can be computed by using sum-diagonal adjunction, and for homogeneity reasons (i.e. $\Ext^*$ between two homogeneous strict polynomial multifunctors of different multidegrees vanishes) the cokernel is trivial if $p^{rs}$ is greater or equal to the degrees of $F$ and $G$.
\end{rk}

Let us give the analogue of theorem \ref{thm-verystrong} for $\Tor$. Let $F$ and $G$ be two strict polynomial functors, with $F$ contravariant. The strong comparison map \eqref{eqn-strong-inhom-Tor} for $\Tor$ and the morphisms
\begin{align*}
& t^*F\xrightarrow[]{F(\summ)}(t^{\oplus s^2})^*F\simeq F^{(r|s^2)}\;, &&& t^*G\simeq t^*(G^{(rs^2-rs)})\;,
\end{align*}
(where $\summ:t^{\oplus s^2}\to t$ is the morphism whose restriction to each direct summand $t$ of $t^{\oplus s^2}$ equals the identity of $t$)
induce a graded $k$-linear map:
\begin{align}
\Tor_*^{k[\Proj_\Fq]}(t^*F,t^*G)\to \Tor^\gen_*(F^{(r|s^2)},G^{(rs^2-rs)})\;.\label{eqn-verystrong-Tor}
\end{align}
Proposition \ref{prop-Tor-Ext} allows to dualize theorem \ref{thm-verystrong}, and we obtain the following result.
\begin{cor}\label{cor-verystrong}Let $k$ be a perfect field containing a finite subfield with $q=p^r$ elements, and let $s$ be a positive integer. Assume that $F$ and $G$ are two strict polynomial functors (respectively contravariant and covariant) with degrees less or equal to $q^s$. Then the map \eqref{eqn-verystrong-Tor} is a graded isomorphism.
\end{cor}

\section{An auxiliary comparison map}\label{sec-gen-prelim-comparison}
Throughout this section $k$ is a commutative ring, $\FF$ is a field, $\A$ is a small additive category, and we consider functors 
\[F,G:\Proj_\FF\to k\Md \;,\quad  \pi:\A^{\op}\to \FF\Md \;,\quad \rho:\A\to \FF\Md\;,\]
with $\rho$ and $\pi$ additive. In particular $\pi$ and $\rho$ may be considered as objects of the $\FF$-categories $\Mdd\FA$ and $\FA\Md$ respectively. We define a dual vector space $D_{\pi,\rho}(v)$ of an $\FF$-vector space $v$ by the formula
\begin{align*}
D_{\pi,\rho}(v):= \Hom_\FF(v, \pi\otimes_{\FA} \rho)\;.
\end{align*}

Throughout the section, we allow our additive functors $\pi$ and $\rho$ to have infinite-dimensional values. 
We recall from definition \ref{defi-comp} that a notation such as $\pi^*F$ refers to the composition $\overline{F}\circ\pi$, where $\overline{F}$ is the left Kan extension of $F$ to all $\FF$-vector spaces.  The purpose of this section is to introduce a comparison map:
\[\Theta_\FF: \Tor_*^{k[\A]}(\pi^*F,\rho^*G) \to \Tor_*^{k[\Proj_\FF]}(D_{\pi,\rho}^*F,G)\;.\]
and to establish its main properties. 
In the special case where $\FF=\Fq$ is a finite field and $k$ is an overfield of $\Fq$, the map $\Theta_\FF$ will be a key ingredient  in the proof of the generalized comparison theorem in section \ref{sec-generalized}. The last result of the section, namely theorem \ref{thm-Theta}, is interesting in its own right. 

\begin{warning}
In contrast with the other sections of the article, many constructions of this section are performed over the ground field $\FF$ (that is, we use $\FF$-linear categories, tensor products over $\FF$\dots) rather than over $k$. Since the field $\FF$ is completely independent from the commutative ring $k$, it will be clear from the context over which ring the tensor products are taken.
\end{warning}

\subsection{Construction of $\Theta_\FF$.}\label{subsec-defTheta}
Taking $\K=\FA^\op$ in the isomorphism \eqref{eqn-Formule-Cartan} of section \ref{subsec-tp} yields a pair of adjoint $\FF$-functors 
$-\otimes_{\FA}\rho:\Mdd\FA\leftrightarrows \FF\Md: \Hom_\FF(\rho,-)$.
We denote by $\theta_\FF$ the unit of adjunction: 
\begin{equation}
\theta_\FF:\pi\to \Hom_\FF(\rho,\pi\otimes_{\FA}\rho)=D_{\pi,\rho}\circ\rho\;.
\label{eqn-def-theta}
\end{equation}
Thus the component at $a$ of the natural transformation $\theta_\FF$ is the $\FF$-linear map 
\[(\theta_\FF)_a: \pi(a)\to \Hom_\FF(\rho(a), \pi\otimes_{\FA}\rho)\]
which sends $x\in\pi(a)$ to the $\FF$-linear map $y\mapsto \llbracket x\otimes y\rrbracket$ where the brackets denote the image of the tensor in $\rho\otimes_{\FA}\pi$.
Next, we choose a cardinal $\aleph$ such that the images of $\rho$ and $\pi$ are contained in the category $\Proj_\FF^{\aleph}$ of vector spaces of dimension less or equal to $\aleph$, and we let $\iota^\aleph:\Proj_\FF^{\aleph}\hookrightarrow \FF\Md$ be the inclusion of categories. We define $\Theta_\FF$ as the unique graded $k$-linear map fitting into the commutative square (note that $\res^{\iota^\aleph}$ is an isomorphism by proposition \ref{pr-Kexact}): 
\begin{equation}
\begin{tikzcd}[column sep=large]
\Tor_*^{k[\A]}(\overline{F}\circ\pi,\overline{G}\circ\rho)\ar[dashed]{r}{\Theta_\FF}
\ar{d}[swap]{\Tor_*^{k[\A]}(\overline{F}(\theta_\FF),\overline{G}\circ\rho)}
 & \Tor_*^{k[\Proj_\FF]}(\overline{F}\circ D_{\pi,\rho},G)\ar{d}{\res^{\iota^\aleph}}[swap]{\simeq}\\
\Tor_*^{k[\A]}(\overline{F}\circ D_{\pi,\rho}\circ\rho,\overline{G}\circ\rho)\ar{r}{\res^\rho} & \Tor_*^{k[\Proj_\FF^\aleph]}(\overline{F}\circ D_{\pi,\rho},\overline{G})
\end{tikzcd}.\label{eqn-def-Theta}
\end{equation}

\begin{lm}\label{lm-independant-card}
The map $\Theta_\FF$ does not depend on the choice of $\aleph$. 
\end{lm}
\begin{proof}
This is a consequence of the fact that for a cardinal $\beth$ greater than $\aleph$ we have a commutative diagram (where $\iota^{\aleph,\beth}$ is the inclusion of $\Proj_\FF^\aleph$ into $\Proj_\FF^\beth$):
\begin{equation*}
\begin{tikzcd}[column sep=large]
&  
 \Tor_*^{k[\Proj_\FF]}(\overline{F}\circ D_{\pi,\rho}F,G)\ar{d}[swap]{\simeq}{\res^{\iota^\aleph}}\ar[bend left, shift left=10ex]{dd}[swap, near start]{\simeq}[near start]{\res^{\iota^\beth}}\\
\Tor_*^{k[\A]}(\overline{F}\circ D_{\pi,\rho}\circ \pi,\overline{G}\circ \rho)\ar{r}{\res^\rho}\ar{rd}{\res^\rho} & \Tor_*^{k[\Proj_\FF^\aleph]}(\overline{F}\circ D_{\pi,\rho},\overline{G})\ar{d}{\res^{\iota^{\aleph,\beth}}}\\
& \Tor_*^{k[\Proj_\FF^\beth]}(\overline{F}\circ D_{\pi,\rho},\overline{G})
\end{tikzcd}.
\end{equation*}
\end{proof}

\begin{lm}\label{lm-nat-Theta}
The map $\Theta_\FF$ is natural with respect to $F$, $G$, $\pi$ and $\rho$.
\end{lm}
\begin{proof}
It is equivalent to prove the naturality of $\res^{\iota^\aleph}\circ\Theta_k$ with respect to $F$, $G$, $\pi$ and $\rho$.
Naturality with respect to $F$, $G$ and $\pi$ is a straightforward verification. We check naturality with respect to $\rho$, which is less straightforward since $\theta_\FF$ is not natural with respect to $\rho$. Let $f:\rho\to \rho'$ be a natural transformation, and let $D_f:D:=D_{\pi,\rho}\to D':=D_{\pi,\rho'}$ be the natural transformation induced by $f$. We consider the following diagram of graded $k$-modules, in which the composition operator for functors is omitted, e.g. `$\overline{F}\pi$' means $\overline{F}\circ \pi$, and the arrows are labelled by the natural transformations which induce them.
\[
\begin{tikzcd}
\Tor_*^{k[\A]}(\overline{F}\pi,\overline{G}\rho)\ar{rr}{\overline{G}f}\ar{d}{F\theta_\FF}\ar{dr}{\overline{F}\theta_\FF}&&\Tor_*^{k[\A]}(\overline{F}\pi,G\rho')\ar{dd}{\overline{F}\theta_\FF}\\
\Tor_*^{k[\A]}(\overline{F}D\rho,\overline{G}\rho)\ar{dd}{\res^\rho}\ar{dr}[swap]{\overline{F}D_f\rho}&\Tor_*^{k[\A]}(\overline{F}D'\rho',\overline{G}\rho)\ar{dr}{\overline{G}f}\ar{d}{\overline{F}D'f}&\\
& \Tor_*^{k[\A]}(\overline{F}D'\rho,\overline{G}\rho) \ar{dr}[swap]{\res^\rho}& \Tor_*^{k[\A]}(\overline{F}D'\rho',\overline{G}\rho')\ar{d}{\res^{\rho'}}\\
\Tor_*^{k[\Proj_\FF^\aleph]}(\overline{F}  D,\overline{G})\ar{rr}{\overline{F}D_f}&& \Tor_*^{k[\Proj_\FF^\aleph]}(\overline{F} D',\overline{G})
\end{tikzcd}
\]
The upper right triangle and the lower left triangle of the diagram are obviously commutative. The upper left parallelogram is commutative because of the dinaturality of $\theta_\FF$, i.e. because the following square commutes:
\[
\begin{tikzcd}[column sep=large]
\pi \ar{rr}{\theta_\FF}\ar{d}{\theta_\FF} &&D'\rho'=\Hom_\FF(\rho', \pi\otimes_{\FA}\rho')\ar{d}{\Hom_\FF(f, \pi\otimes_{\FA}\rho')}\\
D\rho=\Hom_\FF(\rho, \pi\otimes_{\FA}\rho)\ar{rr}{\Hom_\FF(\rho, \pi\otimes_{\FA}f)} && D'\rho=\Hom_\FF(\rho, \pi\otimes_{\FA}\rho')
\end{tikzcd}\;.
\] 
Finally, the lower right parallelogram commutes by dinaturality of restriction maps between $\Tor$-modules (which comes from the fact that tensor products are defined by a coend formula). Thus the outer square is commutative, which shows that $\res^{\iota^\aleph}\circ\Theta_\FF$, hence $\Theta_\FF$, is natural with respect to $\pi$.
\end{proof}

\subsection{Base change}\label{subsec-base-change-theta}
We fix a field morphism $\FF\to \KK$, and we let $t:\FF\Md\to \KK\Md$ denote the extension of scalars $t(v)=\KK\otimes_\FF v$. 
The canonical isomorphisms $t(\pi(a))\otimes_\KK t(\rho(a))\simeq t(\pi(a)\otimes_{\FF}\rho(a))$ induce a canonical isomorphism:
\begin{align} 
t(\pi\otimes_{\FA}\rho)\simeq (t\circ \pi)\otimes_{\KA} (t\circ \rho)\;.\label{eq-iso-tens-base-change}
\end{align}
Thus, extension of scalars induces a $\KK$-linear morphism:
\begin{align}\KK\otimes_\FF\Hom_\FF(v,\pi\otimes_{\FA}\rho) \xrightarrow[]{\,(f\otimes\lambda\mapsto f\otimes\lambda)\,} &\Hom_\KK(t(v),t(\pi\otimes_{\FA}\rho))\notag\\
&\simeq \Hom_\KK(t(v), (t\circ \pi)\otimes_{\KA} (t\circ \rho))\label{eqn-can}
\end{align}  
which is an isomorphism when $v$ has finite dimension. If we let $D_{\pi,\rho}$ and $D_{t\circ\pi,t\circ\rho}$ be the duality functors respectively defined by: 
\begin{align*}
&D_{\pi,\rho}(v)=\Hom_{\FF}(v,\pi\otimes_{\FA}\rho)\\
& D_{t\circ\pi,t\circ\rho}(w)=\Hom_\KK(w,(t\circ\pi)\otimes_{\KA}(t\circ\rho))
\end{align*}
then the morphism \eqref{eqn-can} can be written as a canonical morphism of functors
\begin{equation}t\circ D_{\pi,\rho}\xrightarrow[]{\mathrm{can}} D_{t\circ\pi,t\circ\rho}\circ t
\label{eqn-can-fct}
\end{equation}
whose component at every finite-dimensional $\FF$-vector space $v$ is an isomorphism. 

\begin{pr}\label{prop-chgt-base}
Let $\FF\to \KK$ be a field morphism. 
For all additive functors $\pi:\A^\op\to \FF\Md$ and $\rho:\A\to \FF\Md$, and for all objects $F$ and $G$ in $k[\Proj_\KK]\Md$, we have a commutative diagram in which the lower horizontal isomorphism is induced by the isomorphism $F(\mathrm{can})$:
\[
\begin{tikzcd}[column sep=large]
\Tor_*^{k[\A]}(\overline{F}\circ t\circ \pi,\overline{G}\circ t\circ \rho)\ar{r}{\Theta_\KK}\ar{d}{\Theta_\FF}& 
\Tor_*^{k[\Proj_\KK]}(\overline{F}\circ D_{t\circ \pi,t\circ \rho},G)\\
\Tor_*^{k[\Proj_\FF]}(\overline{F}\circ t\circ D_{\pi,\rho},G\circ t)\ar{r}[swap]{\simeq} &
\Tor_*^{k[\Proj_\FF]}(\overline{F}\circ D_{t\circ \pi,t\circ \rho}\circ t,G\circ t) \ar{u}{\res^t}
\end{tikzcd}.
\]
\end{pr}
\begin{proof}
Let us denote $D=D_{t\circ\pi,t\circ\rho}$ and $D'=D_{\pi,\rho}$ for short and let $\aleph$ be a big enough cardinal. 
We have a diagram of graded $k$-modules, in which the composition symbol for functors has been omitted and the arrows are labelled by the name of the morphisms which induce them. 
\[
\begin{tikzcd}
\Tor_*^{k[\A]}(\overline{F}t\pi,\overline{G}t\rho) \ar{r}{\overline{F}\theta_\KK}\ar{d}{\overline{F}t\theta_\FF}
& \Tor_*^{k[\A]}(\overline{F}Dt\rho,\overline{G}t\rho)\ar{d}{\res^\rho}\ar{r}{\res^{t\rho}}
& \Tor_*^{k[\Proj_\KK^\aleph]}(\overline{F}D,\overline{G}) \\
\Tor_*^{k[\A]}(\overline{F}tD'\rho,\overline{G}t\rho) \ar{ru}{\overline{F}\mathrm{can}\rho}\ar{d}{\res^\rho}
& \Tor_*^{k[\Proj_\FF^\aleph]}(\overline{F}Dt,\overline{G}t) \ar{ru}{\res^t}
& \Tor_*^{k[\Proj_\KK]}(\overline{F}D,\overline{G}) \ar{u}[swap]{\res}{\simeq}\\
\Tor_*^{k[\Proj_\FF^\aleph]}(\overline{F}tD',\overline{G}t) \ar{ru}{\overline{F}\mathrm{can}}
& \Tor_*^{k[\Proj_\FF]}(\overline{F}tD',\overline{G}t) \ar{l}{\res}[swap]{\simeq}\ar{r}{\overline{F}\mathrm{can}}
 & \Tor_*^{k[\Proj_\FF^\aleph]}(\overline{F}Dt,\overline{G}t) \ar{u}{\res^t}\ar{ul}[swap]{\res}{\simeq}\\
\end{tikzcd}
\]
One readily checks from the explicit expressions of $\theta_\KK$, $\theta_\FF$ and of the canonical morphism $\mathrm{can}:tD'\to Dt$ that $\theta_\KK=(\mathrm{can}\,\rho)\circ (t\theta_\FF)$, hence the upper left triangle of the diagram commutes. The other cells of the diagram obviously commute. The commutativity of the outer square proves proposition \ref{prop-chgt-base}.
\end{proof}

\subsection{Isomorphism conditions}

We now investigate some conditions which ensure that our comparison map $\Theta_\FF$ is an isomorphism. The next proposition provides the base case.

\begin{pr}\label{pr-iso-rep-Theta}
If $\A$ is $\FF$-linear and if $\pi=\A(-,a)$ and $\rho=\A(b,-)$, then $\Theta_\FF$ is an isomorphism.
\end{pr}
\begin{proof}
By lemma \ref{lm-independant-card}, we may assume $\aleph$  as big as we want in the definition of $\Theta_\FF$, so that the functor $\rho^\aleph:=\A^\aleph(b,-):\A^\aleph\to \Proj_\FF^\aleph$ has a left adjoint $\tau:=b\otimes_{\FF}-$ by proposition \ref{pr-Eadjoint}. We also let $\pi^\aleph:=\A^\aleph(-,a):\A^\aleph\to \FF\Md$.

Let us first reinterpret $\theta_\FF$ in the situation of proposition \ref{pr-iso-rep-Theta}.
We have an isomorphism
\[\phi:\pi\otimes_{\FA}\rho \xrightarrow[]{\simeq}\A(b,a)\]
which sends the class of $f\otimes g\in \A(x,a)\otimes_\FF \A(b,x)$ to  $f\circ g\in \A(b,a)$. (The inverse of $\phi$ sends an element $f\in\A(a,b)$ to the class of $\id_a\otimes f\in\A(a,a)\otimes_\FF \A(b,a)$). From the explicit expressions of $\theta_\FF$ and $\phi$, one sees that the lower left triangle of the following diagram commutes.
\begin{equation}
\begin{tikzcd}[column sep=large]
\A(x,a)\ar{rr}{\A(\epsilon_x,a)}\ar{d}[swap]{(\theta_{\FF})_x}\ar{rrd}[swap]{\A(b,-)} && \A^\aleph(b\otimes_\FF \A(b,x),a)\ar{d}{\alpha}[swap]{\simeq}\\
\Hom_\FF(\A(b,x), \pi\otimes_{\FA}\rho)\ar{rr}[swap]{\Hom_\FF(\A(b,x),\phi)}{\simeq}&& \Hom_\FF(\A(b,x),\A(b,a))
\end{tikzcd}
\label{eqn-diagr-com}
\end{equation}
The upper right triangle of diagram \eqref{eqn-diagr-com} also commutes: here $\alpha$ is an adjunction isomorphism for the adjunction between $\tau$ and $\rho^\aleph$, and $\epsilon_x$ is the associated counit of adjunction. Diagram \eqref{eqn-diagr-com} is our new interpretation of $\theta_\FF$.

Next we prove that $\res^{\iota^\aleph}\circ\Theta_\FF$ is an isomorphism. We let $D:=D_{\pi,\rho}$ for short, and we let $\chi:D\simeq \pi^\aleph\circ \tau$ be the isomorphism whose component at $v$ is given by the composition:
\[\Hom_\FF(v,\pi\otimes_{\FA}\rho)\xrightarrow[\simeq]{\Hom_\FF(v,\phi)} \Hom_\FF(v, \A(a,b))\xrightarrow[\simeq]{\alpha^{-1}} \A^\aleph(b\otimes_\FF v,a)\;.\]
We consider the following diagram of graded $k$-modules, in which the composition operator for functors is omitted and the arrows are labelled by the natural transformations which induce them. The vertical maps $\res^j$ are induced by restriction along the canonical inclusion $j:\A\hookrightarrow \A^\aleph$.
\[
\begin{tikzcd}
\Tor_*^{k[\A^\aleph]}(\overline{F}\pi^\aleph,\overline{G}\rho^\aleph)\ar{r}{\overline{F}\pi^\aleph\epsilon}&
\Tor_*^{k[\A^\aleph]}(\overline{F}\pi^\aleph \tau \rho^\aleph,\overline{G}\rho^\aleph)\ar{r}{\res^{\rho}} &
\Tor_*^{k[\Proj_\FF^\aleph]}(\overline{F}\pi^\aleph \tau,\overline{G})\ar[equal]{d}
\\
\Tor_*^{k[\A]}(\overline{F}\pi^\aleph,\overline{G}\rho^\aleph)
\ar{u}[swap]{\res^j}{\simeq}
\ar{r}{\overline{F}\pi^\aleph\epsilon}\ar[equal]{d}&
\Tor_*^{k[\A]}(\overline{F}\pi^\aleph \tau \rho,\overline{G}\rho)
\ar{u}[swap]{\res^j}{\simeq}
\ar{r}{\res^{\rho}} &
\Tor_*^{k[\Proj_\FF^\aleph]}(\overline{F}\pi^\aleph \tau,\overline{G})\\
\Tor_*^{k[\A]}(\overline{F}\pi,\overline{G}\rho)
\ar{r}{\overline{F}\theta_\FF}
&\Tor_*^{k[\A]}(\overline{F}D\rho,\overline{G}\rho)
\ar{r}{\res^{\rho}}
\ar{u}[swap]{\overline{F}\chi \rho}{\simeq}&
\Tor_*^{k[\Proj_\FF^\aleph]}(\overline{F}D,\overline{G})\ar{u}[swap]{\overline{F}\chi}{\simeq}
\end{tikzcd}
\]
All the squares of the diagram are obviously commutative, but the lower left square which commutes by commutativity of diagram \eqref{eqn-diagr-com}. The maps $\res^j$ are isomorphisms by proposition \ref{pr-Kexact} because $\overline{G}\rho^\aleph$ is the left Kan extension of $\overline{G}\rho$ along $j$. (To see this, use that $\rho^\aleph=\A^\aleph(b,-):\A^\aleph\to \FF\Md$ is $\aleph$-additive by proposition \ref{pr-elt-prop}, hence it is the left Kan extension of $\rho$ and $\overline{G}$ is already a left Kan extension.) The composite in the top row is the $\Tor$-map induced by the adjunction between $\tau$ and $\rho^\aleph$, hence it is an isomorphism by proposition \ref{pr-iso-tens-explicit2}. We deduce that the composite in the bottom row, which is nothing but $\res^{\iota^\aleph}\circ\Theta_\FF$, is an isomorphism. Hence $\Theta_\FF$ is an isomorphism.
\end{proof}

\begin{cor}\label{cor-direct-sum-rep-Theta}
If $\A$ is $\FF$-linear, and if 
$$\pi=\bigoplus_{i\in I}\A(-,a_i)\;,\qquad \rho=\bigoplus_{j\in J}\A(b_j,-)$$
for some possibly infinite indexing sets $I$ and $J$, then $\Theta_\FF$ is an isomorphism.
\end{cor}
\begin{proof}
If $I$ and $J$ are finite, then $\pi\simeq \A(-,a)$ and $\rho\simeq \A(b,-)$ for $a=\bigoplus a_i$ and $b=\bigoplus b_j$, hence $\Theta_\FF$ is an isomorphism by proposition \ref{pr-iso-rep-Theta}. For arbitrary $I$ and $J$, the functors $\pi$ and $\rho$ are filtered colimits of monomorphisms of functors of the form $\A(-,a)$ and $\A(b,-)$. So the result follows from the fact that the target and the source of $\Theta_\FF$ both preserve filtered colimits of monomorphisms of functors, when viewed as functors of the variables $\pi$ and $\rho$. (Indeed $\Tor_*$, $\overline{F}\circ\pi$, $\overline{G}\circ\rho$ and $\overline{F}\circ D_{\pi,\rho}$ preserve filtered colimits of monomorphisms -- for $\overline{F}\circ\pi$, $\overline{G}\circ\rho$, this follows from the fact that $\overline{F}$ and $\overline{G}$ are left Kan extensions of $F$ and $G$, and for $\overline{F}\circ D_{\pi,\rho}$, one uses in addition the isomorphism $D_{\pi,\rho}(v)\simeq \Hom_\FF(v,\FF)\otimes_\FF (\pi\otimes_{\FA}\rho)$, which holds because we view $D_{\pi,\rho}$ as a functor from $\Proj_\FF$ to $\FF\Md$.)
\end{proof}

We are going to extend corollary \ref{cor-direct-sum-rep-Theta} to more general $\FF$-linear functors $\pi$ and $\rho$ by taking simplicial resolutions. We refer the reader to section \ref{sec-simplicial} for recollections of simplicial techniques. 
If $X$ is a simplicial object in $k[\Proj_\FF]\Md$, and $\mu$ is a simplicial object in the category of additive functors $\A\to \FF\Md$, we let $\mu^*X$ be the diagonal simplicial object $\overline{X}_n\circ \mu_n$. Thus $\mu^*X$ is a simplicial object in $k[\A]\Md$ natural with respect to $\mu$ and $X$. The next two lemmas are our main tools to contruct convenient simplicial resolutions.
\begin{lm}\label{lm-Whitehead-thm}
If $X_1\to X_2$ and  $\mu_1\to \mu_2$ are $e$-connected morphisms, the induced morphism $\mu_1^*X_1\to \mu_2^*X_2$ is $e$-connected. 
\end{lm}
\begin{proof}
We have to show that for all $a$ in $\A$ the morphism of simplicial $k$-modules $\overline{X_1}(\mu_1(a))\to \overline{X_2}(\mu_2(a))$ is $e$-connected.
The morphism $\overline{X_1}\to \overline{X_2}$ is $e$-connected because it is defined as a filtered colimit and homotopy groups commute with filtered colimits. Hence the result follows from proposition \ref{prop-simpl-connected}.
\end{proof}

\begin{lm}[linearization of projective additive functors]\label{lm-flat}
Let $\rho:\A\to \FF\Md$ be an additive functor, and let $P$ be a projective object in $k[\Proj_\FF]\Md$. If $\rho$ is projective as an {\em additive} functor from $\A$ to {\em abelian groups}, 
then $\rho^*P$ is a projective over $k[\A]$.
\end{lm}
\begin{proof}
It suffices to prove the result when $P=k[\Hom_\FF(\FF^n,-)]$. As $\rho$, seen as an object of $\A\Md$, is a direct summand of a direct sum of representable functors, it is enough to see that $k[\bigoplus_{i\in E}\A(a_i,-)]$ is projective over $k[\A]$ for every family $(a_i)_{i\in E}$ of objects of $\A$.
The result follows from the cross-effect type decomposition
$$k[\bigoplus_{i\in E}\A(a_i,-)]\simeq\bigoplus_{I\in\mathcal{P}_f(E)}\bigotimes_{i\in I}k[\A(a_i,-)]^\red$$
where $^\red$ refers to the reduced part of a functor (see lemma~\ref{lm-vanish-pol} in section~\ref{sec-AP}) and $\mathcal{P}_f(E)$ denotes the set of finite subsets of $E$. Indeed, each functor $\bigotimes_{i\in I}k[\A(a_i,-)]^\red$ is a direct summand of the projective functor $k[\A(\bigoplus_{i\in I}a_i,-)]$, hence the left-hand side of the isomorphism is a projective functor as claimed.
\end{proof}

\begin{thm}\label{thm-Theta}
Let $k$ be a commutative ring and let $\A$ be a small additive $\FF$-category for a field $\FF$. Assume that $\pi$ and $\rho$ are $\FF$-linear, consider them as objects of the $\FF$-categories $\Mdd\A$ and $\A\Md$, and let $e$ be a positive integer such that $\Tor_i^{\A}(\pi,\rho)=0$ for $0<i<e$.
Then for all objects $F$, $G$ of $k[\Proj_\FF]\Md$, the map
$$\Theta_\FF:\Tor^{k[\A]}_*(\pi^*F,\rho^*G)\to \Tor_*^{k[\Proj_\FF]}(D^*_{\pi,\rho}F,G)$$
defined by diagram \eqref{eqn-def-Theta} is $e$-connected.
\end{thm}
\begin{proof}
Let $\mathcal{G}\to G$, $\varpi\to \pi$ and $\varrho\to \rho$ be simplicial resolutions by direct sums of standard projectives in the categories $k[\Proj_\FF]\Md$, $\Mdd\A$ and $\A\Md$ respectively. 
We have a commutative diagram of simplicial $k$-modules, in which the maps $(\dag)$ are induced by the morphisms $\varpi\to \pi$ and $\varrho\to \rho$, and the maps $\widetilde{\Theta_\FF}$ are degreewise equal to the degree zero component of $\Theta_\FF$, hence they are isomorphisms by corollary \ref{cor-direct-sum-rep-Theta}:
\[
\begin{tikzcd}[column sep=large]
\varpi^*F\otimes_{k[\A]}\varrho^*\G\ar{dd}[swap]{\simeq}{\widetilde{\Theta_\FF}}\ar{rd}[swap]{(\dag)}\ar{r}{F(\theta_\FF)\otimes\id}
&\rho^*D_{\varpi,\rho}^*F\otimes_{k[\A]}\varrho^*\G\ar{r}{(\dag)} & 
\rho^*D_{\pi,\rho}^*F\otimes_{k[\A]}\rho^*\G
\ar{dd}{\res^\rho} \\
& \pi^*F\otimes_{k[\A]}\rho^*G\ar{rd}[swap]{\simeq}{\widetilde{\Theta_\FF}}\ar{ru}[swap]{F(\theta_\FF)\otimes\id} &\\
D_{\varpi,\varrho}^*F\otimes_{k[\Proj_k]}\G\ar{rr}{(\dag)} &&  D_{\pi,\rho}^*F\otimes_{k[\Proj_k]}\G
\end{tikzcd}\;.
\]
The homotopy groups of $D_{\pi,\rho}^*F\otimes_{k[\Proj_\FF]}\mathcal{G}$ compute $\Tor_*^{k[\Proj_\FF]}(D_{\pi,\rho}^*F,G)$ because $\mathcal{G}\to G$ is a projective simplicial resolution. The homotopy groups of $\varpi^*F\otimes_{k[\A]}\varrho^*\mathcal{G}$ compute $\Tor_*^{k[\A]}(\pi^*F,\rho^*G)$ because $\varrho^*\mathcal{G}\to \rho^*G$ is a simplicial projective resolution by lemmas \ref{lm-Whitehead-thm} and \ref{lm-flat}. And the map induced on the level of homotopy groups by the top right corner of the diagram is $\Theta_\FF$. Thus, to prove the theorem, it remains to prove that the bottom horizontal map $(\dag)$ is $e$-connected.

In order to do this, we first observe that the $\FF$-category $\A\Md$ of $\FF$-functors from $\A$ to $\FF\Md$ is a full subcategory of the $\FF$-category $\FA\Md$ of additive functors from $\A$ to $\FF\Md$. Hence for all $\FF$-functors $\pi$ and $\rho$, restriction along the functor $\FA\to \A$, $f\otimes\lambda\mapsto \lambda f$ induces an isomorphism 
$\pi\otimes_{\FA}\rho\simeq  \pi\otimes_{\A}\rho$ by corollary \ref{cor-memechose2}.
Thus the $\Tor$-condition in the theorem ensures that $\varpi\otimes_\A\varrho\to \pi\otimes_\A\rho$ is $e$-connected. Thus $D_{\varpi,\varrho}\to D_{\pi,\rho}$ is $e$-connected, hence $D_{\varpi,\varrho}^*F\to D_{\pi,\rho}^*F$ is $e$-connected by lemma \ref{lm-Whitehead-thm}. This implies that the bottom horizontal map is $e$-connected by a standard spectral sequence argument (use the spectral sequence of a bisimplicial $k$-module as in \cite[IV section 2.2]{GoerssJardine}).
\end{proof}

\section{The generalized comparison theorem}\label{sec-generalized}
Throughout this section, $k$ is an infinite perfect field of positive characteristic $p$, $\pi:\A^\op\to k\Md$ and $\rho:\A\to k\Md$ are two additive functors and $F$ and $G$ are two strict polynomial functors over $k$ (possibly non-homogeneous, cf. section \ref{subsec-nonhomogeneous}). In this section we state and prove the generalized comparison theorem, which computes $\Tor^{k[\A]}_*(\pi^*F,\rho^*G)$ in terms of the generic homology $\Tor^\gen_*(F^\dag,G)$ for some functor $F^\dag$ constructed from $F$, $\pi$ and $\rho$. Then we spell out explicitly some special cases of the generalized comparison theorem in subsection \ref{sec-special-cases}, in particular we establish theorem \ref{thm-intro-F-lin-case} of the introduction. Finally, we give the $\Ext$-versions of the generalized comparison theorem in subsection \ref{subsec-Ext-comp}. 

\subsection{Statement of the generalized comparison theorem}
Let us fix two positive integers $r$ and $s$, and for $0\le i<s^2$ we let $T_i$ denote the $k$-vector space:
\begin{align}
T_i=({}^{(-ri)}\pi)\otimes_{\kATor}({}^{(rs-rs^2)}\rho)\label{eqn-def-Ti}
\end{align}
where a notation such as $^{(-ri)}\pi$ refers to the additive functor obtained as the composition of the functor $\pi:\A\to k\Md$ and the Frobenius twist $^{(-ri)}- :k\Md\to k\Md$. 
Each vector space $T_i$ determines a duality functor $\hom_k(-,T_i):\Proj_k^\op\to k\Md$. By composing these duality functors with Frobenius twists and by taking their direct sum, one obtains a contravariant functor $A:\Proj_k^\op\to k\Md$ with
\begin{align}
A(v)= \bigoplus_{0\le i<s^2} {}^{(ri)}\Hom_k(v,T_i)
\end{align}
We define a functor $F^\dag$ in $\Mdd k[\Proj_k]$ by
\begin{align}
F^\dag= \overline{F}\circ A\label{eqn-def-Fprime}
\end{align}
where as in definition \ref{defi-comp}, the notation $\overline{F}$ refers to the left Kan extension of $F$ to all vector spaces. Since $k$ is an infinite field, the category $\Mdd\Gamma\Proj_k$ of contravariant strict polynomial functors identifies with a full subcategory of $\Mdd k[\Proj_k]$, stable under colimits (see section \ref{subsec-nonhomogeneous}). 
\begin{lm}
If $F$ is a strict polynomial functor of degree $d$, then $F^\dag$ is a strict polynomial functor of degree $dp^{r(s^2-1)}$.
\end{lm}
\begin{proof}
The vector spaces $T_i$ can be written as a filtered colimit of finite-dimensional subspaces $T_{i,\alpha}$. For all $\alpha$, $F^\dag_\alpha(v)=F(\bigoplus_{0\le i<s^2} {}^{(ri)}\Hom_k(v,T_{i,\alpha}))$ is a strict polynomial functor of the variable $v$ with degree $dp^{r(s^2-1)}$ since it is the composition of a strict polynomial functor of degree $d$ by a strict polynomial functor of degree $p^{r(s^2-1)}$. Now $F^\dag=\colim_\alpha F^\dag_\alpha$ whence the result.
\end{proof}

In subsection \ref{subsec-gen-comp-map}, we will first construct an explicit morphism of graded vector spaces, natural with respect to $F$, $G$, $\pi$ and $\rho$:
\begin{align}
\Tor_*^{k[\A]}(\pi^*F,\rho^*G)\to \Tor_*^{\gen}(F^\dag,G^{(rs^2-rs)})\;.\label{eqn-gencompmap}
\end{align}
We will refer to morphism \eqref{eqn-gencompmap} as the \emph{generalized comparison map} in the remainder of the article. The following generalized comparison theorem will be proved in subsection \ref{subsec-proof-gencomp}.
Notice that the vector spaces $T_i$ appearing in the construction of $F^\dag$ are the degree zero components of the $\Tor$ appearing in the vanishing condition.

\begin{thm}[generalized comparison]\label{thm-gen-comp}
Let $k$ be an infinite perfect field of characteristic $p$, containing a finite field $\Fq$ of cardinal $q=p^r$. 
Let $\A$ be a small additive category, let $\pi:\A^\op\to k\Md$ and $\rho:\A\to k\Md$ be two additive functors. Assume that there are positive integers $s$ and $e$ such that 
\[\Tor_j^{\kATor}\left({}^{(-ri)}\pi,{}^{(rs-rs^2)}\rho\right)=0\]
for $0< j<e$ and $0\le i<s^2$. Assume further that $\A$ is $\Fq$-linear and that $\rho$ and $\pi$ are $\Fq$-linear. Then for all strict polynomial functors $F$ and $G$ of degrees less or equal to $q^s$, the generalized comparison map \eqref{eqn-gencompmap} is $e$-connected.
\end{thm}

\begin{rk}
If $\pi$ and $\rho$ are $\Fq$-linear, so are ${}^{(ri)}\pi$ and ${}^{(rs-rs^2)}\rho$, and we can consider them as objects of $k\otimes_\Fq\A\Md$ and $\Mdd k\otimes_\Fq\A$. Hence it is natural to ask about the relation between the $\Tor$ hypothesis in theorem \ref{thm-gen-comp} and the vanishing of $\Tor_j^{k\otimes_\Fq\A}({}^{(ri)}\pi,{}^{(rs-rs^2)}\rho)$. These two conditions are actually equivalent as we shall see it in lemma \ref{lm-FpFq}. 
\end{rk}


\subsection{Construction of the generalized comparison map}\label{subsec-gen-comp-map}
We define the generalized comparison map \eqref{eqn-gencompmap} as the composition of three maps $\Lambda_k$, $\Theta'_k$ and $\Phi_k$: 
\[
\begin{tikzcd}
\Tor^{k[\A]}_*(\pi^*F,\rho^*G)\ar[dashed]{d}{\text{\eqref{eqn-gencompmap}}}\ar{rr}{\Lambda_k} && \Tor_*^{k[\A]}((\pi^{\oplus s^2})^*F,\rho^*G)\ar{d}{\Theta^\dag_k}\\
\Tor_*^{\gen}(F',G^{(rs^2-rs)})&&\ar{ll}{\Phi_k}\Tor_*^{k[\Proj_k]}(F',G^{(rs^2-rs)})
\end{tikzcd}\;.
\]

The map $\Lambda_k$ is induced by the morphism of additive functors $\mathrm{diag}:\pi\to \pi^{\oplus s^2}$ whose components all equal the identity morphism of $\pi$.

The map $\Phi_k$ is an immediate generalization of the strong comparison map \eqref{eqn-strong-inhom-Tor} over an infinite perfect field $k$. Namely, the additive functor 
$^{(n)}-:\Proj_k\to \Proj_k$
has a quasi-inverse $^{(-n)}-$ for all positive integers $n$. Restriction along this quasi-inverse yields an isomorphism of graded vector spaces, natural in the functors $H$ and $K$:
\[\Tor_*^{k[\Proj_k]}(H,K)\xrightarrow[]{\simeq}  \Tor_*^{k[\Proj_k]}(H^{(n)},K^{(n)})\;,\]
where a notation such as $H^{(n)}$ denotes the precomposition of $H$ by $^{(n)}-$. 
Since $k$ is infinite, the category $\Gamma\Proj_k\Md$ of strict polynomial functors  identifies with a full subcategory of $k[\Proj_k]\Md$ (see section \ref{subsec-nonhomogeneous}). If $H$ and $K$ are strict polynomial functors, we let $\Phi_k$ be the unique morphism of graded $k$-vector spaces fitting in the commutative diagrams for all $i\ge 0$ and all $r\gg 0$:
\begin{equation}
\begin{tikzcd}
\Tor_i^{k[\Proj_k]}(H,K)\ar{d}{\res^{^{(-n)}-}}[swap]{\simeq}\ar[dashed]{rr}{\Phi_k}&& \Tor_i^{\gen}(H,K)\ar{d}{\simeq}\\
\Tor_i^{k[\Proj_k]}(H^{(n)},K^{(n)})\ar{rr}{\res}&& \Tor_i^{\Gamma\Proj_k}(H^{(n)},K^{(n)})
\end{tikzcd}\;.
\end{equation}

It remains to define $\Theta^\dag_k$. This map $\Theta^\dag_k$ is a generalization of the comparison map $\Theta_k$ constructed in section \ref{sec-gen-prelim-comparison}, and just as $\Theta_k$, its definition involves three graded maps between $\Tor$. In order to avoid heavy notations, we set:
\begin{align*}\pi_i:={}^{(-ri)}\pi\;,&&&\sigma:={}^{(rs-rs^2)}\rho\;.
\end{align*}
To define the first graded map, recall that every $k$-vector space $v$ has a dual $D_{\pi_i,\sigma}(v)= \Hom_k(v,\pi_i\otimes_{\kA}\sigma)$. With this notation, one can write $F^\dag$ as a composition:
\[F^\dag=\overline{F}\circ \big(\bigoplus_{0\le i<s^2} {}^{(ri)}D_{\pi_i,\sigma}\big)\;.\] 
Since the functors $\overline{F}$ and $D_{\pi_i,\sigma}$ can be fed with arbitrary dimensional vector spaces, this formula actually defines $F^\dag$ as an endofunctor of $k\Md$. In particular the composition $F^\dag\circ \sigma$ is well defined.
Let $\aleph$ be a cardinal greater than the dimension of $\sigma(a)$ for all objects $a$ of $\A$. Hence $\sigma$ can be considered as a functor $\A\to \Proj_k^\aleph$, and restriction along $\sigma$ defines a graded map:
\[\Tor_*^{k[\A]}(F^\dag\circ \sigma,\overline{G}^{(rs^2-rs)}\circ\sigma)\to \Tor_*^{k[\Proj_k^\aleph]}(F^\dag,\overline{G}^{(rs^2-rs)})\;.\]
To define the second graded map, observe that the functor $\overline{G}{}^{(rs^2-rs)}$ is the left Kan extension of $G^{(rs^2-rs)}$  to all vector spaces (because $-^{(rs^2-rs)}$ is an autoequivalence of the category of $k$-vector spaces, hence it preserves filtered colimits). Therefore, by proposition \ref{pr-Kexact} restriction along the inclusion $\iota^\aleph:\Proj_k\hookrightarrow \Proj_k^\aleph$ yields an isomorphism on the level of $\Tor$:
\[\Tor_*^{k[\Proj_k]}(F^\dag,G^{(rs^2-rs)})\xrightarrow[]{\simeq}\Tor_*^{k[\Proj_k^\aleph]}(F^\dag,\overline{G}^{(rs^2-rs)})\;.\]
Finally the third graded map will be induced by a certain morphism $\theta^\dag_k$ constructed from the morphisms $\theta_k$ defined by equation \eqref{eqn-def-theta} in section \ref{sec-gen-prelim-comparison}. To be more specific we define $\theta_k^\dag$ as the direct sum
\begin{align}
\theta_k^\dag:\pi^{\oplus s^2}=\bigoplus_{0\le i<s^2} {}^{(ri)}\pi_i \xrightarrow[]{\bigoplus {}^{(ri)}\theta_k} \bigoplus_{0\le i<s^2} {}^{(ri)}D_{\pi_i,\sigma}\circ\sigma\;.\label{eqn-def-thetaprime}
\end{align}
We can now define our map $\Theta^\dag_k$ as the unique map making the following square commute:
\begin{equation}
\begin{tikzcd}[column sep=large]
\Tor_*^{k[\A]}(\overline{F}\circ (\pi^{\oplus s^2}),\overline{G}\circ\rho)\ar[dashed]{r}{\Theta^\dag_k}
\ar{d}[swap]{\Tor_*^{k[\A]}(\overline{F}(\theta_k^\dag),\overline{G}\circ\rho )}
 & \Tor_*^{k[\Proj_k]}(F^\dag,G^{(rs^2-rs)})\ar{d}{\res^{\iota^\aleph}}[swap]{\simeq}\\
\Tor_*^{k[\A]}(F^\dag\circ \sigma,\overline{G}^{(rs^2-rs)}\circ\sigma)\ar{r}{\res^\sigma} & \Tor_*^{k[\Proj_k^\aleph]}(F^\dag,\overline{G}^{(rs^2-rs)})
\end{tikzcd}.\label{eqn-def-Thetaprime}
\end{equation}

The following lemma is proved exactly in the same way as lemmas \ref{lm-independant-card} and \ref{lm-nat-Theta}.
\begin{lm}
The map $\Theta^\dag_k$ does not depend on the cardinal $\aleph$. Moreover, it is natural with respect to $F$, $G$, $\pi$ and $\rho$.
\end{lm}

Next, we clarify the relation between $\Theta_k^\dag$ and the comparison map $\Theta_k$ from section \ref{sec-gen-prelim-comparison}.
Assume that we are given isomorphisms of additive functors $\pi\simeq{}^{(-ri)}\pi$ and $\rho\simeq{}^{(rs-rs^2)}\rho$. Then these isomorphisms induce isomorphisms (where $F^{(r|s^2)}$ stands for the strict polynomial functor $v\mapsto F(\bigoplus_{0\le i<s^2} {}^{(ri)}v)$ as in notation \ref{nota-twist2}):
\begin{align*}
&(\pi^{\oplus s^2})^*F\simeq\pi^*(F^{(r|s^2)})\;,&&\rho^*G\simeq \rho^*(G^{(rs^2-rs)})\;,
&&F^\dag\simeq D_{\pi,\rho}^*(F^{(r|s^2)})\;,
\end{align*} 
and the next lemma is a straightforward verification.
\begin{lm}\label{lm-ThetaprimeTheta}
There is a commutative square, whose vertical isomorphisms are induced by the above isomorphisms of functors:
\[
\begin{tikzcd}
\Tor_*^{k[\A]}((\pi^{\oplus s^2})^*F,\rho^* G)\ar{r}{\Theta^\dag_k}\ar{d}{\simeq}
 & \Tor_*^{k[\Proj_k]}(F^\dag,G^{(rs^2-rs)})\ar{d}{\simeq}\\
\Tor_*^{k[\A]}(\pi^*(F^{(r|s^2)}),\rho^*G^{(rs^2-rs)})\ar{r}{\Theta_k}
 & \Tor_*^{k[\Proj_k]}(D_{\pi,\rho}^*(F^{(r|s^2)}),G^{(rs^2-rs)})
\end{tikzcd}\;.
\]
\end{lm}

\subsection{Proof of the generalized comparison theorem \ref{thm-gen-comp}}\label{subsec-proof-gencomp}
The proof of the generalized comparison theorem uses the same ingredients as the proof of theorem \ref{thm-main-additive}. Before diving into technical details, let us give an overview of these ingredients.
\begin{enumerate}[1.]
\item The proof can be reduced to the special case where $\pi$ is a direct sum of functors $k\otimes_\Fq\A(-,a_i)$ and $\rho$ is a direct sum of functors $k\otimes_\Fq\A(b_j,-)$. This reduction step is proved by taking resolutions of $\pi$ and $\rho$ by such functors, and by examining the associated spectral sequences. The main difference with the proof of theorem \ref{thm-main-additive} is that in the context of theorem \ref{thm-gen-comp} we have to use \emph{simplicial} resolutions of $\pi$ and $\rho$. The $\Tor$-vanishing condition is used for this reduction step, to ensure that by tensoring a simplicial resolution of $^{(-ri)}\pi$ with a simplicial resolution of $^{(rs-rs^2)}\rho$ we obtain a simplicial resolution of $^{(-ri)}\pi\otimes_{\kA}{}^{(rs-rs^2)}\rho$.
\item When $\pi$ is a direct sum of functors $k\otimes_\Fq\A(-,a_i)$ and $\rho$ is a direct sum of functors $k\otimes_\Fq\A(b_j,-)$, the generalized comparison map can be rewritten in terms of the map $\Theta_\Fq$ of section \ref{sec-gen-prelim-comparison} and of the comparison map \eqref{eqn-verystrong} of section \ref{sec-Pol-Fq}. The former is an isomorphism by theorem \ref{thm-Theta} (or by the special case given in corollary \ref{cor-direct-sum-rep-Theta}), and the latter is an isomorphism by theorem \ref{thm-verystrong}, which proves that the generalized comparison map is an isomorphism. The details of this argument are given in proposition \ref{prop-base-case}.   
\end{enumerate}

Let us assume that the hypotheses of theorem \ref{thm-gen-comp} are satisfied, in particular $\A$ is $\Fq$-linear over some finite subfield $\Fq$ of $k$, and $\pi$ and $\rho$ are $\Fq$-linear. Next proposition is the base case of the proof.
\begin{pr}\label{prop-base-case}
Assume that
$$\pi=\bigoplus_{i\in I}k\otimes_\Fq\A(-,a_i)\;,\qquad \rho=\bigoplus_{j\in J}k\otimes_\Fq\A(b_j,-)$$
for some possibly infinite indexing sets $I$ and $J$. Then the generalized comparison map \eqref{eqn-gencompmap} is a graded isomorphism. 
\end{pr}
\begin{proof}
If $t(v)=k\otimes_\Fq v$ denotes the extension of scalars from $\Fq$ to $k$, then we have $\pi\simeq t\circ\mu$ and $\rho\simeq t\circ \nu$ for some additive functors $\mu:\A^\op\to \Fq\Md$ and $\nu:\A\to \Fq\Md$. 
The canonical isomorphisms of functors $t\simeq {}^{(ri)}t$ induce isomorphisms $\pi\simeq {}^{(ri)}\pi$ and $\rho\simeq {}^{(rs-rs^2)}\rho$, so by lemma \ref{lm-ThetaprimeTheta} and by naturality of $\Phi_k$ with respect to the isomorphism $F^\dag\xrightarrow[]{\simeq} \overline{F}{}^{(r|s^2)}\circ D_{\pi,\rho}$, we have a commutative square (in which the composition operator of functors is omitted in the arguments of $\Tor$):
\[
\begin{tikzcd}[column sep=large]
\Tor_*^{k[\A]}(\overline{F}\pi,\overline{G}\rho)\ar{dd}{\text{\eqref{eqn-gencompmap}}}
\ar{r}{\Lambda'_k}
& 
\Tor_*^{k[\A]}(\overline{F}^{(r|s^2)}\pi,\overline{G}^{(rs^2-rs)}\rho)\ar{d}{\Theta_k}
\\
&
\Tor_*^{k[\Proj_k]}(\overline{F}^{(r|s^2)}D_{\pi,\rho}, G^{(rs^2-rs)})\ar{d}{\Phi_k}
\\
\Tor_*^{\gen}(F^\dag,G^{(rs^2-rs)})\ar{r}{\simeq}
&
\Tor_*^{\gen}(\overline{F}^{(r|s^2)}D_{\pi,\rho},G^{(rs^2-rs)})
\end{tikzcd}
\]
in which $\Lambda'_k$ is induced by the canonical isomorphism $t\simeq {}^{(rs^2-rs)}t$ and by the morphism of additive functors: 
\[\mathfrak{d}:t\xrightarrow[]{\diag} t^{\oplus s^2}\simeq \bigoplus_{0\le i<s^2}{}^{(ri)}t\;.\]
Therefore in order to prove proposition \ref{prop-base-case} it suffices to prove that the composition $\Phi_k\circ\Theta_k\circ\Lambda'_k$ in top right corner of the diagram is an isomorphism. 

Next, since $\pi=t\circ\mu$ and $\rho=t\circ\nu$, the base change property of proposition \ref{prop-chgt-base} gives a commutative square:
\[
\begin{tikzcd}
\Tor_*^{k[\A]}(\overline{F}^{(r|s^2)}\pi,\overline{G}^{(rs^2-rs)}\rho)\ar{d}{\Theta_k}\ar{r}{\Theta_\Fq}
&
\Tor_*^{k[\Proj_\Fq]}(\overline{F}^{(r|s^2)}t D_{\mu,\nu},\overline{G}^{(rs^2-rs)} t)
\ar{d}{\simeq}
\\
\Tor_*^{k[\Proj_k]}(\overline{F}^{(r|s^2)}D_{\pi,\rho}, G^{(rs^2-rs)})
&
\Tor_*^{k[\Proj_\Fq]}(\overline{F}^{(r|s^2)}D_{\pi,\rho}t,\overline{G}^{(rs^2-rs)} t)
\ar{l}{\res^t}
\end{tikzcd}\;.
\]
Hence, by naturality of $\Theta_{\Fq}$ with respect to the isomorphism $\overline{G}t\simeq \overline{G}{}^{(rs^2-rs)}t$ and the morphism $F(\mathfrak{d}): \overline{F}t\to \overline{F}{}^{(r|s^2)}$, we obtain a commutative square:
\[
\begin{tikzcd}
\Tor_*^{k[\A]}(\overline{F}\pi,\overline{G}\rho)\ar{d}{\Theta_k\circ\Lambda_k'}
\ar{r}{\Theta_\Fq}
&
\Tor_*^{k[\Proj_\Fq]}(\overline{F}t D_{\mu,\nu},\overline{G}t)\ar{d}{\Lambda''_k}
\\
\Tor_*^{k[\Proj_k]}(\overline{F}^{(r|s^2)}D_{\pi,\rho}, G^{(rs^2-rs)})
&
\Tor_*^{k[\Proj_\Fq]}(\overline{F}^{(r|s^2)}D_{\pi,\rho}t,\overline{G}^{(rs^2-rs)} t)
\ar{l}{\res^t}
\end{tikzcd}
\]
where $\Lambda''_\Fq$ is induced by $\mathfrak{d}$, by the isomorphism $t\simeq {}^{(rs^2-rs)}t$ and by the canonical isomorphism $\mathrm{can}: t\circ D_{\mu,\nu}\simeq D_{\pi,\rho}\circ t$.
The map $\Theta_\Fq$ on the top row of this square is an isomorphism by corollary \ref{cor-direct-sum-rep-Theta}. Hence, to prove proposition \ref{prop-base-case} it remains to prove that the composition $\Phi_k\circ\res^t\circ \Lambda''_k$ is an isomorphism.

For this purpose, we are going to rewrite the composition $\Phi_k\circ\res^t\circ \Lambda''_k$ into yet another form.
We claim that there is a $k$-linear isomorphism, natural with respect to $v$, $\mu$ and $\nu$:
\[\psi_v: {}^{(r|s^2)}D_{t\mu,t\nu}(v)\to{D_{t\mu,t\nu}}({}^{(r|s^2)}v)\;, \]
Indeed, we have isomorphisms of vector spaces, natural with respect to $\mu$ and $\nu$:
\[\phi_i:{}^{(ri)}\left(t\mu\otimes_{\kATor}t\nu\right)\xrightarrow[]{\simeq} t\mu\otimes_{\kATor}t\nu\]
which send the class  $\llbracket(\alpha\otimes x)\otimes (\beta\otimes y)\rrbracket$ where $\alpha,\beta\in k$, $x\in \mu(a)$ and $y\in \nu(a)$ to the class $\llbracket(\alpha^{p^{ri}}\otimes x)\otimes (\beta^{p^{ri}}\otimes y)\rrbracket$. We define $\psi_v$ as the following composition, where $T$ stands for $t\mu\otimes_{\kATor}t\nu$, the first and last isomorphisms are the canonical ones and the second morphism is induced by the $\phi_i$:
\begin{align*}
{}^{(r|s^2)}D_{t\mu,t\nu}(v)=\bigoplus_{0\le i<s^2} {}^{(ri)}\Hom_k(v,T) &\xrightarrow[]{\simeq} \bigoplus_{0\le i<s^2} \Hom_k({}^{(ri)}v,{}^{(ri)}T)\\
&\xrightarrow[]{\simeq} \bigoplus_{0\le i<s^2} \Hom_k({}^{(ri)}v,T)\\
&\xrightarrow[]{\simeq} \Hom_k(\bigoplus_{0\le i<s^2} {}^{(ri)}v,T)=D_{t\mu,t\nu}({}^{(r|s^2)}v)\;.
\end{align*}
Moreover, one readily checks that $\psi$ fits into a commutative diagram in the category of functors from $\Proj_\Fq$ to $k\Md$:
\begin{equation}
\begin{tikzcd}[column sep=large]
tD_{\mu,\nu}\ar{d}{\mathrm{can}}[swap]{\simeq}\ar{r}{\mathfrak{d} (D_{\mu,\nu})}& {}^{(r|s^2)}tD_{\mu,\nu}\ar{r}{{}^{(r|s^2)}\mathrm{can}}[swap]{\simeq}&{}^{(r|s^2)}D_{\pi,\rho}t\ar{d}{\psi(t)}[swap]{\simeq}\\
D_{\pi,\rho}t\ar{rr}{D_{\pi,\rho}(\mathfrak{s})}&& D_{\pi,\rho}{}^{(r|s^2)}t
\end{tikzcd}\label{eqn-diagram-fin}
\end{equation}
where $\mathfrak{s}$ denotes the composition $^{(r|s^2)}t\simeq t^{\oplus s^2}\xrightarrow[]{\summ}t$. 
Diagram \eqref{eqn-diagram-fin} and naturality of $\Phi_k$ and $\res^t$ with respect to the map $\overline{F}(\psi)$ yield a commutative square, in which the horizontal isomorphisms are induced by the isomorphisms $\overline{F}(\mathrm{can})$ and $\overline{F}(\psi)$ and the map $\Lambda'''_k$ is induced by $\overline{F}(D_{\pi,\rho}(\mathfrak{s}))$:
\[
\begin{tikzcd}
\Tor_*^{k[\Proj_\Fq]}(\overline{F}t D_{\mu,\nu},\overline{G}t)\ar{r}{\simeq}\ar{d}{\Phi_k\circ\res^t\circ \Lambda''_k}& \Tor_*^{k[\Proj_\Fq]}(\overline{F}D_{\pi,\rho}t,\overline{G}t)\ar{d}{\Phi_k\circ\res^t\circ \Lambda'''_k}\\
\Tor_*^{\gen}(\overline{F}{}^{(r|s^2)}D_{\pi,\rho}, G^{(rs^2-rs)})\ar{r}{\simeq}& \Tor_*^{\gen}(\overline{F}{D_{\pi,\rho}}^{(r|s^2)}, G^{(rs^2-rs)})
\end{tikzcd}\;.
\]

Thus, to prove proposition \ref{prop-base-case}, it suffices to prove that $\Phi_k\circ\res^t\circ \Lambda'''_k$ is an isomorphism. But $\Phi_k\circ\res^t\circ \Lambda'''_k$ equals the comparison map of equation \eqref{eqn-verystrong-Tor}, hence it is an isomorphism by corollary \ref{cor-verystrong}. This concludes the proof.
\end{proof}

We are going to extend the isomorphism of proposition \ref{prop-base-case} to more general $\Fq$-linear functors $\pi:\A^\op\to k\Md$ and $\rho:\A\to k\Md$ by taking simplicial resolutions. We first need an elementary lemma regarding the computation of $\Tor$ between $\Fq$-linear functors.
The functor $\rho$ is an object of the $k$-category of all additive functors from $\A$ to $k\Md$, which identifies with $\kA\Md$. But $\rho$ is also an object an object of the category of all $\Fq$-linear functors from $\A$ to $k\Md$, which identifies with $(k\otimes_\Fq \A)\Md$. Similarly, $\pi$ can be viewed as an object of $\Mdd \kA$ or $\Mdd (k\otimes_\Fq \A)$.
\begin{lm}\label{lm-FpFq}
For all $\Fq$-linear functors $\pi$ and $\rho$, there is an isomorphism, natural in $\pi$ and $\rho$:
\[\Tor_*^{k\otimes_\Fq \A}(\pi,\rho)\simeq \Tor_*^{\kATor}(\pi,\rho)\;.\]
\end{lm}
\begin{proof}
Let $\phi:k\otimes_\Fp\Fq\to k$ denote the surjective morphism of $(k,\Fq)$-bimodules such that $\phi(x\otimes y)=xy$. Restriction along the functor $\phi\otimes_\Fq\A: \kA\to k\otimes_\Fq\A$ yields a fully faithful functor $(k\otimes_\Fq \A)\Md\to \kA\Md$ hence an isomorphism $\pi\otimes_{k\otimes_\Fq\A}\rho\simeq \pi\otimes_{\kA}\rho$ for all $\Fq$-linear functors $\pi$ and $\rho$ by corollary \ref{cor-memechose2}. 
Moreover, $\phi$ admits a section (as a morphism of $(k,\Fq)$-bimodules) because $\Fq$ is a finite separable extension of $\Fp$. Hence for all objects $a$ in $\A$, the additive functor $k\otimes_\Fq\A(a,-)$ is a direct summand of the additive functor $k\otimes_\Fp\A(a,-)$. Thus, every projective resolution $P$ of $\rho$ in the category of $\Fq$-linear functors may be regarded as a projective resolution of $\rho$ in the category of additive functors. 
As a result we have
\[\Tor_*^{k\otimes_\Fq \A}(\pi,\rho)=H_*(\pi\otimes_{k\otimes_\Fq \A}P)\simeq 
H_*(\pi\otimes_{\kA}P)=\Tor_*^{\kA}(\pi,\rho)\;.\]
where the middle isomorphism is given by restriction along $\phi\otimes_\Fq\A$.
\end{proof}

We are now ready to prove the generalized comparison theorem.

\begin{proof}[Proof of theorem  \ref{thm-gen-comp}]

We set $\sigma={}^{(rs-rs^2)}\rho$ and $H=G^{(rs^2-rs)}$ for the sake of concision. Thus, $\rho^*G=\sigma^*H$. We also emphasize the dependence of $F^\dag$ on $\pi$ and $\sigma$ by setting:
\[F^\dag_{\pi,\sigma}(v):=\overline{F}\Big(\bigoplus_{0\le i< s^2}{}^{(ri)}\Hom_k\big(v,({}^{(-ri)}\pi)\otimes_{\kATor}\sigma\big)\,\Big)\;.\]
We fix an integer $n\gg 0$ ($n> \log_p(e/2)$ suffices) such that the canonical map
\begin{align}\Tor_*^{\gen}(F^\dag_{\pi,\sigma},H)\to \Tor_*^{\Gamma\Proj_k}((F^\dag_{\pi,\sigma})^{(n)},H^{(n)})
\label{eqn-cancomp}
\end{align}
is $e$-connected. Let 
\[\Psi_k:  \Tor_*^{k[\Proj_k]}(F^\dag_{\pi,\sigma},H)\to    \Tor_*^{\Gamma\Proj_k}((F^\dag_{\pi,\sigma})^{(n)},H^{(n)})\]
be the morphism given by restriction along $^{(-n)}-$ and by the restriction from ordinary functors to strict polynomial functors. Then $\Psi_k$ is the composition of the canonical map \eqref{eqn-cancomp} with $\Phi_k$, so that it suffices to prove that the composition $\Psi_k\circ\Theta^\dag_k\circ\Lambda_k$ is $e$-connected.

For this purpose we are going to realize the $\Tor$ vector spaces in play as homotopy groups of some simplicial vector spaces, and the maps $\Psi_k$, $\Theta^\dag_k$ and $\Lambda_k$ as the morphisms induced on the level of homotopy groups by some morphisms of simplicial vector spaces. To be more specific, we consider simplicial resolutions by direct sums of standard projectives 
\begin{align*}
\mathcal{H}\to H\;, && \mathcal{H}'\to H^{(n)}\;, && \varpi\to \pi\;, && \varrho\to \rho\;,
\end{align*}
respectively in the categories 
\begin{align*}
k[\Proj_k]\Md\;, && \Gamma\Proj_k\Md\;, && \Mdd (k\otimes_\Fq\A)\;, && (k\otimes_\Fq\A)\Md\;.
\end{align*}
Standard projectives in $ (k\otimes_\Fq\A)\Md$ are of the form $t\circ \A(a,-)$, and we have $^{(rs-rs^2)}(t\circ \A(a,-))\simeq t\circ \A(a,-)$, hence 
\[ \varsigma:={}^{(rs-rs^2)}\varrho\to {}^{(rs-rs^2)}\rho=\sigma\]
is also a simplicial resolution in $ (k\otimes_\Fq\A)\Md$.
Then it follows from lemmas \ref{lm-Whitehead-thm} and  \ref{lm-flat} that $\varsigma^*\mathcal{H}$ is a simplicial projective resolution of $\sigma^*H$ in $k[\A]\Md$, and that $(\varpi^{\oplus i})^*F$ is a simplicial (not projective) resolution of $(\pi^{\oplus i})^*F$ in $\Mdd k[\A]$ for all positive integers $i$. Therefore we have identifications of homotopy groups:
\begin{align*}
&\pi_*\left( \varpi^*F\otimes_{k[\A]}\varsigma^*\mathcal{H}\right)=\Tor_*^{k[\A]}(\pi^*F,\sigma^*H)\;,\\
&\pi_*\left( (\varpi^{\oplus s^2})^*F\otimes_{k[\A]}\varsigma^*\mathcal{H}\right)=\Tor_*^{k[\A]}((\pi^{\oplus s^2})^*F,\sigma^*H)\;,\\
&\pi_*\left( F_{\pi,\sigma}^\dag\otimes_{k[\A]}\mathcal{H}\right)=\Tor_*^{k[\Proj_k]}(F_{\pi,\sigma}'^\dag,H)\;,\\
&\pi_*\left( (F_{\pi,\sigma}^\dag)^{(n)}\otimes_{\Gamma\Proj_k}\mathcal{H}'\right)=\Tor_*^{k[\Proj_k]}((F_{\pi,\sigma}^\dag)^{(n)},H^{(n)})\;.
\end{align*}
Moreover, let $f:\mathcal{H}\to {\mathcal{H}'}^{(-n)}$ be a simplicial morphism in $k[\Proj_k]\Md$ lifting the identity morphism of $H$. Then the maps $\Lambda_k$, $\Theta_k^\dag$ and $\Psi_k$ are respectively induced by the morphisms of simplicial $k$-vector spaces:
\begin{align*}
&\varpi^*F\otimes_{k[\A]}\varsigma^*\mathcal{H}\xrightarrow[]{\widetilde{\Lambda_k}}  (\varpi^{\oplus s^2})^*F\otimes_{k[\A]}\varsigma^*\mathcal{H}\;,\\
&(\varpi^{\oplus s^2})^*F\otimes_{k[\A]}\varsigma^*\mathcal{H}\xrightarrow[]{\widetilde{\Theta_k^\dag}}F_{\varpi,\varsigma}^\dag\otimes_{k[\Proj_k]}\mathcal{H}\to F_{\pi,\sigma}^\dag\otimes_{k[\Proj_k]}\mathcal{H}\;,\\
&F_{\pi,\sigma}^\dag\otimes_{k[\Proj_k]}\mathcal{H}\xrightarrow[]{\id\otimes f} (F_{\pi,\sigma}^\dag)^{(n)(-n)}\otimes_{k[\Proj_k]}{\mathcal{H}'}^{(-n)}\xrightarrow[]{\widetilde{\Psi_k}} (F_{\pi,\sigma}^\dag)^{(n)}\otimes_{\Gamma\Proj_k}\mathcal{H}'\;.
\end{align*} 
Here, the unadorned simplicial morphism following $\widetilde{\Theta_k^\dag}$ on the second line is induced by the simplicial morphisms $\varpi\to \pi$ and $\varsigma\to \sigma$. The simplicial morphisms $\widetilde{\Lambda_k}$, $\widetilde{\Theta'_k}$ and $\widetilde{\Psi_k}$ are degreewise equal to the degree zero component of $\Lambda_k$, $\Theta'_k$ and $\Psi_k$, for example the component of $\widetilde{\Lambda_k}$ in simplicial degree $i$  is equal to 
$$\Tor_0^{k[\A]}(\varpi^*_iF,\varsigma^*_i\mathcal{H}_i)\xrightarrow[]{\Lambda_k} \Tor_0^{k[\A]}((\varpi_i^{\oplus s^2})^*F,\varsigma^*_i\mathcal{H}_i)\;.$$
(That these simplicial morphisms induce our maps $\Lambda_k$, $\Theta_k^\dag$ and $\Psi_k$ is obvious for the first one and the last one. For $\Theta^\dag_k$, this follows by the same reasoning as in the proof of theorem \ref{thm-Theta}).
We deduce that the comparison map $\Psi_k\circ\Theta^\dag_k\circ\Lambda_k$ equals the map induced on homotopy groups by the composition of the simplicial morphism:
\begin{align}
\left(\widetilde{\Psi_k}\circ (\id\otimes f)\circ \widetilde{\Theta^\dag_k}\circ \widetilde{\Lambda_k}\right)\;: \;\varpi^*F\otimes_{k[\A]}\varsigma^*\mathcal{H}\to (F_{\varpi,\varsigma}^\dag)^{(n)}\otimes_{\Gamma\Proj_k}\mathcal{H}'\label{eqn-morph1}
\end{align}
followed by the simplicial morphism induced by $\varpi\to \pi$ and $\varsigma\to \sigma$:
\begin{align}
(F_{\varpi,\varsigma}^\dag)^{(n)}\otimes_{\Gamma\Proj_k}\mathcal{H}'\to (F_{\pi,\sigma}^\dag)^{(n)}\otimes_{\Gamma\Proj_k}\mathcal{H}'
\label{eqn-morph2}
\end{align}

Now we are going to show that the two simplicial morphisms \eqref{eqn-morph1} and \eqref{eqn-morph2} are $e$-connected. 
Let us regard the source and the target of \eqref{eqn-morph1} as the diagonal of bisimplicial objects $\varpi_i^*F\otimes_{k[\A]}\varsigma^*_i\mathcal{H}_j$ and 
$(F_{\varpi_i,\varsigma_i}^\dag)^{(n)}\otimes_{\Gamma\Proj_k}\mathcal{H}_j'$ with bisimplicial degrees $(i,j)$. Then the simplicial morphism \eqref{eqn-morph1} actually comes from a bisimplicial morphism. Spectral sequences of bisimplicial $k$-modules as in \cite[IV section 2.2]{GoerssJardine} yield two spectral sequences:
\begin{align*}
& \mathrm{I}^1_{ij}=\Tor_j^{k[\A]}(\varpi_i^*F,\varsigma^*_i H)\Longrightarrow \pi_{i+j}\left(\varpi^*F\otimes_{k[\A]}\varsigma^*\mathcal{H} \right)\;,\\
& \mathrm{II}^1_{ij}=\Tor_j^{\Gamma\Proj_k}(F^\dag_{\varpi_i,\varsigma_i},H)\Longrightarrow \pi_{i+j}\left((F_{\varpi,\varsigma}^\dag)^{(n)}\otimes_{\Gamma\Proj_k}\mathcal{H}'\right)\;,
\end{align*}
And there is a morphism of spectral sequences $\mathrm{I}\to \mathrm{II}$ which coincides with the morphism \eqref{eqn-morph1} on the abutment, and with the map $\Psi_k\circ\Theta^\dag_k\circ\Lambda_k$ on the first page. 
By proposition \ref{prop-base-case}, the map $\Phi_k\circ\Theta^\dag_k\circ\Lambda_k$ is an isomorphism when $\pi$ and $\rho$ are direct sums of standard projectives, hence the morphism of spectral sequences is an isomorphism on the first page. Hence the simplicial morphism \eqref{eqn-morph1} is an isomorphism on the level of homotopy groups (hence $e$-connected).

Thus, it remains to prove that the simplicial morphism \eqref{eqn-morph2} is $e$-connected. The $\Tor$-vanishing hypothesis of theorem \ref{thm-gen-comp} and lemma \ref{lm-chgt-base-finitefield} imply that the maps $({}^{(-ri)}\varpi)\otimes_{\kATor}\varsigma \to ({}^{(-ri)}\pi)\otimes_{\kATor}\sigma$ are $e$-connected. Hence the map $F^\dag_{\varpi,\varsigma}\to F^\dag_{\pi,\sigma}$ is e-connected by lemma \ref{lm-Whitehead-thm}, hence the simplicial morphism \eqref{eqn-morph2} is $e$-connected by the usual bisimplicial spectral sequence argument.
\end{proof}

\subsection{Consequences of the generalized comparison theorem}\label{sec-special-cases}

We first prove theorem \ref{thm-intro-F-lin-case} from the introduction. We consider a simplification of the generalized comparison map \eqref{eqn-gencompmap}, namely for all strict polynomial functors $F$ and $G$ and all additive functors $\pi:\A^\op\to k\Md$ and $\rho:\A\to k\Md$ we consider the composition:
\begin{align}\Tor^{k[\A]}_*(\pi^*F, \rho^*G)\xrightarrow[]{\Theta_k}  \Tor_*^{k[\Proj_k]}(D_{\pi,\rho}^*F,G) \xrightarrow[]{\Phi_k}
\Tor_*^{\gen}(D_{\pi,\rho}^*F,G)\;,\label{eqn-comp-simplified}
\end{align}
where $\Theta_k$ is the auxiliary comparison map of section \ref{sec-gen-prelim-comparison}. Theorem \ref{thm-intro-F-lin-case} is a direct consequence of the following result.

\begin{thm}\label{thm-F-lin-case}
Let $k$ be an infinite perfect field of positive characteristic, containing a subfield $\FF$ and let $\A$ be an additive $\FF$-linear category. Let $\pi$ and $\rho$ be two $\FF$-linear functors from $\A$ to $k$-modules, respectively contravariant and covariant, and let $F$ and $G$ be two strict polynomial functors with degrees less or equal to the cardinal of $\FF$.
Assume furthermore that 
\[\Tor_i^{\kATor}(\pi,\rho)=0\quad \text{for $0<i<e$.}\]
Then the comparison map \eqref{eqn-comp-simplified} is $e$-connected.
\end{thm}
\begin{proof}
There are two cases. Assume first that $\FF$ is a finite field. Then we can apply theorem \ref{thm-gen-comp} with $s=1$. If $s=1$ the generalized comparison map \eqref{eqn-gencompmap} is equal to the simplified comparison map \eqref{eqn-comp-simplified} and the result follows. 

Now assume that $\FF$ is infinite. We claim that for all $\FF$-linear functors $\alpha,\beta:\A\to k\Md$, we have $\Ext^*_{\kATor}({}^{(i)}\alpha,{}^{(j)}\beta)=0$  for $i\ne j$. Indeed, this can be proved by repeating the argument of the vanishing lemma \ref{lm-homogeneity-vanish} in the case of the $k$-category $\kA$, or alternatively by combining this vanishing lemma \ref{lm-homogeneity-vanish} with theorem \ref{thm-main-additive}. By lemma \ref{lm-iso-dual}, this implies that 
\begin{align}\Tor_*^{\kATor}({}^{(i)}\pi,{}^{(j)}\rho)=0 \text{ for $i\ne j$.}
\label{eqn-vanishing}
\end{align}
Thus we may apply the generalized comparison theorem \ref{thm-gen-comp} with $q=p$ (i.e. $\Fq$ is the prime field), and with an integer $s$ such that $p^s$ is greater or equal to the degrees of $F$ and $G$. 
In this situation, the expression of $F^\dag$ simplifies because of equation \eqref{eqn-vanishing}. Namely if we let $\alpha={}^{(s-s^2)}\pi$ and $\beta={}^{(s-s^2)}\rho$ then we have: 
\[F^\dag=\overline{F}^{(s^2-s)}\circ D_{\alpha,\beta}\;.\]
Moreover one readily checks that the composite map 
\[\Theta^\dag_k\circ \Lambda_k\;:\; \Tor_*^{k[\A]}(\pi^*F,\rho^*G)\to \Tor_*^{k[\Proj_k]}(D_{\alpha,\beta}^*(F^{(s^2-s)}), G^{(s^2-s)})\]
is equal to the map $\Theta_k$ relative to $\alpha={}^{(s-s^2)}\pi$ and $\beta={}^{(s-s^2)}\rho$. 

Now we consider the following diagram, in which the composition of functors is omitted, $T_*$ stands for $\Tor_*$, and we use the following notations $x:=s^2-s$, $D=D_{\pi,\rho}$, $D'=D_{\alpha,\beta}$. The Frobenius twist functor $^{(x)}-:k\Md\to k\Md$ is isomorphic to the extension of scalars along the morphism of fields $k\to k$, $\lambda\mapsto \lambda^{p^x}$, hence we have a canonical isomorphism $\mathrm{can}:{}^{(x)}D'\simeq D^{(x)}$, and the isomorphisms $(*)$ in this diagram are induced by this canonical isomorphism.
\[
\begin{tikzcd}
T_*^{k[\A]}(\overline{F}\pi,\overline{G}\rho)\ar{r}{\Theta_k} \ar{d}{\Theta_k}
& T_*^{k[\Proj_k]}(\overline{F}^{(x)}D',G^{(x)})\ar{r}{\Phi_k}\ar{d}{\res^{^{(x)}-}}[swap]{\simeq}
& T_*^{\gen}(\overline{F}{}^{(x)}D',G^{(x)})\ar{d}{\res^{^{(x)}-}}[swap]{\simeq}\\
T_*^{k[\Proj_k]}(\overline{F}D,G) \ar{r}{\simeq}[swap]{(*)}\ar{d}{\Phi_k}
& T_*^{k[\Proj_k]}(\overline{F}{}^{(x)}D'{}^{(-x)},G)\ar{r}{\Phi_k}
& T_*^{\gen}(\overline{F}{}^{(x)}D'{}^{(-x)},G)\\
T_*^{\gen}(\overline{F}D,G)\ar{rru}{\simeq}[swap]{(*)}& &
\end{tikzcd}
\]
The diagram is commutative. To be more specific, the upper left square of the diagram commutes by the base change property of proposition \ref{prop-chgt-base}, the upper right square and the triangle commute by naturality of $\Phi_k$. As explained above, the composite map corresponding to the upper row is $e$-connected by theorem \ref{thm-gen-comp}. Therefore, the composite given by the first column is also $e$-connected. But this composite is nothing but the simplified comparison map \eqref{eqn-comp-simplified}. This finishes the proof of the theorem.
\end{proof}

\begin{cor}\label{cor-infinite-perfect-field}
Let $k$ be an infinite perfect field of positive characteristic, let $^\vee-:\Proj_k^{\op}\to \Proj_k$ denote the $k$-linear duality functor $^\vee-=\Hom_k(-,k)$, and let $F^\vee$ denote the composition $F\circ {}^\vee-$. The map 
\begin{align*}
&\Phi_k:\Tor_*^{k[\Proj_k]}(F^\vee,G)\to \Tor^\gen_*(F^\vee,G)
\end{align*}
is an isomorphism for all strict polynomial functors $F$ and $G$.
\end{cor}
\begin{proof}
The Eilenberg-Watts theorem gives an equivalence of categories between the category of additive functors $\Proj_R\to k\Md$ and the category of $(R,k)$-bimodules. Under this equivalence, an $(R,k)$-bimodule $M$ corresponds to the functor $-\otimes_R M$. Therefore, if we let $R=k$, $\pi(v)=v^\vee$ and $\rho(v)=v$, we obtain: 
\[\Tor_*^{_{\,k\!}(\Proj_k)}(\pi,\rho)\simeq \Tor_*^{k\otimes_\mathbb{Z}k}(k,k)=\mathrm{HH}_*(k) \;.\]
But $k\otimes_\mathbb{Z} k=k\otimes_\Fp k$ and every field extension of $\Fp$ is a filtered colimit of $\Fp$-subalgebras which are smooth and essentially of finite type. Thus $\mathrm{HH}_*(k)$ is an exterior algebra over the $k$-vector space of Kähler forms $\Omega^1(k/\Fp)$ by the Hochschild-Kostant-Rosenberg theorem \cite[Cor 2.13]{Hubl}. Since $k$ is perfect, $\Omega^1(k/\Fp)$ is zero, so that $\mathrm{HH}_i(k)=k$ for $i=0$, and zero otherwise. Hence $D_{\pi,\rho}\simeq {}^\vee-$, and $\Phi_k\circ \Theta_k$ is an isomorphism by theorem \ref{thm-F-lin-case}. 

Furthermore, the categories of $k$-linear functors $\Mdd\Proj_k$ and $\Proj_k\Md$ are respectively equivalent to $\Mdd k$ and $k\Md$. Under these equivalences of categories, the functors $\pi$ and $\rho$ correspond to $k$. Hence $\Tor_i^{\Proj_k}(\pi,\rho)$ equals $k$ if $i=0$ and zero otherwise. Thus $\Theta_k$ is an isomorphism by theorem \ref{thm-Theta}.

Since $\Phi_k\circ\Theta_k$ and $\Theta_k$ are both isomorphisms, so is $\Phi_k$.
\end{proof}

\subsection{Comparison of $\Ext$}\label{subsec-Ext-comp}

We now indicate how the results of the previous section can be dualized to compare $\Ext$. Let $G$ and $K$ be two strict polynomial functors. 
We let $\Phi_k$ be the unique map making the following diagrams commute for all $i$ and all $n\gg 0$, where the vertical isomorphism on the left hand side is the canonical isomorphism, and the one on the right hand side is given by restriction along the Frobenius twist $^{(-n)}-$:
\[
\begin{tikzcd}
\Ext^i_{\Gamma\Proj_k}(G^{(n)},K^{(n)})\ar{d}{\simeq}\ar{r}&
\Ext^i_{k[\Proj_k]}(G^{(n)},K^{(n)})\ar{d}[swap]{\simeq}\\
\Ext^i_{\gen}(G,K)\ar[dashed]{r}{\Phi_k}&
\Ext^i_{k[\Proj_k]}(G,K)
\end{tikzcd}.
\] 
\begin{cor}\label{cor-infinite-perfect-field-Ext}
Let $k$ be an infinite perfect field of positive characteristic. For all strict polynomial functors $G$ and $K$ the comparison map 
$$\Phi_k:\Ext^i_{\gen}(G,K)\to 
\Ext^i_{k[\Proj_k]}(G,K)$$
is a graded isomorphism.
\end{cor}
\begin{proof}
By a standard spectral sequence argument, the proof reduces to the case where $K$ is a standard injective, hence when $K=\Hom_k(F^\vee,k)$, where $F^\vee$ is the precomposition of a standard projective $F$ in $\Gamma\Proj_k\Md$ by the duality functor $^\vee-=\Hom_k(-,k)$. In this latter case, $\Phi_k$ is an isomorphism because proposition \ref{prop-Tor-Ext} shows that it is dual to the isomorphism $\Phi_k$ of corollary \ref{cor-infinite-perfect-field}.  
\end{proof}

Similarly, one can dualize theorem \ref{thm-F-lin-case}. To be more specific, given two additive functors $\rho,\sigma:\A\to k\Md$ and a $k$-vector space $v$, we let 
$$T_{\rho,\sigma}(v)= \Hom_{\kATor}(\rho,\sigma)\otimes v\;.$$
Then for all strict polynomial functors $G$ and $K$ we have a map 
$$\Theta_k:\Ext^*_{k[\Proj_k]}(G,T^*_{\rho,\sigma}K)\to \Ext^*_{k[\A]}(\rho^*G,\sigma^*K)$$
induced by restriction along $\rho$ and by the canonical evaluation morphism $\mathrm{ev}:\Hom_{\kATor}(\rho,\sigma)\otimes \rho\to \sigma$. 
\begin{cor}\label{cor-thm-F-lin}
Let $k$ be an infinite perfect field of positive characteristic, containing a subfield $\FF$ and let $\A$ be an additive $\FF$-linear category. Let $\rho,\sigma:\A\to \Proj_k$ be two $\FF$-linear functors such that $\Hom_{\kATor}(\rho,\sigma)$ is finite-dimensional. 
Assume that  
\[\Ext^i_{\kATor}(\rho,\sigma)=0\quad \text{for $0<i<e$.}\]
Then for all strict polynomial functors $G$ and $K$ with degrees less or equal to the cardinal of $\FF$, the graded map
$$\Theta_k\circ\Phi_k: \Ext^*_{\gen}(G,T^*_{\rho,\sigma}K)\to \Ext^*_{k[\A]}(\rho^*G,\sigma^*K)$$
is $e$-connected.
\end{cor}
\begin{proof}
In this proof, we let $^\vee-=\Hom_k(-,k)$ and we omit the composition operator for functors, e.g. if $F$ is a strict polynomial functor, $^\vee F^\vee$ stands for $({}^\vee-)\circ F\circ ({}^\vee-)$. 

We first prove the result when $K={}^\vee F^\vee$ for some $F$ in $\Gamma\Proj_k\Md$ and $\sigma={}^\vee\pi$ for some additive functor $\pi:\A^\op\to k\Md$. Let  
$\xi: T_{\sigma,\rho}\to {}^\vee D_{\pi,\rho}$ be
the morphism of functors whose component $\xi_v$ at a vector space $v$ is given by the composition
$$\xi_v\;:\;v\otimes\Hom_{\kATor}(\rho,{}^\vee\pi)\xrightarrow[]{\simeq} v\otimes {}^{\vee}(\pi\otimes_{\kATor}\rho)\to {}^\vee\Hom_k(v, \pi\otimes_{\kATor}\rho)$$
where the first map is provided by lemma \ref{lm-iso-dual} and the second map is the canonical map $\mathrm{can}:v\otimes {}^\vee w\to {}^\vee\Hom_k(v,w)$ such that $\mathrm{can}(x\otimes f)(\phi)=f(\phi(x))$, and which is an isomorphism if $v$ is finite dimensional.
One readily checks that the composition 
$$ T_{\rho,\sigma}\rho\xrightarrow[]{\xi\rho} {}^\vee D_{\pi,\rho}\rho\xrightarrow[]{{}^\vee\theta_\FF} {}^{\vee}\pi=\sigma$$
equals the canonical evaluation map $\mathrm{ev}$. The finite dimensionality hypotheses on the values of $\rho$ and $\sigma$ and on $\Hom_{\kATor}(\rho,\sigma)$ respectively imply that: 
\begin{enumerate}
\item[i)] 
$\xi:T_{\sigma,\rho}\to {}^\vee D_{\pi,\rho}$ is an isomorphism,
\item[ii)] $\sigma^*(F^\vee)=F^{\vee}\sigma=F\pi=\pi^*F$,
\item[iii)]$^{\vee}{}^\vee D_{\pi,\rho}$ identifies with $D_{\pi,\rho}$.
\end{enumerate}
Hence we have a commutative diagram
\[
\begin{tikzcd}
\Ext^*_{k[\Proj_k]}(G, {}^\vee F^\vee T_{\rho,\sigma})\ar{r}{\rho^*}&
\Ext^*_{k[\A]}(G\rho, {}^\vee F^\vee T_{\rho,\sigma}\rho)\ar{r}{{}^\vee F^\vee(\mathrm{ev})}&
\Ext^*_{k[\A]}(G\rho, {}^\vee F^\vee \sigma)\\
{}^\vee\Tor_*^{k[\Proj_k]}(F^\vee T_{\rho,\sigma},G)\ar{u}[swap]{\alpha}{\simeq}\ar{r}{{}^\vee\res_\rho}\ar{d}{F^\vee(\xi)}[swap]{\simeq}&
{}^\vee\Tor_*^{k[\A]}(F^\vee T_{\rho,\sigma}\rho,G\rho)
\ar{u}[swap]{\alpha}{\simeq}\ar{r}{F^\vee(\mathrm{ev})}\ar{d}{F^\vee(\xi\rho)}[swap]{\simeq}
&{}^\vee\Tor_*^{k[\A]}(F^\vee\sigma,G\rho)\ar{u}[swap]{\alpha}{\simeq}
\\
{}^\vee\Tor_*^{k[\Proj_k]}(FD_{\pi,\rho},G)\ar{r}{{}^\vee\res_\rho}&
{}^\vee\Tor_*^{k[\A]}(FD_{\pi,\rho}\rho,G\rho) \ar{r}{F(\theta_\FF)}
& {}^\vee\Tor_*^{k[\A]}(F\pi,G\rho)\ar[equal]{u}
\end{tikzcd}
\]
from which we deduce that the graded map $\Theta_k\circ\Phi_k$ fits into a commutative square 
\begin{equation}
\begin{tikzcd}[column sep=large]
\Ext^*_{\gen}(G, T_{\rho,\sigma}^*({}^\vee F^\vee) )\ar{r}{\Theta_k\circ\Phi_k}\ar{d}{\simeq}&
\Ext^*_{k[\A]}(\rho^*G, \sigma^*({}^\vee F^\vee))\ar{d}{\simeq}\\
{}^\vee\Tor_*^{\gen}(D_{\pi,\rho}^*F,G)\ar{r}&{}^\vee\Tor_*^{k[\A]}(\pi^*F,\rho^*G)
\end{tikzcd}\label{cd-dual}
\end{equation}
where the bottom arrow is dual to the comparison map \eqref{eqn-comp-simplified} of theorem \ref{thm-F-lin-case}. Since $\Ext^*_{\kATor}(\rho,\sigma)\simeq {}^\vee\Tor_*^{\kATor}(\pi,\rho)$, we deduce from the latter theorem that this bottom map is $e$-connected, hence $\Theta_k\circ \Phi_k$ is $e$-connected.

The case $K={}^\vee F^\vee$ proves corollary \ref{cor-thm-F-lin} for all strict polynomial functors $K$ with finite dimensional values, in particular for the standard injectives. For an arbitrary $K$ one can consider an injective resolution and the result follows by a standard spectral sequence argument.
\end{proof}

With the same strategy, one can also dualize theorem \ref{thm-gen-comp}. Given a strict polynomial functor $K$, we denote by $K^\ddag$ the strict polynomial functor such that 
$$K^\ddag(v)=\overline{K}\left(\bigoplus_{0\le i<s^2} {}^{(ri)}\big(v\otimes \Hom_{_\kATor}({}^{(rs-rs^2)}\rho, {}^{(-ri)}\sigma)\big)\right)\;.$$
One defines a comparison map in the same fashion as the map of corollary \ref{cor-thm-F-lin}:
\begin{align}
\Ext^*_\gen(G^{(rs^2-rs)},K^\ddag)\to \Ext^*_{k[\A]}(\rho^*G,\sigma^*K)\;.
\label{eq-gen-comp-dual}
\end{align}
The proof of the following corollary is similar to the proof of corollary \ref{cor-thm-F-lin} and is left to the reader.
\begin{cor}\label{cor-thm-gencomp}
Let $k$ be an infinite perfect field of characteristic $p$, containing a finite field $\Fq$ of cardinal $q=p^r$. 
Let $\A$ be a small additive category, let $\rho,\sigma:\A\to \Proj_k$ be two additive functors such that $\Hom_{\kATor}(\rho,\sigma)$ is finite dimensional. Assume that there are positive integers $s$ and $e$ such that 
\[\Ext^j_{\kATor} \left(   {}^{(rs-rs^2)}\rho,  {}^{(ri)}\sigma  \right)=0\]
for $0< j<e$ and $0\le i<s^2$. Assume further that $\A$ is $\Fq$-linear,  that $\rho$ and $\sigma$ are $\Fq$-linear. Then for all strict polynomial functors $G$ and $K$ of degrees less or equal to $q^s$, the map \eqref{eq-gen-comp-dual} is $e$-connected.
\end{cor}

\begin{rk}\label{rk-limit-pb}
The finite dimensionality hypotheses on the values of $\rho$, $\sigma$ and on $\Hom_{\kATor}(\rho,\sigma)$ are necessary in the proof of corollary \ref{cor-thm-F-lin} in order that i), ii) and iii) are satisfied. Without them, we would not obtain a commutative square \eqref{cd-dual} with vertical \emph{isomorphisms}. Similarly, the finite dimensionality hypotheses are needed for the proof of corollary \ref{cor-thm-gencomp}. Instead of dualizing, one could try to prove corollaries \ref{cor-thm-F-lin} and \ref{cor-thm-gencomp} by a direct approach, following the same strategy as the proofs of theorems \ref{thm-gen-comp} and \ref{thm-F-lin-case}. However, such a direct approach seems to raise inextricable problems with (co)limits.
\end{rk}

\section{Applications}\label{sec-applic}
The goal of this section is to prove applications of the generalized comparison theorem, or to be more specific, of the corollaries \ref{cor-infinite-perfect-field} and \ref{cor-infinite-perfect-field-Ext} which deal with the very specific case $\A=\Proj_k$. We first generalize the computations of \cite{FFSS} over an infinite perfect field. We also generalize some results of \cite{FF} to infinite perfect fields. Finally, the most important applications are probably theorems \ref{thm-comp-GL} and \ref{thm-comp-OSp}. These theorems are the analogues for classical groups over infinite perfect fields of the main result of Cline Parshall Scott and van der Kallen \cite{CPSVdK} which compares the cohomology of and algebraic group with the cohomology of its underlying discrete group.
Throughout the section, $k$ is a field of positive characteristic $p$.

\subsection{A sample of functor homology computations}\label{subsec-sample}
Many computations of generic $\Ext$ can be found in the literature. Thus, the isomorphism of corollary \ref{cor-infinite-perfect-field-Ext}
provides many concrete $\Ext$-computations in $k[\Proj_k]\Md$ over an infinite field. We briefly illustrate this fact here. 

We first point out that computations of generic $\Ext$ between strict polynomial functors are insensitive to field extensions. To be more specific, if $k\to L$ is any field extension, the base change formula \cite[Section 2.7]{SFB} yields an isomorphism
\[\Ext^*_{\gen,k}(F,G)\otimes L \simeq \Ext^*_{\gen,L}(F_L,G_L)\]
where the generic extensions on the left hand side are computed in the $k$-category $\Gamma\Proj_k\Md$, while the generic extensions on the right hand side are computed in the $L$-category $\Gamma_L(\Proj_L)\Md$. The functors $F_L$ and $G_L$ obtained by base change from $F$ and $G$ are usually easy to compute, e.g. if $F$ is the $d$-th symmetric power over $k$ then $F_L$ is the $d$-th symmetric power over $L$. In particular, all the computations of generic $\Ext$ over finite fields actually hold over arbitrary fields $k$ of positive characteristic, and can therefore be converted into computations in $k[\Proj_k]\Md$ by corollary \ref{cor-infinite-perfect-field-Ext} when $k$ is infinite and perfect.
This is the case of the computations of generic $\Ext$ given in \cite[Thm 5.8]{FFSS} (which are established in \cite{TouzeBar} by different methods, without spectral sequences). To be more specific, let $C^*$ be a graded coalgebra in $k[\Proj_k]\Md$ and let $A^*$ be a graded algebra in $k[\Proj_k]\Md$. We consider the trigraded vector space
$$\mathrm{E}^*(C^*,A^*):=\bigoplus_{i,d,e\ge 0}\Ext^i_{k[\Proj_k]}(C^d,A^e)$$
equipped with the algebra structure given by convolution: 
$$\mathrm{E}^i(C^d,A^e)\otimes \mathrm{E}^j(C^f,A^g)\xrightarrow[]{\cup}\mathrm{E}^{i+j}(C^d\otimes C^f,A^e\otimes A^g)\to \mathrm{E}^{i+j}(C^{d+f},A^{e+g})\;.$$
By letting $s\to \infty$ in \cite[Thm 15.1]{TouzeBar} and by applying corollary \ref{cor-infinite-perfect-field-Ext} one obtains the following infinite field version of the computations of \cite[Thm 6.3]{FFSS}.
\begin{cor}\label{cor-FFSS-infinite}
Let $k$ be an infinite perfect field of positive characteristic $p$, and let $r$ be a nonnegative integer. Let $V_{s,r}$ denote the trigraded vector space with homogeneous basis $(e_i)_{i\ge 0}$ where each $e_i$ is placed in tridegree $(2ip^r+sp^r-s,1,p^r)$.
Then we have isomorphisms of trigraded algebras:
\begin{align*}
&\mathrm{E}^*(\Gamma^{*(r)},S^*)\simeq S(V_{0,r})\;,
&&
\mathrm{E}^*(\Gamma^{*(r)},\Lambda^*)\simeq \Lambda(V_{1,r})
\;, \\
&\mathrm{E}^*(\Lambda^{*(r)},S^*)\simeq \Lambda(V_{0,r})\;,
&&
\mathrm{E}^*(\Lambda^{*(r)},\Lambda^*)\simeq \Gamma(V_{1,r})
\;, \\
&\mathrm{E}^*(S^{*(r)},S^*)\simeq \Gamma(V_{0,r})\;,
&&
\mathrm{E}^*(\Gamma^{*(r)},\Gamma^*)\simeq \Gamma(V_{2,r})
\;. 
\end{align*}
\end{cor}

The approach of \cite{TouzeBar} relies on a formula computing extensions between twisted strict polynomial functors, see \cite{Chalupnik}, \cite{TouzeUnivNew} and \cite{Touze-Survey}. Namely, if $v$ is a finite-dimensional vector space and $G$ is a strict polynomial functor, we let $G_v$ be the strict polynomial functor `with parameter $v$' defined by $G_v(-):=G(v\otimes -)$. If $v$ is graded, then $G_v$ inherits a grading. It is the unique grading natural with respect to $G$ and $v$, which coincides with the usual grading on symmetric powers of a graded vector space see \cite[Section 2.5]{TouzeENS} and \cite[Section 4.2]{Touze-Survey}. Let $E_r$ denote the finite-dimensional graded vector space
$E_r=\Ext_{\Gamma^{p^r}\Proj_k}^*(I^{(r)},I^{(r)})$
which equals $k$ in degrees $2i$ for $0\le i<p^r$ and which is zero in the other degrees. Then we have a graded isomorphism, where the degree on the right hand side is obtained by totalizing the $\Ext$-degree with the degree of the functor $G_{E_r}$ (that is, if $G_{E_r}^j$ is the component of degree $j$ then the summand $\Ext^i_{\Gamma\Proj_k}(F,G_{E_r}^j)$ is placed in degree $i+j$):
\[\Ext^*_{\Gamma\Proj_k}(F^{(r)},G^{(r)})\simeq \Ext^*_{\Gamma\Proj_k}(F,G_{E_r})\;.\]

We can extend the parametrization of a strict polynomial functor $G$ to infinite-dimensional graded vector spaces $v$ by letting 
$G_v:=\colim G_u$, where the colimit is taken over the poset of all finite-dimensional graded vector spaces $u\subset v$. By taking the colimit over $r$ in the previous isomorphism, and by using corollary \ref{cor-infinite-perfect-field-Ext} we obtain the following result (in which no Frobenius twist appear in the $\Ext$ of the right hand side).
\begin{cor}\label{cor-calcul-kProjk}
Let $k$ be an infinite perfect field of positive characteristic. Let $E_\infty$ be the graded vector space equal to $k$ in even degrees and to $0$ in odd degrees. There is a graded isomorphism, natural with respect to the strict polynomial functors $F$ and $G$, and where the degree on the right hand side is computed by totalizing the $\Ext$-degree with the degree of the functor $G_{E_\infty}$:
\[\Ext^*_{k[\Proj_k]} (F,G)\simeq \Ext^*_{\Gamma\Proj_k}(F,G_{E_\infty})\;.\]
\end{cor}

\subsection{Bifunctor cohomology}\label{subsec-bifunctor}
The words `bifunctor cohomology' are sometimes used \cite{FF,TvdK} to denote the Hochschild cohomology of $k[\Proj_k]$ or $\Gamma^d\Proj_k$.
The study of bifunctor cohomology was initiated in \cite{FF} for a finite field $k$. In this subsection, we extend two of the main results of \cite{FF} to infinite perfect fields of positive characteristic, in propositions \ref{pr-infinite-perfect-field-HH} and \ref{pr-Dja-HH}. We also provide some explicit bifunctor homology computations in corollary \ref{cor-calculbif}.

Let $\K$ denote either $k[\Proj_k]$ or $\Gamma^d\Proj_k$. The bifunctor cohomology of $B\in \K^\op\otimes\K\Md$ is defined as the extensions
\[\mathrm{HH}^*(\K,B):= \Ext^*_{\K^\op\otimes\K}(\K,B)\]
where the first argument in the $\Ext$ is the bifunctor given by homomorphisms in $\K$. Thus, if $\mathrm{gl}(v,w):=\Hom_k(v,w)$, then $\K(v,w)=k[\mathrm{gl}(v,w)]$ in the case of ordinary functors and $\K(v,w)=\Gamma^d(\mathrm{gl}(v,w))$ in the case of strict polynomial functors. If $B$ has \emph{separable type}, that is, if $B(v,w)=\Hom_k(F(v),G(w))$ for some functors $F$ and $G$, we have isomorphisms natural with respect to $F$ and $G$ \cite[Prop 2.2]{FF}:
\begin{align}\mathrm{HH}^*(\K,B)\simeq \Ext^*_{\K}(F,G)\;.\label{eqn-sep}
\end{align}
These isomorphisms can often be used to reduce questions regarding bifunctor cohomology to questions regarding functor cohomology, especially for strict polynomial bifunctors since the standard injectives of the category have separable type.

Just like for functors of one variable, we have a forgetful functor 
\[k[\Proj_k^\op]\otimes k[\Proj_k]\to \Gamma^d\Proj_k^\op\otimes\Gamma^d\Proj_k\]
induced by restriction along the functor $\gamma^d\otimes\gamma^d$, where $\gamma^d$ is defined in example \ref{ex-ordvsstrict}. If $k$ is infinite, this forgetful functor is fully faithful.
By restricting extensions along $\gamma^d\otimes\gamma^d$ and by pulling back along the morphism of functors $\gamma^d(\mathrm{gl}): k[\mathrm{gl}]\to\Gamma^d\mathrm{gl}$ we obtain a restriction map:
\[\mathrm{HH}^*(\Gamma^{d}\Proj_k,B)\to \mathrm{HH}^*(k[\Proj_k],B)\;.\]
For all $r\ge 0$, we let $\Phi_k'$ denote the composition of this restriction map together with the isomorphism induced by restriction along the $(-r)$-th Frobenius twist and by the canonical isomorphism $k[\mathrm{gl}]^{(-r)}=k[\mathrm{gl}^{(-r)}]\simeq k[{}^{(-r)}\mathrm{gl}]=k[\mathrm{gl}]$:
\begin{align}
\Phi_k':\mathrm{HH}^*(\Gamma^{dp^r}\Proj_k,B^{(r)})\to \mathrm{HH}^*(k[\Proj_k],B^{(r)})\xrightarrow[]{\simeq} \mathrm{HH}^*(k[\Proj_k],B).\label{eqn-phiprimek}
\end{align}
The next proposition is the analogue of \cite[Thm 7.6]{FF} for infinite perfect fields.
\begin{pr}\label{pr-infinite-perfect-field-HH}
Let $k$ be an infinite perfect field of positive characteristic $p$ and let $B$ be a strict polynomial bifunctor in $\Gamma^d\Proj_k^\op\otimes \Gamma^d\Proj_k\Md$. Then the map \eqref{eqn-phiprimek} is $2p^r$-connected.
\end{pr}
\begin{proof}
By considering a coresolution of $B$ by products of standard injectives, we reduce the proof to the case where $B$ is a standard injective, hence when $B=\Hom_k(F,G)$. In this case, the map $\Phi_k'$ identifies with the composition
\[\Ext^*_{\Gamma^dp^r\Proj_k}(F^{(r)},G^{(r)})\to \Ext^*_\gen(F,G)\xrightarrow[]{\Phi_k} \Ext^*_{k[\Proj_k]}(F,G)\]
where the first map is the canonical inclusion (which is $2p^r$-connected by proposition-definition \ref{pdef-gen-Ext}) and the isomorphism $\Phi_k$ of corollary \ref{cor-infinite-perfect-field-Ext}. Whence the result.
\end{proof}

Proposition \ref{pr-infinite-perfect-field-HH} should be seen as a way to obtain explicit bifunctor cohomology over $k[\Proj_k]$, in the 
 spirit of section \ref{subsec-sample}. 
 For example, we obtain the following computation by letting $r\to \infty$ in \cite[Thm 1]{TouzeCRAS} and applying proposition \ref{pr-infinite-perfect-field-HH}. 
\begin{cor}\label{cor-calculbif}
Let $E_\infty$ denote the graded vector space which equals $k$ in even degrees and $0$ in odd degrees and consider $k[\Si_d]$ as a vector space placed in degree zero. Let the symmetric group $\Si_d$ act on $E_\infty^{\otimes d}$ by permuting the factors of the tensor product, and by conjugation on $k[\Si_d]$. There is an isomorphism of graded vector spaces
\[\mathrm{HH}^*(k[\Proj_k],S^d\mathrm{gl})\simeq (E_\infty^{\otimes d})\otimes_{\Si_d}k[\Si_d]\;.\]
\end{cor}

We finish our section on bifunctor cohomology by describing its relation  with the cohomology of $\GL_n(k)$. This relation provides a motivation for computing bifunctor cohomology over $k[\Proj_k]$, and we will also need it in the proof of theorem \ref{thm-comp-GL} in section \ref{subsec-ratdiscrete}. We will need the next lemma, which follows from the $p$-local Hurewicz theorem \cite[Thm 1.8.1]{Neisendorfer} and the unique $p$-divisibility of the homotopy groups of  $B\GL(k)^+$ \cite[Lm 5.2 and Cor 5.5]{Kratzer}.
\begin{lm}\label{lm-vanish-HGL}
Let $k$ be a perfect field of positive characteristic $p$. The mod $p$ homology of $\GL_\infty(k)$ is zero in positive degrees.
\end{lm}

%
If $B$ is an object of $k[\Proj_k^\op]\otimes k[\Proj_k]\Md$, then $B(k^n,k^n)$ is endowed with an action of $\GL_n(k)$. Namely, an element $g\in \GL_n(k)$ acts as $B(g^{-1},g)$ on $B(k^n,k^n)$. For example $\GL_n(k)$ acts by conjugation on $\mathrm{gl}(k^n,k^n)=\End_k(k^n)$. The identity of $k^n$ is invariant under conjugation, hence we have a morphism of representations
\[f_n:k\to k[\mathrm{gl}](k^n,k^n)=k[\End_k(k^n)]\]
defined by $f(\lambda)=\lambda f(\mathrm{id}_{k^n})$. Then evaluation on $k^n$ and pullback along $f_n$ yields a graded map
\begin{align}
\mathrm{HH}^*(k[\Proj_k],B)\to \rmH^*(\GL_n(k),B(k^n,k^n))\;.\label{eq-comp-GLn}
\end{align}
The next proposition \ref{pr-Dja-HH} and its corollary \ref{cor-Dja-Ext} generalize Suslin's comparison result \cite[Thm A.1]{FFSS} and its extension to bifunctors \cite[Thm 7.4]{FF} from finite fields to arbitrary perfect fields of positive characteristic. It is a consequence of the stable $K$-theory computations of Scorichenko \cite{Sco}, which are reformulated in terms of stable homological calculations in \cite{DjaR}, the homological stabilization result \cite[Thm 5.11]{RWW}, together with the vanishing lemma \ref{lm-vanish-HGL}.

\begin{pr}\label{pr-Dja-HH}
Let $k$ be a perfect field of characteristic $p$. 
Assume that $B$ is polynomial of degree $d$ with finite-dimensional values.
Then the comparison map \eqref{eq-comp-GLn} is $\frac{1}{2}(n-1-d)$-connected.
\end{pr}
\begin{proof}

Let $B^\sharp$ denote the Kuhn dual of $B$, that is the bifunctor defined by $B^\sharp(v,w)={}^\vee B({}^{\vee}v,{}^\vee w)$ where $^{\vee}v$ refers to the dual of a $k$-vector space $v$. 
By using proposition \ref{prop-Tor-Ext}, the symmetry of $\Tor$ (that is $\Tor_*^\K(F,G)\simeq \Tor_*^{\K^\op}(G,F)$) and the fact that $^{\vee}-:\Proj_k^\op\to \Proj_k$ is an equivalence of categories, the proof reduces to show that evaluation on $k^n$ and restriction along $f_{n}$ yields a $\frac{1}{2}(n-1-d)$-connected map:
\[\rmH_*(\GL_n(k),B^\sharp(k^n,k^n))\to \Tor_*^{k[\Proj_k]}(k[\mathrm{gl}],B^\sharp)\;.\] 
In order to achieve this, we compare the homology of $\GL_n(k)$ with the homology of $\GL_\infty(k)$. Namely 
we let $B^\sharp(k^\infty,k^\infty)$ denote representation of $\GL_\infty(k)$ obtained by taking the colimit of the $B^\sharp(k^n,k^n)$. 
Let $\rho_n$ denote the composition
\[\rmH_*(\GL_n(k),B^\sharp(k^n,k^n)) \to \rmH_*(\GL_n(k),B^\sharp(k^\infty,k^\infty))\to \rmH_*(\GL_\infty(k),B^\sharp(k^\infty,k^\infty))\]
where the first map is induced by the canonical inclusion $B^\sharp(k^n,k^n)\to B^\sharp(k^\infty,k^\infty)$ and the second one is given by restriction along $\GL_n(k)\hookrightarrow \GL_\infty(k)$. We have a commutative square, in which the vertical isomorphism on the right hand side is the base change isomorphism of \cite[Thm 14.2]{Mi72} and the bottom arrow $(\dag)$ is induced by the map $f_\infty: \mathbb{Z}\to \mathbb{Z}[\Proj_k](k^\infty,k^\infty)$ and by evaluation on $k^\infty$.
\[
\begin{tikzcd}
\rmH_*(\GL_n(k),B^\sharp(k^n,k^n))\ar{d}[swap]{\rho_n}\ar{r}& 
\Tor_*^{k[\Proj_k^\op\times\Proj_k]}(k[\Proj_k],B^\sharp)\ar{d}{\simeq}[swap]{}\\
\rmH_*(\GL_\infty(k),B^\sharp(k^\infty,k^\infty))\ar{r}{(\dag)}&
\Tor_*^{\mathbb{Z}[\Proj_k^\op\times\Proj_k]}(\mathbb{Z}[\Proj_k],B^\sharp)
\end{tikzcd}\;.
\]
The map $(\dag)$ is an isomorphism by \cite[Thm 5.6]{DjaR} and the vanishing lemma \ref{lm-vanish-HGL}, and $\rho_n$ is $\frac{1}{2}(n-1-d)$-connected by \cite[Thm 5.11]{RWW}. Whence the result.
\end{proof}

\begin{cor}\label{cor-Dja-Ext}
Let $k$ be a perfect field of characteristic $p$. Assume either that (i) both $F$ and $G$ are polynomial functors of degree less or equal to $d$ with finite-dimensional values, or that (ii) both $F$ and $G$ are strict polynomial functors of degree less or equal to $d$. Evaluation on $k^n$ yields a $\frac{1}{2}(n-1-2d)$-connected map
\[\mathrm{ev_n}:\Ext^*_{k[\Proj_k]}(F,G)\to \Ext^*_{\GL_n(k)}(F(k^n),G(k^n))\;.\] 
\end{cor}
\begin{proof}
Assume (ii). Take a resolution or $F$ by direct sums of standard projectives in $\Gamma\Proj_k\Md$ and a coresolution of $G$ by products of standard injectives in $\Gamma\Proj_k\Md$. Then by a standard spectral sequence argument we can restrict ourselves to the case where $F$ is a standard projective and $G$ is a standard injective, in particular to the case where $F$ and $G$ have finite-dimensional values. Morover strict polynomial functors of degree less or equal to $d$ are polynomial of degree less or equal to $d$, see remark \ref{rk-ordvsstrict}. Hence it suffices to prove the corollary under hypothesis (i).

Assume (i) and let $B$ denote the bifunctor $B(v,w)=\Hom_k(F(v),G(w))$. There is a commutative diagram whose horizontal maps are the canonical isomorphisms
\[
\begin{tikzcd}
\mathrm{HH}^*(k[\Proj_k],B)\ar{r}{\simeq}\ar{d}{\eqref{eq-comp-GLn}}&\Ext^*_{k[\Proj_k]}(F,G)\ar{d}{\mathrm{ev}_n}\\
\rmH^*(\GL_n(k),B(k^n,k^n))\ar{r}{\simeq}& \Ext^*_{\GL_n(k)}(F(k^n),G(k^n))
\end{tikzcd}.
\]
Hence the result follows from proposition \ref{pr-Dja-HH}. (Note that $B$ has degree less or equal to $2d$).
\end{proof}

\subsection{Orthogonal and symplectic cohomology} Bifunctor cohomology and its relation to the cohomology of general linear groups has an analogue for symplectic and orthogonal groups that we now describe. Here we assume that $k$ has odd characteristic $p$. 

Assume that $G=\orth_{n,n}(k)$ or $G=\Sp_{2n}(k)$. 
We associate to $G$ a `characteristic functor' $X:\Proj_k\to \Proj_k$, namely $X=S^2$ in the orthogonal case  $X=\Lambda^2$ in the symplectic case. We define an analogue of bifunctor cohomology as follows. Let $F$ be an object of $k[\Proj_k]\Md$ or of $\Gamma^{d}\Proj_k\Md$, we set:
\begin{align*}
&\rmH^*_X(k[\Proj_k],F)=\Ext^*_{k[\Proj_k]}(k[X],F), \\
&\rmH^*_X(\Gamma^d\Proj_k,F)=\begin{cases}
\Ext^*_{\Gamma^d\Proj_k}(\Gamma^{d/2}\circ X,F) & \text{ if $d$ is even,}\\
0 &\text{ if $d$ is odd.}
\end{cases}
\end{align*}
By restricting extensions along the functor $\gamma^d: k[\Proj_k]\to \Gamma^{d}\Proj_k$ and by pulling back along the morphism of functors $\gamma^{d/2}(X):k[X]\to \Gamma^{d/2}X$ we obtain a restriction map (which is the zero map if $d$ is odd):
\[\rmH^*_X(\Gamma^d\Proj_k,F)\to \rmH^*_X(k[\Proj_k],F)\;.\]
For all $r\ge 0$, we let $\Phi_{k,X}'$ denote the composition of this restriction map together with the isomorphism induced by restriction along the $(-r)$-th Frobenius twist and by the canonical isomorphism $k[X]^{(-r)}=k[X^{(-r)}]\simeq k[{}^{(-r)}X]=k[X]$:
\begin{align}
\Phi_{k,X}':\rmH_X^*(\Gamma^{dp^r}\Proj_k,F^{(r)})\to \rmH^*_X(k[\Proj_k],F^{(r)})\xrightarrow[]{\simeq} \rmH^*_X(k[\Proj_k],F).\label{eqn-phiprimekosp}
\end{align}
\begin{pr}\label{pr-infinite-perfect-field-Hosp}
Let $k$ be an infinite perfect field of odd positive characteristic $p$ and let $F$ be a $d$-homogeneous strict polynomial functor. The map \eqref{eqn-phiprimekosp} is $2p^r$-connected.
\end{pr}
\begin{proof}
Since $p$ is odd, $X$ is a direct summand of the second tensor power functor $\otimes^2$, hence $\Phi'_{k,X}$ is a retract of $\Phi'_{k,\otimes^2}$. Thus we have to show that  $\Phi'_{k,\otimes^2}$ is $2p^r$-connected. 
We achieve this by reformulating the problem in terms of bifunctor cohomology. Let $B$ be the object of $\Gamma^d(\Proj_k^\op\times\Proj_k)\Md$ such that $B(v,w)=F({}^\vee v\oplus w)$ where $^\vee v$ denotes the dual of the $k$-vector space $v$. 
We have a finite direct sum decomposition
\[\Gamma^d(\Proj_k^\op\times\Proj_k)\Md=\bigoplus_{i+j=d}\Gamma^i\Proj_k^\op\otimes\Gamma^j\Proj_k\Md\]
hence $B$ decomposes into a direct sum of homogeneous summands $B=\bigoplus_{i+j=d}B^{i,j}$.

We claim that for any bifunctor $B^{i,j}\in \Gamma^i\Proj_k^\op\otimes\Gamma^j\Proj_k\Md$, the bifunctor cohomology $\mathrm{HH}^*(k[\Proj_k],B^{i,j})$ is zero if $i\ne j$. Indeed, by considering an injective coresolution of $B$ in $\Gamma^i\Proj_k^\op\otimes\Gamma^j\Proj_k\Md$ it suffices to prove the result when $B$ is a standard injective, hence when $B$ has separable type. In that case, the question reduces to the vanishing of $\Ext$ in $k[\Proj_k]\Md$ between a $i$-homogeneous strict polynomial functor and a $j$-homogeneous strict polynomial functor hence the claim follows from the vanishing lemma \ref{lm-homogeneity-vanish}.

Sum-diagonal adjunction (as in example \ref{ex-sum-diagonal}) and restriction along the equivalence of categories $k[\Proj_k^\op]\otimes k[\Proj_k]\simeq k[\Proj_k]\otimes k[\Proj_k]$ given by the duality functor $^\vee -:\Proj_k^\op\to \Proj_k$ yields an isomorphism 
\[\rmH^*_{\otimes^2}(k[\Proj_k],F)\simeq \mathrm{HH}^*(k[\Proj_k],B)\]
Our claim implies that for $d$ odd, the right hand side is zero, hence that the comparison map \eqref{eqn-phiprimekosp} is an isomorphism. 

Assume now that $d$ is even. We have a commutative diagram where the bottom isomorphism is described above, and the top isomorphism is its analogue for strict polynomial functors (induced by sum-diagonal adjunction, restriction along the equivalence of categories $\Gamma^{dp^r}(\Proj_k^\op\times \Proj_k)\simeq \Gamma^{dp^r}(\Proj_k\times \Proj_k)$ provided by the duality functor and projection onto the summand $B^{d,d}$ of $B$):
\[
\begin{tikzcd}
\rmH^*_{\otimes^2}(\Gamma^{dp^r}\Proj_k,F^{(r)})\ar{d}{\Phi'_{k,\otimes^2}}\ar{rr}{\simeq}
&&  \mathrm{HH}^*(k[\Proj_k],B^{d,d})\ar{d}{\Phi'_k}\\
\rmH^*_{\otimes^2}(k[\Proj_k],F)\ar{r}{\simeq}&\mathrm{HH}^*(k[\Proj_k],B)\ar[equal]{r}&\mathrm{HH}^*(k[\Proj_k],B^{d,d})
\end{tikzcd}\;.
\]
Hence $\Phi'_{k,\otimes^2}$ is $2p^r$-connected by proposition \ref{pr-infinite-perfect-field-HH}.
\end{proof}

Now we explain the relation between the cohomology groups $\rmH^*_X(k[\Proj_k],F)$ and the cohomology of the symplectic and orthogonal group. We first need the following vanishing result. We gratefully thank Baptiste Calm\`es for helping us with the literature  relative to hermitian $K$-theory. 
\begin{lm}\label{lm-vanish-HOSp}
Let $k$ be a perfect field of odd characteristic $p$. Then the mod $p$ homology of the groups $\Sp_\infty(k)$ and $\orth_{\infty,\infty}(k)$ is zero in positive degrees.
\end{lm}
\begin{proof}
Let $G=\Sp_\infty(k)$ or $\orth_{\infty,\infty}(k)$. By the universal coefficient theorem, it is equivalent to prove that $H_i(G,\mathbb{Z})$ is uniquely $p$-divisible for $i>0$. 
If $A$ is an abelian group, we let $A[1/2]$ denote the tensor product $A\otimes_{\mathbb{Z}}\mathbb{Z}[1/2]$. Since $p$ is odd, $A$ is uniquely $p$-divisible if and only if $A[1/2]$ is uniquely $p$-divisible. And since $\mathbb{Z}[1/2]$ is flat we have $H_*(G,\mathbb{Z})[1/2]=H_*(G,\mathbb{Z}[1/2])$. Thus the statement of the lemma is equivalent to $H_i(G,\mathbb{Z}[1/2])$ being uniquely $p$-divisible for $i>0$.

Since $k$ is a field of odd characteristic, the Witt groups $W(k)$ are an $\mathbb{F}_2$-vector space \cite[Chap. 2, Thm 6.4]{Scharlau}. Thus by \cite[Thm 3.18]{Karoubi}
$H_*(G,\mathbb{Z}[1/2])$ is equal to $T_*(k)$, that is, to the homology of a space $\mathcal{C}(k)$ which is a retract of the localized classifying space $({B\mathcal{P}_k'}^+)_{(2)}$, see \cite[p. 253]{Karoubi} for the latter point. Since the integral homology of $({B\mathcal{P}_k'}^+)_{(2)}$ is equal to $H_*({B\mathcal{P}_k'}^+,\mathbb{Z})[1/2]$, the lemma will be proved if we can prove that ${B\mathcal{P}_k'}^+$ has uniquely $p$-divisible positive integral homology groups.

But ${B\mathcal{P}_k'}^+$ has the weak homotopy type of $K_{0}(k)\times B\GL_\infty(k)^+$, hence its integral homology is direct sum of copies of the integral homology of $B\GL_\infty(k)^+$. Since $k$ is perfect, these integral homology groups are uniquely $p$-divisible in positive degrees by lemma \ref{lm-vanish-HGL} and by the universal coefficient theorem.
\end{proof}

\begin{rk}
Instead of relying on the results of \cite{Karoubi}, one could prove the lemma by using the formula of \cite[Rk 7.8]{Schlichting}, which says that after tensoring by $\mathbb{Z}[1/2]$, the hermitian $K$-theory (hence \cite[App A]{Schlichting} the homotopy groups of $BG^+$ for $G=\Sp_\infty(k)$ or $\orth_{\infty,\infty}(k)$) is the direct sum of a term computed from the $K$-theory of $k$ and a term given by Balmer's Witt groups of $k$ tensored with $\mathbb{Z}[1/2]$. 
\end{rk}
 
For all functors $F$ in $k[\Proj_k]\Md$, the vector space $F(k^{2n})$ has a natural action of $G$, where $g\in G$ acts as $F(g)$. The quadratic form on $k^{2n}$ used to define $G$ yields an invariant $\omega\in X(k^{2n})$ under the action of $G$, hence we have a $G$-equivariant map 
\[f_{2n,X}:k\to k[X](k^{2n})=k[X(k^{2n})]\] 
such that $f_{2n,X}(\lambda)=\lambda[\omega]$. Evaluation on $k^{2n}$ and pullback along $f_{2n,X}$ yields a graded map 
\begin{align}
\rmH^*_X(k[\Proj_k],F)\to \rmH^*(G,F(k^{2n})).
\label{eq-comp-OSp}
\end{align}
The following proposition and its corollary are proved exactly in the same fashion as proposition \ref{pr-Dja-HH} and corollary \ref{cor-Dja-Ext}, relying on the stable homology computations of \cite[Cor 5.4]{DjaR}, the homological stabilization result \cite[Thm 5.15]{RWW} and the vanishing lemma \ref{lm-vanish-HOSp}. 
\begin{pr}\label{pr-Dja-HOSp}
Let $k$ be a perfect field of odd characteristic $p$. 
Assume that $F$ is polynomial of degree $d$ with finite-dimensional values. Then the comparison map \eqref{eq-comp-OSp} is $\frac{1}{2}(n-2-d)$-connected.
\end{pr}
\begin{cor}\label{cor-Dja-HOSp}
Let $k$ be a perfect field of odd characteristic $p$. 
Assume that $F$ is a strict polynomial functor of degree less or equal to $d$. Then the comparison map \eqref{eq-comp-OSp} is $\frac{1}{2}(n-2-d)$-connected.
\end{cor}

\subsection{Rational and discrete cohomology of classical groups}\label{subsec-ratdiscrete}
Let $G$ be an algebraic group over an infinite field $k$, let $\mathbf{Mod}_G$ denote the category of all $k$-linear representations of the discrete group $G$, and let $\mathbf{Rat}_G$ denote the full subcategory of $\mathbf{Mod}_G$ on the rational representations as in \cite{Jantzen}. 
Extensions between two rational representations $V$ and $W$ can be computed in $\mathbf{Rat}_G$ or $\mathbf{Mod}_G$. In the sequel, we let 
\begin{align*}
&\EExt_G^*(V,W):=\Ext^*_{\mathbf{Rat}_G}(V,W)\;,&&&\HH^*(G,k):=\EExt_G^*(k,k)\;,\\
&\Ext_G(V,W):=\Ext^*_{\mathbf{Mod}_G}(V,W)\;,&&&\rmH^*(G,k):=\Ext_G^*(k,k)\;.
\end{align*}
There is a canonical morphism:
\begin{align*}
\EExt^*_G(V,W)\to \Ext_G^*(V,W) 
\end{align*}
which is far from being an isomorphism in general. 
An important difference between the source and the target of the canonical morphism is the behaviour of Frobenius morphisms. Namely assume that $G$ is one of the classical groups $\GL_n(k)$, $\Sp_{2n}(k)$ or $\orth_{n,n}(k)$ (and if $k$ has characteristic $\ne 2$ in the latter case), and let let $\phi:G\to G$, $[a_{ij}]\mapsto [a_{ij}^p]$, denote the morphism of algebraic groups induced by the Frobenius endomorphism of $k$, and let $V^{[r]}$ denote the restriction of $V$ along $\phi^r$. We have a commutative ladder whose horizontal arrows are induced by restriction along $\phi$:
\[
\begin{tikzcd}[column sep=small]
\EExt^i_G(V,W)\ar{r}\ar{d}&\cdots \ar{r}&\EExt^i_G(V^{[r]},W^{[r]})\ar{d}\ar{r}&\EExt^i_G(V^{[r+1]},W^{[r+1]})\ar{r}\ar{d}&\cdots\\
\Ext^i_G(V,W)\ar{r}&\cdots \ar{r}&\Ext^i_G(V^{[r]},W^{[r]})\ar{r}&\Ext^i_G(V^{[r+1]},W^{[r+1]})\ar{r}&\cdots\;.
\end{tikzcd}
\]
Assume that $k$ is perfect. Then  $\phi$ has an inverse $\phi^{-1}([a_{ij}])=[a_{ij}^{-p}]$ so that the morphisms in the bottom row are all isomorphisms. However $\phi$ has no inverse in the sense of algebraic groups, so this argument does not apply to the top row. Instead, it is known \cite[II 10.14]{Jantzen} that the morphisms in the top row are all injective (and that they are isomorphisms for $r\gg 0$ only). Hence, if $k$ is perfect the ladder yields canonical maps:
\begin{align}
\Phi_{k,G}:\EExt^*_{G}(V^{[r]},W^{[r]})\to \Ext^*_G(V,W)\;.
\label{eqn-compare-gen}
\end{align}

Embed $G=\GL_n(k)$, $\Sp_{2n}(k)$ or $\orth_{n,n}(k)$ in the multiplicative monoïd of matrices $M_m(k)$ in the usual way (here $m=n$ for $\GL_n(k)$ and $m=2n$ for in the case of the symplectic or orthogonal groups). Then a finite-dimensional representation $V$ of  is called \emph{polynomial of degree less or equal to $d$} if it is the restriction to $G$ of a representation of $M_m(k)$, such that the coordinate maps of the action morphism
$$\rho_V: M_m(k)\to \End_k(V)\simeq M_{\dim V}(k)$$
are polynomials of degree less or equal to $d$ of the $m^2$ entries of $[a_{ij}]\in M_m(k)$. An infinite-dimensional representation $V$ is \emph{polynomial of degree less or equal to $d$} if every element of $V$ is contained in a finite-dimensional subrepresentation which is polynomial of degree less or equal to $d$.

\begin{thm}\label{thm-comp-GL}
Let $k$ be an infinite perfect field of positive characteristic $p$, let $G=\GL_n(k)$, let $V$ and $W$ be two polynomial representations of degree less or equal to $d$, and let $r$ be a nonnegative integer.
Assume that $n\ge \max\{dp^r, 4p^r+2d+1\}$.
Then the comparison map in equation \eqref{eqn-compare-gen} is $2p^r$-connected.
\end{thm}

\begin{proof}
Evaluation on $k^n$ yields an exact functor $\mathrm{ev}_n:k[\Proj_k]\Md\to \mathbf{Mod}_{\GL_n(k)}$, which restricts to a exact functor
$\mathrm{ev}_n:\Gamma\Proj_k\Md\to \mathbf{Rat}_{\GL_n(k)}$. Since $n\ge d$ we know from \cite[Lm 3.4]{FS} that there exists two strict polynomial functors $F$ and $G$ of degree less or equal to $d$ such that $V\simeq F(k^n)$ and $W=G(k^n)$, hence such that $V^{[r]}\simeq F^{(r)}(k^n)$ and $W^{[r]}\simeq G^{(r)}(k^n)$ for all $r\ge 0$. We consider the following commutative diagram:
\[
\begin{tikzcd}
\EExt_{\GL_n(k)}^*(V^{[r]},W^{[r]})\ar{d}[swap]{\Phi_{k,\GL_n(k)}}& 
\Ext_{\Gamma\Proj_k}^*(F^{(r)},G^{(r)})\ar{d}{\Phi_k'}\ar{l}{\mathrm{ev}_n}
\\
\Ext^*_{\GL_n(k)}(V,W)&
\Ext^*_{k[\Proj_k]}(F,G)\ar{l}{\mathrm{ev}_n}
\end{tikzcd}\;.
\]
in which $\Phi'_k$ is the composition
\[\Ext_{\Gamma\Proj_k}(F^{(r)},G^{(r)})\to \Ext^*_{\gen}(F,G)\xrightarrow[\simeq]{\Phi_k}\Ext_{k[\Proj_k]}^*(F,G)\]
where the second map is the isomorphism of corollary \ref{cor-infinite-perfect-field-Ext} and the first one is the canonical inclusion. This canonical inclusion is $2p^r$-connected by proposition-definition \ref{pdef-gen-Ext}, hence $\Phi'_k$ is $2p^r$-connected. Moreover, since $n\ge dp^r$ the top horizontal map is an isomorphism  by \cite[Cor 3.13]{FS}. Finally, the bottom horizontal map is $2p^r$-connected by corollary \ref{cor-Dja-Ext}. Thus $\Phi_{k,\GL_n(k)}$ is $2p^r$-connected.
\end{proof}

\begin{rk}[Generic extensions of $\GL_n(k)$]\label{rk-generic}
Assume that $V$ and $W$ are finite dimensional polynomial representations of degree $d$ of $\GL_n(k)$.
It is known \cite[II.10.16]{Jantzen} that the maps $\EExt^*_{\GL_n(k)}(V^{[r]},W^{[r]})\to \EExt^*_{\GL_n(k)}(V^{[r+1]},W^{[r+1]})$ are injective, their colimit is called the \emph{generic extensions} between $V$ and $W$. We denote it by $\EExt^*_{\gen}(V,W)$. The comparison map $\Phi_{k,\GL_n(k)}$ factors through generic extensions:
\[
\begin{tikzcd}[column sep=large]
\EExt^*_{\gen}(V,W)\ar[dashed]{rd}{\Phi_\gen}&\\
\EExt^*_{\GL_n(k)}(V^{[r]},W^{[r]})\ar[hook]{u}\ar{r}{\Phi_{k,\GL_n(k)}}&\Ext^*_{\GL_n(k)}(V,W)
\end{tikzcd}\;.
\]
The main theorem of \cite{CPSVdK} or rather the general linear group version given in \cite[Thm 7.3]{FS} imply that the vertical arrow is $((p-1)r +2)$-connected provided that (i) $V$ and $W$ are defined over $\mathbb{F}_p$ and (ii) $k$ is a big enough finite field (with respect to $r$, $V$ and $W$). By base change \cite[I.4.13]{Jantzen}, the vertical map is also an isomorphism when condition (ii) is replaced by: (ii') $k$ is an infinite field. Moreover, every finite dimensional polynomial representation has a filtration whose associated graded object is defined over $\Fp$, namely its Jordan-H\"older filtration. Hence, condition (i) can be removed by inspecting the long exact sequences associated to the Jordan-H\"older filtration. 
Thus we can state a version of theorem \ref{thm-comp-GL} in terms of generic cohomology, at the price of a worse connectivity bound. Namely, the comparison map $\Phi_\gen$ above is $((p-1)r+2)$-connected provided that $n\ge \max\{dp^r, 4p^r+2d+1\}$.
\end{rk}

Theorem \ref{thm-comp-GL} has an analogue for orthogonal and symplectic groups. Here we take $V=k$ hence $V^{[r]}=k$ and the comparison map \eqref{eqn-compare-gen} can be rewritten as a map
\begin{align}
\Phi_{k,G}:\rmH^*(G,W^{[r]})\to \HH^*(G,W)\;.
\label{eqn-compare-osp}
\end{align}

\begin{thm}\label{thm-comp-OSp}
Let $k$ be an infinite perfect field of odd characteristic $p$, let $G=\Sp_{2n}(k)$ or $\orth_{n,n}(k)$ and let $W$ be a polynomial representation of degree less or equal to $d$. Assume that $2n\ge \max\{dp^r, 8p^r+4+2d\}$.
Then the comparison map \eqref{eqn-compare-osp} is $2p^r$-connected.
\end{thm}
\begin{proof}
We proceed in the same way as in the proof of theorem \ref{thm-comp-OSp}. We know that $W=F(k^{2n})$ for some strict polynomial functor of degree less or equal to $d$. Furthermore, if $F_i$ is the $i$-homogeneous component of $F$ then $W=\bigoplus_{0\le i\le d}F_i(k^{2n})$, and since the source and the target of $\Phi_{k,G}$ are additive with respect to $W$, it suffices to prove the isomorphism when $F$ is homogeneous of degree (less or equal to) $d$.

We have a commutative diagram: 
\[
\begin{tikzcd}
\HH^*(G,W^{[r]})\ar{d}[swap]{\Phi_{k,G}}& 
\rmH^*_X(\Gamma^{dp^r}\Proj_k,F^{(r)})\ar{d}{\Phi_{k,X}'}\ar{l}
\\
\rmH^*(G,W)&
\rmH^*_X(k[\Proj_k],F)\ar{l}
\end{tikzcd}.
\]
To be more specific, the bottom horizontal map of the diagram is the comparison map of corollary \ref{cor-Dja-HOSp} hence it is $2p^r$-connected. The top horizontal map has a similar definition, namely it is zero if $d$ is odd, and if $d$ is even it is induced by evaluation on $k^{2n}$ and pullback along the $G$-equivariant morphism $f_{2n,X}':k\to \Gamma^{d/2}(X(k^{2n}))$ such that $f_{2n,X'}(\lambda)=\lambda\omega^{\otimes d/2}$, where $\omega\in X(k^{2n})$ is the invariant element associated to the quadratic form defining $G$. This top horizontal arrow is an isomorphism by \cite[Thm 3.17]{TouzeClassical} or \cite[Thm 3.24]{TouzeClassical}. Moreover $\Phi'_k$ is $2p^r$-connected by proposition \ref{pr-infinite-perfect-field-Hosp}. The connectivity of $\Phi_{k,G}$ follows.
\end{proof}

\begin{rk}
Theorem \ref{thm-comp-OSp} can be reformulated in terms of generic cohomology in the same fashion as we explained it for $\GL_n(k)$ in remark \ref{rk-generic}.
\end{rk}

\bibliographystyle{amsplain}

\bibliography{biblio-DT-JEMS}

\providecommand{\bysame}{\leavevmode\hbox to3em{\hrulefill}\thinspace}
\providecommand{\MR}{\relax\ifhmode\unskip\space\fi MR }
\providecommand{\MRhref}[2]{%
  \href{http://www.ams.org/mathscinet-getitem?mr=#1}{#2}
}
\providecommand{\href}[2]{#2}
\begin{thebibliography}{10}

\bibitem{Cartan}
\emph{{S\'eminaire Henri Cartan: Alg\`ebres d'Eilenberg-MacLane et homotopie.
  7e ann\'ee 1954/55. 2e \'ed., revue et corrig\'ee}}, {\'Ecole Normale
  sup\'erieure. Paris: Secr\'etariat math\'ematique (Hektograph.)}, 1956.

\bibitem{ABW}
Kaan Akin, David~A. Buchsbaum, and Jerzy Weyman, \emph{Schur functors and
  {S}chur complexes}, Adv. in Math. \textbf{44} (1982), no.~3, 207--278.
  \MR{658729}

\bibitem{Brown}
Kenneth~S. Brown, \emph{Cohomology of groups}, Graduate Texts in Mathematics,
  vol.~87, Springer-Verlag, New York, 1994, Corrected reprint of the 1982
  original. \MR{1324339}

\bibitem{CE}
Henri Cartan and Samuel Eilenberg, \emph{Homological algebra}, Princeton
  Landmarks in Mathematics, Princeton University Press, Princeton, NJ, 1999,
  With an appendix by David A. Buchsbaum, Reprint of the 1956 original.
  \MR{1731415}

\bibitem{Chalupnik}
Marcin Cha\l~upnik, \emph{Derived {K}an extension for strict polynomial
  functors}, Int. Math. Res. Not. IMRN (2015), no.~20, 10017--10040.
  \MR{3455857}

\bibitem{CPSVdK}
E.~Cline, B.~Parshall, L.~Scott, and Wilberd van~der Kallen, \emph{Rational and
  generic cohomology}, Invent. Math. \textbf{39} (1977), no.~2, 143--163.
  \MR{439856}

\bibitem{DjaR}
Aur\'elien Djament, \emph{Sur l'homologie des groupes unitaires \`a
  coefficients polynomiaux}, J. K-Theory \textbf{10} (2012), no.~1, 87--139.
  \MR{2990563}

\bibitem{Dja-Ext-Pol}
\bysame, \emph{Groupes d'extensions et foncteurs polynomiaux}, J. Lond. Math.
  Soc. (2) \textbf{92} (2015), no.~1, 63--88. \MR{3384505}

\bibitem{DTSchwartz}
Aur\'elien Djament and Antoine Touz\'{e}, \emph{Finitude homologique des
  foncteurs et applications}, Preprint available on
  https://hal.archives-ouvertes.fr/hal-03432809, to appear in {\em Transactions
  of the AMS}.

\bibitem{DTV}
Aur\'elien Djament, Antoine Touz\'e, and Christine Vespa,
  \emph{D\'ecompositions à la steinberg sur une cat\'egorie additive}, 2019,
  ArXiv:1904.09190, to appear in \emph{Ann. Sci. \'Ec. Norm. Sup\'er.}

\bibitem{EML}
Samuel Eilenberg and Saunders Mac~Lane, \emph{On the groups {$H(\Pi,n)$}. {II}.
  {M}ethods of computation}, Ann. of Math. (2) \textbf{60} (1954), 49--139.
  \MR{0065162 (16,391a)}

\bibitem{FF}
Vincent Franjou and Eric~M. Friedlander, \emph{Cohomology of bifunctors}, Proc.
  Lond. Math. Soc. (3) \textbf{97} (2008), no.~2, 514--544. \MR{2439671}

\bibitem{FFSS}
Vincent Franjou, Eric~M. Friedlander, Alexander Scorichenko, and Andrei Suslin,
  \emph{General linear and functor cohomology over finite fields}, Ann. of
  Math. (2) \textbf{150} (1999), no.~2, 663--728. \MR{1726705}

\bibitem{FLS}
Vincent Franjou, Jean Lannes, and Lionel Schwartz, \emph{Autour de la
  cohomologie de {M}ac {L}ane des corps finis}, Invent. Math. \textbf{115}
  (1994), no.~3, 513--538. \MR{1262942 (95d:19002)}

\bibitem{FPmaclane}
Vincent Franjou and Teimuraz Pirashvili, \emph{On the {M}ac {L}ane cohomology
  for the ring of integers}, Topology \textbf{37} (1998), no.~1, 109--114.
  \MR{1480880}

\bibitem{FranjouPira}
\bysame, \emph{Stable {$K$}-theory is bifunctor homology (after {A}.
  {S}corichenko)}, Rational representations, the {S}teenrod algebra and functor
  homology, Panor. Synth\`eses, vol.~16, Soc. Math. France, Paris, 2003,
  pp.~107--126. \MR{2117530}

\bibitem{FS}
Eric~M. Friedlander and Andrei Suslin, \emph{Cohomology of finite group schemes
  over a field}, Invent. Math. \textbf{127} (1997), no.~2, 209--270.
  \MR{1427618}

\bibitem{GKRW}
Soren Galatius, Alexander Kupers, and Oscar Randal-Williams,
  \emph{$e_\infty$-cells and general linear groups of infinite fields}, 2020.

\bibitem{GoerssJardine}
Paul~G. Goerss and John~F. Jardine, \emph{Simplicial homotopy theory}, Modern
  Birkh\"{a}user Classics, Birkh\"{a}user Verlag, Basel, 2009, Reprint of the
  1999 edition. \MR{2840650}

\bibitem{Green}
James~A. Green, \emph{Polynomial representations of {${\rm GL}_{n}$}}, Lecture
  Notes in Mathematics, vol. 830, Springer-Verlag, Berlin-New York, 1980.
  \MR{606556}

\bibitem{Hubl}
Reinhold H\"{u}bl, \emph{Traces of differential forms and {H}ochschild
  homology}, Lecture Notes in Mathematics, vol. 1368, Springer-Verlag, Berlin,
  1989. \MR{995670}

\bibitem{Jantzen}
Jens~Carsten Jantzen, \emph{Representations of algebraic groups}, second ed.,
  Mathematical Surveys and Monographs, vol. 107, American Mathematical Society,
  Providence, RI, 2003. \MR{2015057}

\bibitem{Karoubi}
Max Karoubi, \emph{Th\'{e}orie de {Q}uillen et homologie du groupe orthogonal},
  Ann. of Math. (2) \textbf{112} (1980), no.~2, 207--257. \MR{592291}

\bibitem{Kratzer}
Ch. Kratzer, \emph{{$\lambda $}-structure en {$K$}-th\'{e}orie alg\'{e}brique},
  Comment. Math. Helv. \textbf{55} (1980), no.~2, 233--254. \MR{576604}

\bibitem{Ku-adv}
Nicholas~J. Kuhn, \emph{Generic representation theory of finite fields in
  nondescribing characteristic}, Adv. Math. \textbf{272} (2015), 598--610.
  \MR{3303242}

\bibitem{LL}
Michael Larsen and Ayelet Lindenstrauss, \emph{Topological {H}ochschild
  homology of algebras in characteristic {$p$}}, J. Pure Appl. Algebra
  \textbf{145} (2000), no.~1, 45--58. \MR{1732287}

\bibitem{LLbis}
\bysame, \emph{Topological {H}ochschild homology and the condition of
  {H}ochschild-{K}ostant-{R}osenberg}, Comm. Algebra \textbf{29} (2001), no.~4,
  1627--1638. \MR{1853116}

\bibitem{ML}
Saunders Mac~Lane, \emph{Categories for the working mathematician}, second ed.,
  Graduate Texts in Mathematics, vol.~5, Springer-Verlag, New York, 1998.
  \MR{1712872}

\bibitem{MLHom}
Saunders MacLane, \emph{Homology}, first ed., Springer-Verlag, Berlin-New York,
  1967, Die Grundlehren der mathematischen Wissenschaften, Band 114.
  \MR{0349792}

\bibitem{Mirzaii}
Behrooz Mirzaii, \emph{Homology of {${\rm GL}_n$} over infinite fields outside
  the stability range}, J. Pure Appl. Algebra \textbf{226} (2022), no.~5, Paper
  No. 106916, 33. \MR{4328647}

\bibitem{Mi72}
Barry Mitchell, \emph{Rings with several objects}, Advances in Math. \textbf{8}
  (1972), 1--161. \MR{0294454}

\bibitem{Neisendorfer}
Joseph Neisendorfer, \emph{Algebraic methods in unstable homotopy theory}, New
  Mathematical Monographs, vol.~12, Cambridge University Press, Cambridge,
  2010. \MR{2604913}

\bibitem{PiraZnZ}
T.~Pirashvili, \emph{On the topological {H}ochschild homology of {${\bf
  Z}/p^k{\bf Z}$}}, Comm. Algebra \textbf{23} (1995), no.~4, 1545--1549.
  \MR{1317414}

\bibitem{PiraSpectral}
Teimuraz Pirashvili, \emph{Spectral sequence for {M}ac {L}ane homology}, J.
  Algebra \textbf{170} (1994), no.~2, 422--428. \MR{1302848}

\bibitem{Pira-Pan}
\bysame, \emph{Introduction to functor homology}, Rational representations, the
  {S}teenrod algebra and functor homology, Panor. Synth\`eses, vol.~16, Soc.
  Math. France, Paris, 2003, pp.~1--26. \MR{2117526}

\bibitem{PiraWald}
Teimuraz Pirashvili and Friedhelm Waldhausen, \emph{Mac {L}ane homology and
  topological {H}ochschild homology}, J. Pure Appl. Algebra \textbf{82} (1992),
  no.~1, 81--98. \MR{1181095}

\bibitem{RWW}
Oscar Randal-Williams and Nathalie Wahl, \emph{Homological stability for
  automorphism groups}, Adv. Math. \textbf{318} (2017), 534--626. \MR{3689750}

\bibitem{Riehl}
Emily {Riehl}, \emph{{Category theory in context.}}, Mineola, NY: Dover
  Publications, 2016 (English).

\bibitem{SamSn}
Steven~V. Sam and Andrew Snowden, \emph{Gr\"obner methods for representations
  of combinatorial categories}, J. Amer. Math. Soc. \textbf{30} (2017), no.~1,
  159--203. \MR{3556290}

\bibitem{Scharlau}
Winfried Scharlau, \emph{Quadratic and {H}ermitian forms}, Grundlehren der
  Mathematischen Wissenschaften [Fundamental Principles of Mathematical
  Sciences], vol. 270, Springer-Verlag, Berlin, 1985. \MR{770063}

\bibitem{Schlichting}
Marco Schlichting, \emph{Hermitian {$K$}-theory, derived equivalences and
  {K}aroubi's fundamental theorem}, J. Pure Appl. Algebra \textbf{221} (2017),
  no.~7, 1729--1844. \MR{3614976}

\bibitem{Sco}
Alexander Scorichenko, \emph{Stable {K}-theory and functor homology over a
  ring}, Thesis, Evanston.

\bibitem{Suslin-Kexcision}
A.~A. Suslin, \emph{Excision in integer algebraic {$K$}-theory}, vol. 208,
  1995, Dedicated to Academician Igor Rostislavovich Shafarevich on the
  occasion of his seventieth birthday (Russian), pp.~290--317. \MR{1730271}

\bibitem{SFB}
Andrei Suslin, Eric~M. Friedlander, and Christopher~P. Bendel,
  \emph{Infinitesimal {$1$}-parameter subgroups and cohomology}, J. Amer. Math.
  Soc. \textbf{10} (1997), no.~3, 693--728. \MR{1443546}

\bibitem{Suslin-Wodz}
Andrei~A. Suslin and Mariusz Wodzicki, \emph{Excision in algebraic
  {$K$}-theory}, Ann. of Math. (2) \textbf{136} (1992), no.~1, 51--122.
  \MR{1173926}

\bibitem{TouzeUnivNew}
A.~Touz\'{e}, \emph{A construction of the universal classes for algebraic
  groups with the twisting spectral sequence}, Transform. Groups \textbf{18}
  (2013), no.~2, 539--556. \MR{3055776}

\bibitem{TouzeCRAS}
Antoine Touz\'{e}, \emph{Cohomologie du groupe lin\'{e}aire \`a coefficients
  dans les polyn\^{o}mes de matrices}, C. R. Math. Acad. Sci. Paris
  \textbf{345} (2007), no.~4, 193--198. \MR{2352918}

\bibitem{TouzeClassical}
\bysame, \emph{Cohomology of classical algebraic groups from the functorial
  viewpoint}, Adv. Math. \textbf{225} (2010), no.~1, 33--68. \MR{2669348}

\bibitem{TouzeENS}
\bysame, \emph{Troesch complexes and extensions of strict polynomial functors},
  Ann. Sci. \'{E}c. Norm. Sup\'{e}r. (4) \textbf{45} (2012), no.~1, 53--99.
  \MR{2961787}

\bibitem{TouzeBar}
\bysame, \emph{Bar complexes and extensions of classical exponential functors},
  Ann. Inst. Fourier (Grenoble) \textbf{64} (2014), no.~6, 2563--2637.
  \MR{3331175}

\bibitem{TouzeFund}
\bysame, \emph{A functorial control of integral torsion in homology}, Fund.
  Math. \textbf{237} (2017), no.~2, 135--163. \MR{3615049}

\bibitem{Touze-Survey}
\bysame, \emph{Cohomology of algebraic groups with coefficients in twisted
  representations}, Geometric and topological aspects of the representation
  theory of finite groups, Springer Proc. Math. Stat., vol. 242, Springer,
  Cham, 2018, pp.~425--463. \MR{3901171}

\bibitem{TvdK}
Antoine Touz\'{e} and Wilberd van~der Kallen, \emph{Bifunctor cohomology and
  cohomological finite generation for reductive groups}, Duke Math. J.
  \textbf{151} (2010), no.~2, 251--278. \MR{2598378}

\bibitem{Weibel}
Charles~A. Weibel, \emph{An introduction to homological algebra}, Cambridge
  Studies in Advanced Mathematics, vol.~38, Cambridge University Press,
  Cambridge, 1994. \MR{1269324}

\end{thebibliography}

\end{document}